\documentclass[final]{elsarticle}

\usepackage{amssymb}
\usepackage{latexsym}
\usepackage{amsmath}
\usepackage{amsthm}
\usepackage{amsfonts}
\usepackage{verbatim}
\usepackage{graphicx}
\usepackage{epstopdf}
\usepackage{subfigure}
\usepackage{multirow}
\usepackage{color}
\usepackage{verbatim}
\usepackage{subfigure}
\usepackage{enumerate}

\usepackage[notref, notcite]{showkeys}

\newtheorem{theorem}{Theorem}[section]
\newtheorem{remark}{Remark}[section]
\newtheorem{assump}{Assumption}[section]
\newtheorem{prop}{Proposition}[section]
\newtheorem{lemma}{Lemma}[section]
\newtheorem{ex}{Example}[section]
\newtheorem{al}{Algorithm}[section]

\DeclareMathOperator*{\esssup}{ess\,sup}

\journal{Journal of Computational and Applied Mathematics}

\begin{document}

\begin{frontmatter}


\title{An Energy Stable Linear Diffusive 
	Crank-Nicolson Scheme for the Cahn-Hilliard Gradient Flow}

\author[label2]{Lin Wang}
\ead{wanglin@csrc.ac.cn}
\fntext[label2]{School of Applied Mathematics, Guangdong University of Technology, Guangzhou, Guangdong, 510006, China}
\cortext[cor1]{corresponding author}
 	
\author[label3]{Haijun Yu}
\ead{hyu@lsec.cc.ac.cn}
\fntext[label3]{LSEC \& NCMIS, Institute of
	 	Computational Mathematics and Scientific/Engineering
	 	Computing, Academy of Mathematics and Systems Science,
	 	Beijing 100190, China; School of Mathematical Sciences,
	 	University of Chinese Academy of Sciences, Beijing
	 	100049, China} 	

%

\begin{abstract}
We propose and analyze a linearly stabilized semi-implicit diffusive Crank--Nicolson scheme for the Cahn--Hilliard gradient flow. In this scheme, the nonlinear bulk force is treated explicitly with two second-order stabilization terms. This treatment leads to linear elliptic system with constant coefficients and provable discrete energy dissipation. Rigorous error analysis is carried out for the fully discrete scheme. When the time step-size and the space step-size are small enough, second order accuracy in time is obtained with a prefactor controlled by some lower degree polynomial of $1/\varepsilon$. {Here $\varepsilon$ is the thickness of the interface}. Numerical results together with an adaptive time stepping are presented to verify the accuracy and efficiency of the proposed scheme.
\end{abstract}


\end{frontmatter}


\section{Introduction}
\label{1}

The Cahn-Hilliard equation is a widely used phase-field model. 
It was originally introduced by Cahn
and Hilliard~\cite{cahn_free_1958} to describe the complicated 
phase separation and coarsening phenomena in non-uniform systems
such as alloys, glasses and polymer mixtures. 
An important feature of the phase field model is that it
can be viewed as the gradient flow of the Liapunov energy functional
\begin{equation}\label{eq:CH:E}
E_{\varepsilon}(\phi):=\int_{\Omega}\Big(\frac{\varepsilon}{2}|\nabla \phi|^{2} 
+\frac{1}{\varepsilon} F(\phi)\Big)dx.
\end{equation}
We consider the Liapunov energy functional $E_{\varepsilon}(\phi)$
in \eqref{eq:CH:E} and the corresponding gradient flow in $H^{-1}$
to get the Cahn-Hilliard equation 
\begin{equation}\label{eq:CH}
\begin{cases}
\phi_{t}=\gamma\Delta \mu,   & (x,t)\in \Omega \times (0,T],\\
\mu=-\varepsilon \Delta \phi + \dfrac{1}{\varepsilon}f(\phi),  
&(x,t)\in \Omega \times (0,T],\\
\end{cases}
\end{equation}
subject to initial value
\begin{equation}\label{eq:CH:ini}
\phi |_{t=0} =\phi_{0}(x), x \in \Omega,
\end{equation}
and Neumann boundary condition
\begin{equation}\label{eq:CH:nbc}
\partial_{n}\phi =0, \quad \partial_{n}\mu =0, \quad
x\in \partial\Omega.
\end{equation} 
In the above, $\Omega \in R^{d}, d=1, 2,3$ is a bounded 
domain with a locally Lipschitz boundary (for the $d=2,3$ case), $n$ is the 
outward normal of $\partial \Omega$, $T$ is a given 
time, $\phi(x,t)$ is the phase-field variable.
$f(\phi)=F'(\phi)$ with $F(\phi)$ being a given
energy potential with two local minima. 
In this paper, we take the double well
potential $F(\phi)=\frac{1}{4}(\phi^2-1)^2$.
$\varepsilon$ is the thickness of the interface
between two phases. 
$\gamma$ is the mobility, which is related to 
the characteristic relaxation time of the system.
On other hand, taking the inner product of
the first equation in \eqref{eq:CH} with $\mu$
and the second equation in \eqref{eq:CH} with 
$\frac{\partial \phi}{\partial t}$,
we obtain immediately the energy dissipation law:
\begin{equation}\label{eq:CH:Edis}
\frac{\partial}{\partial t}E_{\varepsilon}(\phi) 
= - \gamma\|\nabla\mu\|^{2}
= - \gamma\|\phi_t\|_{-1}^{2},
\end{equation}
where $\|\cdot\|$ is the $L^2$ norm,
$\|\cdot\|_{-1}$ is the $H^{-1}$ norm defined in Section 2.

The Cahn-Hilliard equation is frequently used in mathematical 
models for problems in many fields of science and engineering,
particularly in materials science and fluid dynamics 
(cf. e.g. \cite{ cahn_free_1958,shen_efficient_2015, 
	anderson_diffuse-interface_2003, yue_diffuse-interface_2004, Brenner_Mathematical_2010,  
	elliott_cahnhilliard_1996, wu_stabilized_2014}).
For this reason, Cahn-Hilliard equation
has been the subject of many theoretical and
numerical investigations for several decades, see,
for instance, \cite{du_numerical_1991, elliott_error_1992, 
	chen_spectrum_1994, caffarelli_l_1995, elliott_cahnhilliard_1996,
	eyre_unconditionally_1998, furihata_stable_2001, liu_phase_2003, 
	feng_error_2004, shen_numerical_2010, condette_spectral_2011} 
and the references therein.
To obtain an energy dissipative scheme, 
the linear term is usually treated implicitly in some manners, 
while different approaches are used for nonlinear terms $F(\phi)$.
A very popular approach is the convex splitting method
which was first introduced in \cite{elliott_global_1993}, 
and popularized by \cite{eyre_unconditionally_1998}, 
in which, the convex part of $F(\phi)$ is treated implicitly
and the concave part of $F(\phi)$ is treated explicitly.
The convex splitting method was used widely,
and several second order extensions were proposed
based on either the Crank-Nicolson scheme (see e.g.\cite{
baskaran_energy_2013, wu_stabilized_2014,guo_h2_2016, diegel_stability_2016, cheng_second-order_2016, li_second-order_2017}), 
or second order backward differentiation formula (BDF2) \cite{yan_second-order_2018, li_second_2018}.

The stabilization method is another efficient algorithm
to improve the numerical stability, 
which is a special class of convex splitting method, 
see \cite{he_large_2007, shen_numerical_2010}.
The main idea is to introduce an artificial stabilization term 
to balance the explicit treatment of the nonlinear term, 
which avoids strict time step constraint.
This idea was followed up in \cite{feng_stabilized_2013} 
for the stabilized Crank-Nicolson schemes for phase field equations.
Those time marching schemes all lead to linear systems. On the other
hand, one needs to introduce a proper stabilization term and a suitably
truncated nonlinear function $\tilde{f}(\phi)$ instead of
$f(\phi)$ to prove the unconditionally energy stable
property. It is worth to mention that with no truncation made 
to $f(\phi)$, Li et al \cite{li_characterizing_2016,li_second_2017}
proved that the energy stable property can be obtained as
well, but a much larger stability constant needs be used.
The main advantage of the stabilized scheme is its 
simplicity and efficiency.

An interesting approach, named invariant energy quadratization (IEQ), is proposed in \cite{yang_linear_2016} for
dealing with phase-field equations with nonlinear Flory–
Huggins potential. The  IEQ method is a generalization of the method of Lagrange multipliers proposed in 
\cite{guillen-gonzalez_linear_2013,guillen-gonzalez_second_2014}.
It was extended to a lot of other applications, see e.g. \cite{han_numerical_2017,
	yang_efficient_2017,yang_yu_efficient_2017}. 
Recently, a scalar auxiliary variable (SAV) approach was introduced by Shen et al.\cite{shen_scalar_2017,shen_new_2017}.
SAV approach inherits all advantages of IEQ approach but also 
overcomes the shortcomings of solving variable-coefficient 
systems at each time step.

In this paper, we focus on the proof of the stability and convergence properties of energy stable linear diffusive Crank-Nicolson (SLD-CN) 
scheme for the Cahn-Hilliard Equation.
Recently, we proposed two second-order unconditionally stable linear 
schemes based on Crank-Nicolson method (SL-CN) and second-order backward differentiation formula (SL-BDF2) with stabilization for the Cahn-Hilliard equation and the Allen-Chan equation \cite{wang_two_2018, WangYu2018b, wang_energy_2018}.
 In both schemes, the nonlinear bulk forces are
treated explicitly with two additional linear stabilization terms: 
$A\tau \Delta(\phi^{n+1}-\phi^n)$ and $B(\phi^{n+1}-2\phi^n+\phi^{n-1})$.
An optimal error estimate with a prefactor depending on 
$1/\varepsilon$ only in some lower polynomial order 
is obtained for the two second-order unconditionally 
stable linear schemes for the first time, although 
some progress has been made in \cite{feng_numerical_2003,
	feng_error_2004, kessler_posteriori_2004, yang_error_2009, 
	feng_analysis_2015, feng_analysis_2016}
for the first-order stable schemes in the last dozen years.
We observe that one shortcoming
of the SL-CN scheme is that the convergence analysis requires 
the second stability constant $B > L/2\varepsilon$.
Therefore, instead of the standard Crank-Nicolson
scheme, we now use the diffusive Crank-Nicolson 
scheme, i.e., replacing $\Delta(\phi^{n+1}+\phi^{n})/2$ 
with $\Delta(3\phi^{n+1}+\phi^{n-1})/4$ to approximate
$\Delta\phi(t^{n+\frac12})$.
The proposed method enjoys all the advantages of the SL-CN scheme:
being second order accurate, time semi-discrete system is linear with constant coefficients, both finite element 
methods and spectral methods can be used for spatial discretization 
to conserve volume fraction and satisfy discrete energy dissipation law.
Furthermore, it possesses the following additional advantage:
an optimal error estimate is valid for the special
cases $A=0$ and/or $B=0$. 
We present in this paper the convergence analysis of the fully discrete 
SLD-CN scheme instead of the time semi-discrete scheme presented in the 
previous papers.
Time adaptive numerical results are carried out to 
demonstrate the reliability and robustness of this method.

The present paper is built up as follows.
Section 2 provides SLD-CN scheme for the Cahn-Hilliard equation 
and the proof of its unconditionally energy stability property. 
In Section 3, we establish the error estimate of the fully discrete numerical scheme that does not depend
on $1/\varepsilon$ exponentially .
Some 2-dimensional numerical experiments are then presented in Section 4,
showing that our proposed approaches are more robust than existing 
methods. Some concluding remarks are provided in Section 5.

\section{The stabilized linear semi-implicit Crank-Nicolson scheme}
\label{2}

	We first introduce some notations. For any given function $\phi(t)$ of $t$, we use $\phi^n$ to
	denote an approximation of $\phi(n\tau)$, where $\tau$ is
	the step-size. We will frequently use the shorthand
	notations: $\delta_{t}\phi^{n+1}:=\phi^{n+1}-\phi^{n}$,
	$\delta_{tt}\phi^{n+1}:=\phi^{n+1}-2\phi^{n}+\phi^{n-1}$,
	$\hat{\phi}^{n+\frac{1}{2}}
	:=\frac{3}{2}\phi^{n}-\frac{1}{2}\phi^{n-1}$ and $\hat{\phi}^{n+1}
	:=2\phi^{n}-\phi^{n-1}$.

We now present the stabilized linearly diffusive Crank-Nicolson scheme 
(abbr. SLD-CN) for the Cahn-Hilliard equation \eqref{eq:CH}.
Suppose $\phi^0=\phi_0(\cdot)$ and $\phi^1\approx \phi(\cdot,\tau)$ 
are given, we calculate $\phi^{n+1}, n=1,2,\ldots,N=T/\tau-1$
iteratively, using
\begin{gather}\label{newcn:eq:CN:1}
\frac{\phi^{n+1}-\phi^{n}}{\tau}=\gamma\Delta \mu^{n+\frac{1}{2}},\\
\label{newcn:eq:CN:2}
\mu^{n+\frac{1}{2}}=-\varepsilon \Delta \Big(\frac{3\phi^{n+1}+\phi^{n-1}}{4} \Big)
+\frac{1}{\varepsilon}f\Big(\frac{3}{2}\phi^{n}-\frac{1}{2}\phi^{n-1}\Big)
-A\tau \Delta \delta_{t}\phi^{n+1}
+B\delta_{tt}\phi^{n+1},
\end{gather}
where $A$ and $B$ are two non-negative constants to stabilize the scheme.

In this paper, we assume that potential function $F(\phi)$ whose 
derivative $f(\phi)$ is uniformly bounded, i.e.
\begin{equation}\label{newcn:eq:Lip}
\max_{\phi\in\mathbf{R}} | f'(\phi) | \le L,
\end{equation}
where $L$ is a non-negative constant.
\begin{remark}\label{newcn:rmk:Lip}
	Note that, Caffarelli proved that the maximum norm of the solution 
	to the Cahn-Hilliard equation is bounded for a truncated potential 
	$F$ with quadratic growth at infinities in \cite{caffarelli_l_1995}.
	On the other hand, for a more general potential $F$, Feng and Prohl \cite{feng_numerical_2005} proved that if the Cahn-Hilliard equation converges to 
	its sharp-interface limit, then its solution has a $L^\infty$ bound.
	Therefore, it has been a common practice
	(cf. \cite{kessler_posteriori_2004,shen_numerical_2010,condette_spectral_2011}) to consider the Cahn-Hilliard equations is satisfied with a truncated double-well potential $F$ such that \eqref{newcn:eq:Lip}.
\end{remark}

	For the Ginzburg-Landau double-well potential $F(\phi)=\frac{1}{4}(\phi^{2}-1)^{2}$, to get a $C^4$ smooth
	double-well potential with quadratic growth, we introduce
	$\tilde{F}(\phi) \in C^{\infty}(\mathbf{R})$ as a smooth
	mollification of
	\begin{equation}\label{newcn:efnew}
	\hat{F}(\phi)=
	\begin{cases}
	\frac{11}{2}(\phi-2)^{2}+6(\phi-2)+\frac94, &\phi>2, \\
	\frac{1}{4 }(\phi^{2}-1)^{2}, &\phi\in [-2,2], \\
	\frac{11}{2}(\phi+2)^{2}-6(\phi+2)+\frac94, &\phi<-2.
	\end{cases}
	\end{equation}
	with a mollification parameter much smaller than 1, to
	replace $F(\phi)$. Note that the truncation points $-2$ and
	$2$ used here are for convenience only. Other values outside
	of region $[-1,1]$ can be used as well.  For simplicity, we
	still denote the potential function $\tilde{F}$ by $F$.
	
	Our scheme can also be applied to the log-log Flory-Huggins energy potential by similar modification. E.g. the modified Flory-Huggins potential given in [46] satisfies our assumptions.

We introduce some notations which will be used
in the analysis. We use $\|\cdot\|_{m,p}$ to denote the
standard norm of the Sobolev space $W^{m,p}(\Omega)$. In
particular, we use $\|\cdot\|_{L^p}$ to denote the norm of
$W^{0,p}(\Omega)=L^{p}(\Omega)$; $\|\cdot\|_{{m}}$ to denote
the norm of $W^{m,2}(\Omega)=H^{m}(\Omega)$; and $\|\cdot\|$
to denote the norm of $W^{0,2}(\Omega)=L^{2}(\Omega)$.  Let
$(\cdot, \cdot)$ represent the $L^{2}$ inner product.  In
addition, define for $p\geq 0$
\begin{equation}\notag
H^{-p}(\Omega):=\left(H^{p}(\Omega)\right)^{*},\quad
H_{0}^{-p}(\Omega):=\left\{ u \in H^{-p}(\Omega),
\ \langle u,1 \rangle_{p}=0 \right\},
\end{equation}
where $\langle \cdot,\cdot \rangle_{p}$ stands for the dual product
between $H^{p}(\Omega)$ and $H^{-p}(\Omega)$. We denote
$L_{0}^{2}(\Omega):= H_{0}^{0}(\Omega)$. For
$v \in L_{0}^{2}(\Omega)$, let
$-\Delta^{-1}v:=v_{1} \in H^{1}(\Omega)\cap
L_{0}^{2}(\Omega)$, where $v_{1}$ is the solution to
\begin{equation}\notag
-\Delta v_{1}=v \ \ {\rm in}\  \Omega ,\quad 
\ \frac{\partial v_{1}}{\partial n}=0 \ \ {\rm on}\  \partial \Omega,
\end{equation}
and $\|v\|_{-1}:=\sqrt{(v,-\Delta^{-1}v) }$.

Following identities and inequality will 
be used frequently.
\begin{equation}\label{newcn:eq:ID:1}
2(h^{n+1}-h^n, h^{n+1}) = \|h^{n+1}\|^2 - \|h^n\|^2 + \|h^{n+1}-h^n\|^2,
\end{equation}
\begin{equation}\label{newcn:eq:ID:3}
(u, v) \leq \|u\|_{-1}  \| \nabla v\|, \quad \forall\   u\in L_0^2, v\in H^1.
\end{equation}

\begin{theorem}\label{newcn:cn}
	Under the condition
	\begin{equation}\label{newcn:eq:CN:ABcond}
	A\geq \dfrac{L^{2}}{16\varepsilon^{2}}\gamma, \quad  
	B\geq \dfrac{L}{2\varepsilon},
	\end{equation}
	the following energy dissipation law
	\begin{equation}\label{newcn:eq:CN:Edis}
	\begin{split}
	E_{C}^{n+1}\leq& E_{C}^{n}-\Big(2\sqrt{\frac{A}{\gamma}}-\frac{L}{2\varepsilon}\Big)
	\|\delta_{t}\phi^{n+1}\|^{2}
	-\Big(\frac{B}{2}-\frac{L}{4\varepsilon}\Big)\|\delta_{tt}\phi^{n+1}\|^{2}\\
	&-\frac{\varepsilon}{8}\|\nabla\delta_{tt} \phi^{n+1}\|^{2},\hspace{1cm} \forall n\geq1,
	\end{split}
	\end{equation}
	holds for the scheme (\ref{newcn:eq:CN:1})-(\ref{newcn:eq:CN:2}), where
	\begin{equation}\label{newcn:eq:CN:E}
	E_{C}^{n+1}=E_{\varepsilon}(\phi^{n+1})
	+\Big(\frac{L}{4\varepsilon}+\frac{B}{2}\Big)\|\delta_{t}\phi^{n+1}\|^{2}
	+\frac{\varepsilon}{8}\|\nabla\delta_t \phi^{n+1}\|^{2}.
	\end{equation}
\end{theorem}

\begin{proof}
	Pairing \eqref{newcn:eq:CN:1} with
	$\tau \mu^{n+\frac{1}{2}}$, \eqref{newcn:eq:CN:2} with
	$-\delta_t\phi^{n+1}$, and combining the results, we get
	\begin{equation}\label{newcn:cn5}
	\begin{split}
	&\frac{\varepsilon}{2}(\|\nabla \phi^{n+1}\|^{2} -
	\|\nabla \phi^{n}\|^{2})
	+\frac{\varepsilon}{8}(\|\nabla\delta_t \phi^{n+1}\|^{2} -
	\|\nabla\delta_t \phi^{n}\|^{2})
	+\frac{1}{\varepsilon}(f\big(\hat{\phi}^{n+\frac{1}{2}}\big),\delta_t\phi^{n+1} )\\
	=&-\gamma\tau \|\nabla \mu^{n+\frac{1}{2}}\|^{2} 
	-A\tau \|\nabla\delta_{t}\phi^{n+1}\|^{2}
	-\frac{\varepsilon}{8}\|\nabla\delta_{tt} \phi^{n+1}\|^{2}
	-B(\delta_{tt}\phi^{n+1},\delta_t\phi^{n+1}).
	\end{split}
	\end{equation}
	Pairing (\ref{newcn:eq:CN:1}) with $2\sqrt{A/\gamma}\tau\delta_t\phi^{n+1}$, 
	then using Cauchy-Schwarz inequality, we get
	\begin{equation}\label{newcn:cn7}
	\begin{split}
	2\sqrt{\tfrac{A}{\gamma}}\|\delta_{t} \phi^{n+1}\|^{2}
	= &-2\sqrt{A\gamma}\tau(\nabla \mu^{n+\frac12},\nabla\delta_t 
	\phi^{n+1})\\\leq& \gamma\tau \|\nabla \mu^{n+\frac12}\|^{2} +
	A\tau\|\nabla \delta_{t}\phi^{n+1}\|^{2}.
	\end{split}
	\end{equation}
	To handle the term involving $f$, we expand $F(\phi^{n+1})$ 
	and $F(\phi^n)$ at
	$\hat{\phi}^{n+\frac{1}{2}}$ as
	\begin{align*}
	F(\phi^{n+1}) &= F(\hat{\phi}^{n+\frac{1}{2}})+f(\hat{\phi}^{n+\frac{1}{2}})(\phi^{n+1}-\hat{\phi}^{n+\frac{1}{2}})+\frac{1}{2}f'(\xi^{n}_{1})(\phi^{n+1}-\hat{\phi}^{n+\frac{1}{2}})^{2},\\
	F(\phi^{n}) &= F(\hat{\phi}^{n+\frac{1}{2}})+f(\hat{\phi}^{n+\frac{1}{2}})(\phi^{n}-\hat{\phi}^{n+\frac{1}{2}})+\frac{1}{2}f'(\xi^{n}_{2})(\phi^{n}-\hat{\phi}^{n+\frac{1}{2}})^{2},
	\end{align*}
	where $\xi^{n}_1$ is a number between $\phi^{n+1}$ and
	$\hat{\phi}^{n+\frac12}$, $\xi^{n}_2$ is a number
	between $\phi^n$ and $\hat{\phi}^{n+\frac12}$.  Taking
	the difference of above two equations, we have
	\begin{equation}\nonumber
	\begin{split}
	&  F(\phi^{n+1})-F(\phi^{n}) - f(\hat{\phi}^{n+\frac{1}{2}})(\phi^{n+1}-\phi^{n})\\
	={} &  \frac{1}{2}f'(\xi^{n}_{1})
	\left[(\phi^{n+1}-\hat{\phi}^{n+\frac{1}{2}})^{2} - (\phi^{n}-\hat{\phi}^{n+\frac{1}{2}})^{2}
	\right]
	- \frac{1}{2}(f'(\xi^{n}_{2})-f'(\xi^{n}_{1}))(\phi^{n}
	-\hat{\phi}^{n+\frac{1}{2}})^{2}\\
	={} &  
	\frac{1}{2}f'(\xi^{n}_{1})\delta_{t}\phi^{n+1}\delta_{tt}\phi^{n+1}
	- \frac{1}{8}(f'(\xi^{n}_{2})-f'(\xi^{n}_{1}))(\delta_{t}\phi^{n})^{2}\\
	\le{} & 
	\frac{L}{4}(|\delta_t \phi^{n+1}|^2 + |\delta_{tt}\phi^{n+1}|^2)
	+ \frac{L}{4}|\delta_t\phi^n|^2.
	\end{split}
	\end{equation}
	Multiplying the above equation with $1/\varepsilon$, then taking integration leads to
	\begin{equation}\label{newcn:eq:cn2}
	\begin{split}
	&\frac{1}{\varepsilon}(F(\phi^{n+1})-F(\phi^{n})-
	f(\hat{\phi}^{n+\frac{1}{2}})\delta_t\phi^{n+1}, 1) \\
	\leq&
	\frac{L}{4\varepsilon}(\|\delta_t \phi^{n+1}\|^2 + \|\delta_{tt}\phi^{n+1}\|^2
	+ \|\delta_t\phi^n\|^2).
	\end{split}
	\end{equation}
	For the term involving $B$, by using identity\eqref{newcn:eq:ID:1} with $h^{n+1}=\delta_t \phi^{n+1}$, one gets
	\begin{equation}\label{newcn:cs11-0}
	-B(\delta_{tt}\phi^{n+1},\delta_t\phi^{n+1})
	=-\frac{B}{2}\|\delta_{t}\phi^{n+1}\|^{2}+\frac{B}{2}
	\|\delta_{t}\phi^{n}\|^{2}-\frac{B}{2}\|\delta_{tt}\phi^{n+1}\|^{2}.
	\end{equation}
	Summing up \eqref{newcn:cn5}-\eqref{newcn:cs11-0}, we obtain
	\begin{equation}\label{newcn:cn8}
	\begin{split}
	&\frac{\varepsilon}{2}(\|\nabla \phi^{n+1}\|^{2} -
	\|\nabla \phi^{n}\|^{2}) 
	+ \frac{1}{\varepsilon}(F(\phi^{n+1})-F(\phi^{n}),1)
	+
	\frac{B}{2}(\|\delta_{t}\phi^{n+1}\|^{2} -\|\delta_{t}\phi^{n}\|^{2})\\
	&+\frac{\varepsilon}{8}(\|\nabla\delta_t \phi^{n+1}\|^{2} -
	\|\nabla\delta_t \phi^{n}\|^{2})\\
	\leq & -
	2\sqrt{\frac{A}{\gamma}}\|\delta_{t}\phi^{n+1}\|^{2}
	+\frac{L}{4\varepsilon}\|\delta_{t} \phi^{n+1}\|^{2}
	+\frac{L}{4\varepsilon}\|\delta_{t} \phi^{n}\|^{2}
	-\frac{B}{2}\|\delta_{tt}\phi^{n+1}\|^{2}\\
	&+\frac{L}{4\varepsilon}\|\delta_{tt} \phi^{n+1}\|^{2}
	-\frac{\varepsilon}{8}\|\nabla\delta_{tt} \phi^{n+1}\|^{2},
	\end{split}
	\end{equation}
	which is the energy estimate \eqref{newcn:eq:CN:Edis}.
\end{proof}

\begin{remark}\label{newcn:rmk:stab3}
	The discrete Energy $E_C$ defined in equation \eqref{newcn:eq:CN:E} 
	is a second order approximation to the original energy $E_{\varepsilon}$, since $ \| \delta_t \phi^{n+1} \|^2 \sim O(\tau^2)$. On the
	other side, summing up the equation \eqref{newcn:eq:CN:Edis} 
	for $n=1,\ldots, N$, we get
	\begin{equation}\label{newcn:eq:steady}
	\begin{split}
	&E^{N+1}_C + \sum_{n=1}^N \left( 
	\Big(2\sqrt{\frac{A}{\gamma}}-\frac{L}{2\varepsilon}\Big)
	\|\delta_{t}\phi^{n+1}\|^{2}
	+\Big(\frac{B}{2}-\frac{L}{4\varepsilon}\Big)
	\|\delta_{tt}\phi^{n+1}\|^{2}+\frac{\varepsilon}{8}
	\|\nabla\delta_{tt} \phi^{n+1}\|^{2}\right)\\
	 \leq& E^1_C.
	\end{split}
	\end{equation}
	Under the condition $\eqref{newcn:eq:CN:ABcond}$,
	$\big(2\sqrt{\frac{A}{\gamma}}-\frac{L}{2\varepsilon}\big)$ and 
	$\big(\frac{B}{2}-\frac{L}{4\varepsilon}\big)$ are positive constants.
	So, for given $\tau$, by taking $N\rightarrow\infty$, we get 
	$\|\delta_t \phi^{N+1}\| \rightarrow 0$. On the other hand,  if we leave a small part of $A$ term in its original form in the proof, denoted by $\delta A$, we will have an diffusion term $\delta A \sum_{n=1}^{N}\tau \|\nabla \delta_t \phi^{n+1}\|^2$, we obtain $\| \nabla \delta_t \phi^{N+1}\|^2 \rightarrow 0$ as well, which means the discrete Energy converge to the original Energy: $E_C^{N+1}\rightarrow E_{\varepsilon}(\phi^{N+1})$
	and the system eventually will converge to a steady state
	for long time run.
\end{remark}

\section{Error estimate}
\label{3}
We use a Legendre Galerkin method similar as in
\cite{shen_efficient_1995, shen_efficient_2015, 
	yang_yu_efficient_2017} for spatial
discretization in 2-dimensional domain.  
Let $L_k(x)$ denote the Legendre polynomial of degree $k$. 
We define
\[
V_M = \mbox{span}\{\,\varphi_k(x)\varphi_j(y),\ k,j=0,\ldots, M-1\,\}\in H^1(\Omega),
\]
where
$\varphi_0(x) = L_0(x); \varphi_1(x)=L_1(x);  \varphi_k(x)=L_k(x)-L_{k+2}(x), k=2,\ldots, M\!-\!1
$,
be the Galerkin approximation space for both $\phi^{n+1}_h$ and 
$\mu^{n+1}_h$. Then the full discretized form for the SLD-CN scheme reads:
Find $(\phi^{n+1}_h, \mu^{n+\frac12}_h) \in (V_M)^{2}$ such that
\begin{equation}\label{eq:CN:full:1}
\frac{1}{\tau}(\phi_h^{n+1}-\phi_h^n, \psi_h)=-\gamma( \nabla \mu_h^{n+\frac{1}{2}}, \nabla \psi_h),   \qquad \forall \psi_h \in V_M,
\end{equation}
\begin{equation}\label{eq:CN:full:2}
\begin{split}
(\mu_h^{n+\frac{1}{2}}, \varphi_h)=&
\varepsilon \left(\nabla \frac{3\phi_h^{n+1}+\phi_h^{n-1}}{4}, 
\nabla \varphi_h\right)
+\frac{1}{\varepsilon}\left(f\Big(\frac{3}{2}\phi_h^{n}-
\frac{1}{2}\phi_h^{n-1}\Big),\varphi_h \right)\\
&+A\tau (\nabla \delta_{t}\phi_h^{n+1}, \nabla \varphi_h )
+B(\delta_{tt}\phi_h^{n+1},\varphi_h ), 
\qquad\  \forall \varphi_h \in V_M.
\end{split}
\end{equation}

In this section, we shall establish the error estimate of
the full discretized form \eqref{eq:CN:full:1}-\eqref{eq:CN:full:2}
for SLD-CN scheme. We will show that, if the interface is well
developed in the initial condition, the error bounds depend
on $1/\varepsilon$ only in some lower polynomial order for small $\varepsilon$.  Let $\phi(t^{n})$ be the exact solution at time 
$t=t^n$ to equation of \eqref{eq:CH}, which is abbreviated as $\phi^n$. 
Let $\phi_h^{n}$ be the solution at time $t=t^{n}$ to the full discrete numerical scheme (\ref{newcn:eq:CN:1})-(\ref{newcn:eq:CN:2}), 
we define error function $e^{n}:=\phi_h^{n}-\phi^n$.

We introduce the Ritz projection operator 
$R_h:H^1(\Omega)\rightarrow V_M$ satisfying 
\begin{equation}\label{eq:Rh1}
(\nabla(R_{h}\varphi-\varphi),\nabla\psi_h)=0,
\ \forall \psi_h \in V_{M}, \ \ \ \ (R_{h}\varphi-\varphi,1)=0.
\end{equation}

The following estimates hold for the Ritz projection \cite{Brenner_Mathematical_2010}: 
\begin{equation}\label{eq:Rh2}
\|R_{h}\varphi\|_{1,p}\leq C\|\varphi\|_{1,p}, 
\ \  \forall 1<p\leq\infty,
\end{equation}
\begin{equation}\label{eq:Rh3}
\|R_{h}\varphi -\varphi\|_{L^p}
+h\|R_{h}\varphi -\varphi\|_{1,p}\leq C h^{q+1}
\|\varphi\|_{q+1,p},\ \ \forall1<p\leq\infty. 
\end{equation}
\begin{equation}\label{eq:Rh4}
\|R_{h}\varphi -\varphi\|
+h^{-1}\|R_{h}\varphi -\varphi\|_{-1}\leq C h^{q+1}
\|\varphi\|_{H^{q+1}}.
\end{equation}

Define $\rho^{n+1}:=R_{h}\phi^{n+1} -\phi^{n+1}$ and $\sigma_{h}^{n+1}:=\phi_{h}^{n+1} -R_{h}\phi^{n+1}$, then $e^{n+1}=\rho^{n+1}+\sigma_{h}^{n+1}$, $\sigma_{h}^{0}\equiv0$. 
By the Ritz projection, $(\nabla\rho^{n+1},\nabla \psi_h)=0$, 
for all $\psi_h \in V_M$.
The proofs base on Galerkin formulation. 
Spectral element method can be used for spatial discretization to
satisfy the estimates for the Ritz projection and error estimate.

Before presenting the detailed error analysis, we first make
some assumptions. For simplicity, we take $\gamma=1$ in this
section, and assume $0<\varepsilon<1$.  We use notation
$\lesssim$ in the way that $f\lesssim g$ means that
$f \le C g$ with positive constant $C$ independent of $\tau$
and $\varepsilon$.

\begin{assump}\label{newcn:ap:1}
	We assume that $f$ either satisfies the following 
	properties (i) and (ii), or (i) and (iii).
	\begin{enumerate}
		\item [(i)]$F\in C^{4}(\mathbf{R})$, 
		$F(\pm 1)=0$, and $F>0$ elsewhere. There exist two
		non-negative constants $B_0,B_1$, such that
		\begin{equation}\label{eq:AP:Fcoercive} \phi^2 \le
		B_0 + B_1 F(\phi),\quad \forall\; \phi\in\mathbf{R}.
		\end{equation}
		\item[(ii)] $f=F'$. $f'$ and $f''$ are uniformly bounded, i.e. 
		$f$ satisfies \eqref{newcn:eq:Lip} and
		\begin{equation}\label{newcn:eq:Lip2}
		\max_{\phi\in\mathbf{R}} | f''(\phi) | \le L_2,
		\end{equation}
		where $L_2$ is a non-negative constant.
		\item[(iii)] $f$  satisfies for some finite $2\le p \leq 
		3+ \frac{d}{3(d-2)}$
		and positive numbers $\tilde{c}_{i} >0$, $i=0,\ldots,5$,
		\begin{equation}\label{eq:AP:fp}
		\tilde{c}_1 |\phi|^{p-2} - \tilde{c}_0 \leq f'(\phi) \leq \tilde{c}_2 |\phi|^{p-2}+\tilde{c}_3,
		\end{equation}
		\begin{equation}\label{eq:AP:fpp}
		| f''(\phi) | \leq \tilde{c}_4 |\phi|^{(p-3)^+}+\tilde{c}_5,
		\end{equation}
		where for any real number $a$, the notation $(a)^+ := \max\{ a, 0\}$.\qed
	\end{enumerate}
\end{assump}

Note that Assumption \ref{newcn:ap:1} (ii) is a special case of 
Assumption \ref{newcn:ap:1} (iii) with $p=2$. The commonly-used 
quartic double-well potential satisfies Assumption (i) and 
(iii) with $p=4$. Furthermore, from equation \eqref{eq:AP:fp} 
we easily get 
\begin{equation} \label{eq:AP:fpl}
-(f'(\phi) u, u) \le \tilde{c}_0 \| u \|^2, \quad \forall\, 
u\in L^2(\Omega).
\end{equation} 

\begin{assump}\label{newcn:ap:2}
	We assume that $\phi^0$ is smooth enough. More precisely,
	there exist constant $m_{0}$ and non-negative constants 
	$\sigma_{1}, \ldots, \sigma_6$, such that
	\begin{equation}\label{eq:AP:m0}
	m_{0}:=\frac{1}{|\Omega|}\int_{\Omega}\phi^{0}(x){\rm d}x \in (-1,1),
	\end{equation}	
	\begin{equation}\label{eq:AP:E0}
	E_{\varepsilon}(\phi^{0}):=\frac{\varepsilon}{2}\|\nabla \phi^{0}\|^{2}+\frac{1}{\varepsilon}\|F(\phi^{0})\|_{L^{1}}\lesssim \varepsilon^{-\sigma_{1}},
	\end{equation}
	\begin{equation}\label{eq:AP:2}
	\|\phi_t^0\|_{-1}^2\lesssim \varepsilon^{-\sigma_2},
	\end{equation}
	\begin{equation}\label{eq:AP:3}
	\|\phi_{t}^0\|^2\lesssim \varepsilon^{-\sigma_3};
	\end{equation}
	\begin{equation}\label{eq:AP:4}
	\varepsilon\|\nabla \phi_{t}^0\|^2 
	+ \frac{1}{\varepsilon} (f'(\phi^0)\phi_t^0,\phi_t^0) 
	\lesssim \varepsilon^{-\sigma_{4}},
	\end{equation}
	\begin{equation}\label{eq:AP:5}
	\|\Delta^{-1}\phi_{tt}^0\|^2
	\lesssim \varepsilon^{-\sigma_{5}},
	\end{equation}
	\begin{equation}\label{eq:AP:6}
	\|\phi_{tt}^0\|_{-1}^2
	\lesssim \varepsilon^{-\sigma_{6}},
	\end{equation}
	\begin{equation}\label{eq:AP:7}
	\|\phi_{tt}^0\|^2
	\lesssim \varepsilon^{-\sigma_{7}}.
	\end{equation}
\end{assump}

Given Assumption \ref{newcn:ap:1} (i)(iii) and Assumption 
\ref{newcn:ap:2}, we have following estimates for the exact
solution to the Cahn-Hilliard equation. 

\begin{assump}\label{lm:1}
	Suppose the exact solution of (\ref{eq:CH}) has the following regularities:\\
	\begin{enumerate}
		\item[(1)] $\Delta^{-1} \phi \in W^{2,2}(0,T;H^{-1})$, or
		$$\int_{0}^{T}\|\Delta^{-1}\phi_{tt}\|_{-1}^{2}{\rm d}t \leq \varepsilon^{-\rho_{1}},$$
		\item[(2)] $\phi\in W^{2,2}(0,T;H^{-1}\bigcap H^{1})$, or
		\begin{align*}
		\int_{0}^{T}\|\phi_{tt}\|_{-1}^{2}{\rm d}t &\leq \varepsilon^{-\rho_{2}},\quad&
		\int_{0}^{T}\| \nabla \phi_{tt}\|^{2}{\rm d}t &\leq \varepsilon^{-\rho_{3}},\quad&
		\int_{0}^{T}\|\phi_{tt} \|^2_{H^{q+1}}{\rm d}t  &\leq \varepsilon^{-\rho_{4}},
		\end{align*}
		\item[(3)] $\phi \in W^{1,2}(0,T;H^{1})$, or
		$$\int_{0}^{T}\|\nabla \phi_{t}\|^{2}{\rm d}t \leq \varepsilon^{-\rho_{5}},
		\quad  \int_{0}^{T}\| \phi_t\|^2_{H^{q+1}}{\rm d}t\leq \varepsilon^{-\rho_{6}},$$
		\item[(4)] $$\tau\sum_{n=1}^{N+1} \| \phi^{n}\|^2_{H^{q+1}} \leq \varepsilon^{-\rho_{7}}, \quad 
		\tau\sum_{n=1}^{N+1} \| \mu^{n}\|^2_{H^{q+1}} \leq \varepsilon^{-\rho_{8}},$$
		\item[(5)] $$ \max_{1\leq n\leq N+1}\|\phi^n\|^2_{H^{q+1}} \leq \varepsilon^{-\rho_9}.$$
	\end{enumerate}
	Here $\rho_1=\beta_8$, $\rho_2=\beta_4$, $\rho_3=\beta_6$,
	$\rho_4=\beta_{11}$, $\rho_5=\beta_2+1$, $\rho_6=\beta_{10}$,
	$\rho_7=\sigma_1+3$, $\rho_8=\beta_{12}$, $\rho_9=\sigma_1+3$, where 
	$\beta_{j}, j=1\cdots12$ are non-negative constants which can be control by $\sigma_{1}, \sigma_{2}, \sigma_{3}.$
\end{assump}
An estimate for $\rho_1, \ldots, \rho_9$,  $q=1$ is given in Appendix.

To get the convergence result of the second order schemes, 
we need make some assumptions on the scheme used to calculate 
the numerical solution at first time step. 

\begin{assump}\label{ap:3}
	We assume that an appropriate scheme is used to
	calculate the numerical solution at first step, such that
	\begin{equation}\label{eq:AP:m1}
	m_1:=\frac{1}{|\Omega|}\int_{\Omega}\phi^{1}_h(x){\rm d}x = m_0,
	\end{equation}
	\begin{equation}\label{eq:AP:phi1}
	E_\varepsilon (\phi^1_h) \le E_\varepsilon (\phi^0_h) \lesssim\varepsilon^{-\sigma_1 },
	\end{equation}
	\begin{equation} \label{eq:AP:phi1Hn1} 
	\frac{1}{\tau}\| \phi^1_h
	- \phi^0_h \|^2_{-1}  \lesssim
	\varepsilon^{-\sigma_1 },
	\end{equation}
	\begin{equation} \label{eq:AP:phi1L2} 
	\frac{1}{\tau}\|\phi^1_h-\phi^0_h\|^2 \lesssim\varepsilon^{-\sigma_1 -2},
	\end{equation}
	and there exist a constant $0<\tilde{\sigma}_1<\rho_5+5$ and
	$0<\tilde{\sigma}_2<\max\{\rho_6 +1,\rho_7+3,\rho_8+1\}$ such that
	\begin{equation}\label{eq:AP:phi1e}
	\| e^1 \|^2_{-1} + A\tau^2\| \nabla e^1 \|^2 \lesssim
	\varepsilon^{-\tilde{\sigma}_1}(\tau^4+h^{2q+4} ),
	\end{equation}
	\begin{equation}\label{eq:AP:phisigma}
	\| \sigma_h^1 \|^2_{-1} + A\tau^2\| \nabla \sigma_h^1 \|^2
	\lesssim \varepsilon^{-\tilde{\sigma}_2}(\tau^4+h^{2q+4} ).
	\end{equation}
\end{assump}

According to the volume conservation property, we easily get 
the following properties. Because the integration of $\phi^n_h$ is conserved, $\delta_t \phi^n_h$ and $e^n$ belong to $L_0^2(\Omega)$ 
such that we can define $H^{-1}$ norm and
use Poincare's inequality for those quantities.

\begin{lemma}\label{lm:stab}
	Suppose \eqref{eq:AP:m0} and \eqref{eq:AP:m1} holds, then
	the numerical solution of \eqref{newcn:eq:CN:1}-\eqref{newcn:eq:CN:2}
	satisfies
	\begin{equation}\label{eq:CA:conserv}
	\frac{1}{|\Omega|}\int_{\Omega}\phi^{n}_h{\rm d}x=m_0,
	\quad n=1,\ldots, N+1, 
	\end{equation}
	and the error function $e^n$ satisfies
	\begin{equation}\label{eq:CA:e0}
	\int_{\Omega}e^{n}(x){\rm d}x=0,\quad n=1,\ldots, N+1. 
	\end{equation}\qed
\end{lemma}

We first carry out a coarse error estimate, which uses
standard approach for the full discretized schemes \eqref{eq:CN:full:1}-\eqref{eq:CN:full:2}.
\begin{prop}(Coarse error estimate)\label{newcn:prio}
	Suppose that $A$ and $ B$ are any non-negative number, $\tau \lesssim \varepsilon^3$. Then for all $N\geq1$, we have estimate
	\begin{equation}\label{newcn:cet3-1}
	\begin{split}
	&\|\sigma_h^{n+1}\|_{-1}^{2}
	+\frac{1}{4}\|\delta_t \sigma_h^{n+1}\|_{-1}^{2}
	+A\tau^2\|\nabla \sigma_h^{n+1}\|^{2}
	+\frac{A\tau^2}{2}\|\nabla \delta_t\sigma_h^{n+1}\|^{2}\\
	&+\frac{A\tau^2}{4}\|\nabla \delta_{tt}\sigma_h^{n+1}\|^{2}
	+\varepsilon \tau \|\nabla \frac{3\sigma_h^{n+1}+\sigma_h^{n-1}}{4}\|^2\\
	\leq
	& \|\sigma_h^{n}\|_{-1}^{2}
	+\frac{1}{4}\|\delta_t \sigma_h^{n}\|_{-1}^{2}
	+A\tau^2\|\nabla \sigma_h^{n}\|^{2}
	+\frac{A\tau^2}{4}\|\nabla \delta_t \sigma_h^{n}\|^{2}\\
	&+\frac{99 L^2}{2\varepsilon^3}  \tau\|\sigma_h^{n}\|^2_{-1}
	+\frac{11 L^2}{2\varepsilon^3} \tau \|\sigma_h^{n-1}\|^2_{-1}
	+\gamma_1(\varepsilon) \tau^4 + \gamma_2(\varepsilon, \tau) h^{2q+4},
	\end{split}
	\end{equation}
	
	\begin{equation}\label{newcn:cet3-2}
	\begin{split}
	&\max_{1\leq n \leq N} \left(\|\sigma_h^{n+1}\|_{-1}^{2}
	+A\tau^2\|\nabla \sigma_h^{n+1}\|^{2} 
	+\frac{1}{4}\|\delta_t \sigma_h^{n+1}\|_{-1}^{2} 
	+\frac{A\tau^2}{2}\|\nabla \delta_t\sigma_h^{n+1}\|^{2}\right)\\
	&+\frac{A\tau^2}{4}\|\nabla \delta_{tt}\sigma_h^{n+1}\|^{2}
	+\varepsilon \tau \|\nabla \frac{3\sigma_h^{n+1}+\sigma_h^{n-1}}{4}\|^2\\
	\lesssim & \exp\left(\frac{55L^2T}{\varepsilon^3}\right)
	(\gamma_1(\varepsilon) \tau^4 + \gamma_2(\varepsilon, \tau ) h^{2q+4}),
	\end{split}
	\end{equation} 
	where
	$ \gamma_1(\varepsilon):=\varepsilon^{-\max\{\rho_{1}+1, \rho_{2}+3, \rho_{3}-1, \rho_{5}+5\}},$\\
	$\gamma_2(\varepsilon,\tau):= \min\{\varepsilon^{-\max\{\rho_6+1,\rho_{7}+3, \rho_{8}+1\}}, 
	\varepsilon^{-(\rho_4+3)}\tau^4\}.$\\
\end{prop}

\begin{proof}
	Here, we can write the error function equations: 
	\begin{equation}\label{newcn:cet}
	\begin{split}
	\left(\frac{e^{n+1}- e^{n}}{\tau}, \psi_h　\right)
	=&-(\nabla(\mu_h^{n+\frac{1}{2}}-\mu^{n+\frac{1}{2}}),  \nabla \psi_h　)\\
	&+\left(\phi^{n+\frac{1}{2}}_{t}-\frac{\phi^{n+1}-\phi^n}{\tau}, \psi_h \right),  \forall \psi_h \in S_h,
	\end{split}
	\end{equation}
	\begin{equation}\label{newcn:cet1}
	\begin{split}
	(\mu^{n+\frac{1}{2}}_h-\mu^{n+\frac{1}{2}}, \varphi_h)=
	&\varepsilon\left(\nabla \frac{3e^{n+1}+e^{n-1}}{4}, \nabla \varphi_h \right)\\
	&+\varepsilon\left(\nabla\left( \frac{3\phi^{n+1}+\phi^{n-1}}{4}-\phi^{n+\frac{1}{2}}\right) , \nabla \varphi_h \right) \\
	&+\frac{1}{\varepsilon}\Big{(}f(\frac{3}{2}\phi_h^{n}-\frac{1}{2}\phi_h^{n-1})-f(\phi^{n+\frac{1}{2}}), \varphi_h
	\Big{)}\\
	&　+A\tau(\nabla \delta_{t}\phi_h^{n+1}, \nabla \varphi_h)
	+B(\delta_{tt}\phi_h^{n+1}, \varphi_h), \qquad \forall \varphi_h \in S_h.
	\end{split}
	\end{equation}	
	By using $\mu_h^{n+\frac{1}{2}}-\mu^{n+\frac{1}{2}}=\mu_h^{n+\frac{1}{2}}
	-R_h\mu^{n+\frac{1}{2}}+ R_h\mu^{n+\frac{1}{2}}
	-\mu^{n+\frac{1}{2}}$ and \eqref{eq:Rh1}, we get 
	\begin{equation}\label{newcn:cet1-1}
	\begin{split}
	-(\nabla(\mu_h^{n+\frac{1}{2}}-\mu^{n+\frac{1}{2}}),  \nabla \psi_h　)
	=&-(\nabla(\mu_h^{n+\frac{1}{2}}-R_h\mu^{n+\frac{1}{2}}+ R_h\mu^{n+\frac{1}{2}}
	-\mu^{n+\frac{1}{2}}),  \nabla \psi_h　)\\
	=&-(\nabla(\mu_h^{n+\frac{1}{2}}-R_h\mu^{n+\frac{1}{2}}),  \nabla \psi_h　)\\
	=&(\mu_h^{n+\frac{1}{2}}-R_h\mu^{n+\frac{1}{2}},  \Delta \psi_h　)\\
	=&(\mu_h^{n+\frac{1}{2}}-\mu^{n+\frac{1}{2}},  \Delta \psi_h　) +(\mu^{n+\frac{1}{2}}-R_h\mu^{n+\frac{1}{2}},  \Delta \psi_h　)
	\end{split}
	\end{equation}
	
	Combining \eqref{newcn:cet}-\eqref{newcn:cet1-1}, taking $\psi_h=-\Delta^{-1}\big{(} \frac{3\sigma_h^{n+1}+\sigma_h^{n-1}}{4}\big{)}$ and 
	$\varphi_h=-\big{(} \frac{3\sigma_h^{n+1}+\sigma_h^{n-1}}{4}\big{)}$, 
	and using $e^{n+1}=\rho^{n+1}+\sigma_{h}^{n+1}$, we get
	\begin{equation}\label{newcn:cet1-3}
	\begin{split}
	&-\left(\frac{\sigma_h^{n+1}- \sigma_h^{n}}{\tau}, \Delta^{-1}\left( \frac{3\sigma_h^{n+1}+\sigma_h^{n-1}}{4}\right)\right)
	+A\tau\left(\nabla \delta_{t}\sigma_h^{n+1}, \nabla \frac{3\sigma_h^{n+1}+\sigma_h^{n-1}}{4}  \right)\\
	&+\varepsilon \|\nabla \frac{3\sigma_h^{n+1}+\sigma_h^{n-1}}{4}\|^2\\
	= &
	-\varepsilon\left(\nabla \frac{3\rho^{n+1}+\rho^{n-1}}{4}, \nabla \frac{3\sigma_h^{n+1}+\sigma_h^{n-1}}{4}  \right)
	-B\left( \delta_{tt}\sigma_h^{n+1}, \frac{3\sigma_h^{n+1}+\sigma_h^{n-1}}{4}  \right)\\
	&-\frac{1}{\varepsilon}\left(f(\frac{3}{2}\phi_h^{n}-\frac{1}{2}\phi_h^{n-1})-f(\phi^{n+\frac{1}{2}}), \frac{3\sigma_h^{n+1}+\sigma_h^{n-1}}{4}  \right)\\
	&　
	-A\tau\left(\nabla \delta_{t}\rho^{n+1}, \nabla \frac{3\sigma_h^{n+1}+\sigma_h^{n-1}}{4}  \right)   
	-B\left(\delta_{tt}\rho^{n+1},  \frac{3\sigma_h^{n+1}+\sigma_h^{n-1}}{4}  \right)\\
	&+\left(\frac{\rho^{n+1}- \rho^{n}}{\tau}, \Delta^{-1}\left( \frac{3\sigma_h^{n+1}+\sigma_h^{n-1}}{4}\right)\right)
	-\left(\mu^{n+\frac{1}{2}}-R_h\mu^{n+\frac{1}{2}},  \frac{3\sigma_h^{n+1}+\sigma_h^{n-1}}{4} \right)\\
	&-\left(R_1^{n+1}, \Delta^{-1}\left( \frac{3\sigma_h^{n+1}+\sigma_h^{n-1}}{4}\right)\right)
	-A\left(\nabla R_2^{n+1}, \nabla \frac{3\sigma_h^{n+1}+\sigma_h^{n-1}}{4}  \right)\\
	&-B\left(R_3^{n+1},  \frac{3\sigma_h^{n+1}+\sigma_h^{n-1}}{4}  \right)
	-\varepsilon\left(\nabla R_4^{n+1}, \nabla \frac{3\sigma_h^{n+1}+\sigma_h^{n-1}}{4}  \right)\\
	=&:J_1 + J_2 + J'_3 +J_5 + J_6+ J_7 + J_8 +J_9+J_{10}+J_{11}+ J_{12}.
	\end{split}
	\end{equation}
	where
	$R_1^{n+1}=\phi^{n+\frac{1}{2}}_{t}-\frac{\phi^{n+1}-\phi^n}{\tau}$, 
	$R_2^{n+1} =\tau\delta_{t}\phi^{n+1} $,
	$R_3^{n+1} =\delta_{tt}\phi^{n+1} $,
	$R_4^{n+1}= \frac{3\phi^{n+1}+\phi^{n-1}}{4}-\phi^{n+\frac{1}{2}}$.
	For the left side, we have 
	\begin{equation}\label{eq:left:1}
	\begin{split}
	&-\left(\frac{\sigma_h^{n+1}- \sigma_h^{n}}{\tau}, \Delta^{-1} \frac{3\sigma_h^{n+1}+\sigma_h^{n-1}}{4}\right)\\
	=&\frac{1}{2\tau }(\|\sigma_h^{n+1}\|_{-1}^{2}-\|\sigma_h^{n}\|_{-1}^{2})
	+\frac{1}{8\tau }(\|\delta_t \sigma_h^{n+1}\|_{-1}^{2}-\|\delta_t \sigma_h^{n}\|_{-1}^{2})
	+\frac{1}{8\tau }\|\delta_{tt} \sigma_h^{n+1}\|_{-1}^{2},
	\end{split}
	\end{equation}
	
	\begin{equation}\label{eq:left:2}
	\begin{split}
	&A\tau\left(\nabla \delta_{t}\sigma_h^{n+1}, \nabla \frac{3\sigma_h^{n+1}+\sigma_h^{n-1}}{4}  \right)\\
	=&\frac{A\tau}{2}(\|\nabla \sigma_h^{n+1}\|^{2}-\|\nabla \sigma_h^{n}\|^{2})
	+\frac{A\tau}{8}(\|\nabla  \delta_t\sigma_h^{n+1}\|^{2}-\|\nabla \delta_t \sigma_h^{n}\|^{2})
	+\frac{A\tau}{8}\|\nabla \delta_{tt}\sigma_h^{n+1}\|^{2},
	\end{split}
	\end{equation}	
	For the right side, by using $(\nabla \rho, \nabla \psi_h)=0 $, $\forall \psi_h \in S_h$, we have
	\begin{equation}\label{eq:right:J1}
	\begin{split}
	J_1
	= -\varepsilon\left(\nabla \frac{3\rho^{n+1}+\rho^{n-1}}{4}, \nabla \frac{3\sigma_h^{n+1}+\sigma_h^{n-1}}{4}  \right)=0,
	\end{split}
	\end{equation}  
	and
	\begin{equation}\label{eq:right:J5}
	\begin{split}
	J_5
	= -A\tau\left(\nabla \delta_{t}\rho^{n+1}, \nabla \frac{3\sigma_h^{n+1}+\sigma_h^{n-1}}{4}  \right)
	=0,
	\end{split}
	\end{equation} 
	Then, we estimate the terms on the right hand side of \eqref{newcn:cet1-3}
	\begin{equation}\label{eq:right:J2}
	\begin{split}
	J_2
	=& -B\left( \delta_{tt}\sigma_h^{n+1}, \frac{3\sigma_h^{n+1}+\sigma_h^{n-1}}{4}  \right)\\
	\leq& \frac{B^2}{\eta_0}\|\delta_{tt}\sigma_h^{n+1}\|_{-1}^{2}+\frac{\eta_0}{4}\|\nabla \frac{3\sigma_h^{n+1}+\sigma_h^{n-1}}{4}\|^{2},
	\end{split}
	\end{equation}    
	
	\begin{equation}\label{eq:right:J3:1}
	\begin{split}
	J'_3
	=& -\frac{1}{\varepsilon}\left(f(\frac{3}{2}\phi_h^{n}-\frac{1}{2}\phi_h^{n-1})-f(\phi^{n+\frac{1}{2}}), \frac{3\sigma_h^{n+1}+\sigma_h^{n-1}}{4}  \right)\\
	\leq&\frac{L}{ \varepsilon}\left( |\frac{3}{2}\sigma_h^{n}-\frac{1}{2}\sigma_h^{n-1}+ \frac{3}{2}\rho^{n}-\frac{1}{2}\rho^{n-1} + R_5^{n+1}|,| \frac{3\sigma_h^{n+1}+\sigma_h^{n-1}}{4}|  \right)\\
	\leq& \frac{ L^2}{ \varepsilon^2 \eta_0}\|\frac{3}{2}\sigma_h^{n}-\frac{1}{2}\sigma_h^{n-1}\|_{-1}^{2}
	+\frac{ L^2}{ \varepsilon^2 \eta_0}\|\frac{3}{2}\rho^{n}-\frac{1}{2}\rho^{n-1}\|_{-1}^{2}
	+\frac{ L^2}{ \varepsilon^2 \eta_0}\|R_5^{n+1}\|_{-1}^{2}\\
	&+\frac{3\eta_0}{4}\|\nabla \frac{3\sigma_h^{n+1}+\sigma_h^{n-1}}{4}\|^{2},
	\end{split}
	\end{equation}   
	where
	\begin{equation} \label{newcn:e5-5}
	{R_{5}^{n+1}=}
	\frac{3}{2}\phi(t^{n})-\frac{1}{2}\phi(t^{n-1})-\phi(t^{n+\frac{1}{2}}).
	\end{equation}       		
	\begin{equation}\label{eq:right:J6}
	\begin{split}
	J_6
	=&-B\left( \delta_{tt}\rho^{n+1}, \frac{3\sigma_h^{n+1}+\sigma_h^{n-1}}{4}  \right)\\
	\leq& \frac{B^2 }{\eta_0}\| \delta_{tt}\rho^{n+1}\|_{-1}^{2}
	+\frac{\eta_0}{4}\|\nabla \frac{3\sigma_h^{n+1}+\sigma_h^{n-1}}{4}\|^{2},
	\end{split}
	\end{equation}   	
	\begin{equation}\label{eq:right:J7}
	\begin{split}
	J_7
	=&\left(\frac{\rho^{n+1}- \rho^{n}}{\tau}, \Delta^{-1}\left( \frac{3\sigma_h^{n+1}+\sigma_h^{n-1}}{4}\right)\right)\\
	\leq& \frac{1 }{\eta_0 }\|\Delta^{-1} \frac{\delta_{t}\rho^{n+1}}{\tau}\|_{-1}^{2}+\frac{\eta_0}{4}\|\nabla \frac{3\sigma_h^{n+1}+\sigma_h^{n-1}}{4}\|^{2},
	\end{split}
	\end{equation}       	
	\begin{equation}\label{eq:right:J8}
	\begin{split}
	J_8
	=&-\left(\mu^{n+\frac{1}{2}}-R_h\mu^{n+\frac{1}{2}},  \frac{3\sigma_h^{n+1}+\sigma_h^{n-1}}{4} \right)\\
	\leq& \frac{1 }{\eta_0 }\| \mu^{n+\frac{1}{2}}-R_h\mu^{n+\frac{1}{2}}\|_{-1}^{2}+\frac{\eta_0}{4}\|\nabla \frac{3\sigma_h^{n+1}+\sigma_h^{n-1}}{4}\|^{2},
	\end{split}
	\end{equation}      	
	\begin{equation}\label{eq:right:J9}
	\begin{split}
	J_9
	=&-\left(R_1^{n+1}, \Delta^{-1}\left( \frac{3\sigma_h^{n+1}+\sigma_h^{n-1}}{4}\right)\right)\\
	\leq& \frac{1 }{\eta_0 }\|\Delta^{-1} R_1^{n+1}\|_{-1}^{2}+\frac{\eta_0}{4}\|\nabla \frac{3\sigma_h^{n+1}+\sigma_h^{n-1}}{4}\|^{2},
	\end{split}
	\end{equation} 	
	\begin{equation}\label{eq:right:J10}
	\begin{split}
	J_{10}
	=&-A\left(\nabla R_2^{n+1}, \nabla \frac{3\sigma_h^{n+1}+\sigma_h^{n-1}}{4}  \right)\\
	\leq &\frac{A^2 }{\eta_0 }\|\nabla R_2^{n+1}\|^{2}+\frac{\eta_0}{4}\|\nabla \frac{3\sigma_h^{n+1}+\sigma_h^{n-1}}{4}\|^{2},
	\end{split}
	\end{equation} 	
	\begin{equation}\label{eq:right:J11}
	\begin{split}
	J_{11}
	=&-B\left(R_3^{n+1},  \frac{3\sigma_h^{n+1}+\sigma_h^{n-1}}{4}  \right)\\
	\leq& \frac{B^2 }{\eta_0 }\| R_3^{n+1}\|_{-1}^{2}+\frac{\eta_0}{4}\|\nabla \frac{3\sigma_h^{n+1}+\sigma_h^{n-1}}{4}\|^{2},
	\end{split}
	\end{equation}    	
	\begin{equation}\label{eq:right:J12}
	\begin{split}
	J_{12}
	=&    -\varepsilon\left(\nabla R_4^{n+1}, \nabla \frac{3\sigma_h^{n+1}+\sigma_h^{n-1}}{4}  \right)\\
	\leq &\frac{\varepsilon^2 }{\eta_0 }\|\nabla R_4^{n+1}\|^{2}+\frac{\eta_0}{4}\|\nabla \frac{3\sigma_h^{n+1}+\sigma_h^{n-1}}{4}\|^{2}.
	\end{split}
	\end{equation}
	
	Substituting \eqref{eq:left:1}-\eqref{eq:right:J12} into \eqref{newcn:cet1-3}, we have	
	\begin{equation}\label{newcn:cet2}
	\begin{split}
	&\frac{1}{2\tau }(\|\sigma_h^{n+1}\|_{-1}^{2}-\|\sigma_h^{n}\|_{-1}^{2})
	+\frac{1}{8\tau }(\|\delta_t \sigma_h^{n+1}\|_{-1}^{2}-\|\delta_t \sigma_h^{n}\|_{-1}^{2})
	+\frac{1}{8\tau }\|\delta_{tt} \sigma_h^{n+1}\|_{-1}^{2}\\
	&+\frac{A\tau}{2}(\|\nabla \sigma_h^{n+1}\|^{2}-\|\nabla \sigma_h^{n}\|^{2})
	+\frac{A\tau}{8}(\|\nabla  \delta_t\sigma_h^{n+1}\|^{2}-\|\nabla \delta_t \sigma_h^{n}\|^{2})\\
	&+\frac{A\tau}{8}\|\nabla \delta_{tt}\sigma_h^{n+1}\|^{2}
	+\varepsilon \|\nabla \frac{3\sigma_h^{n+1}+\sigma_h^{n-1}}{4}\|^2\\
	\leq & 
	\frac{B^2}{\eta_0}\|\delta_{tt}\sigma_h^{n+1}\|_{-1}^{2}
	+\frac{ L^2}{ \varepsilon^2 \eta_0}\|\frac{3}{2}\sigma_h^{n}-\frac{1}{2}\sigma_h^{n-1}\|_{-1}^{2}
	+\frac{11\eta_0}{4}\|\nabla \frac{3\sigma_h^{n+1}+\sigma_h^{n-1}}{4}\|^{2}\\
	&
	+\frac{ L^2}{ \varepsilon^2 \eta_0}\|\frac{3}{2}\rho^{n}-\frac{1}{2}\rho^{n-1}\|_{-1}^{2}
	+\frac{B^2 }{\eta_0}\| \delta_{tt}\rho^{n+1}\|_{-1}^{2}
	+\frac{1 }{\eta_0 }\|\Delta^{-1} \frac{\delta_{t}\rho^{n+1}}{\tau}\|_{-1}^{2}\\
	&+\frac{1 }{\eta_0 }\| \mu^{n+\frac{1}{2}}-R_h\mu^{n+\frac{1}{2}}\|_{-1}^{2}
	+\frac{1 }{\eta_0 }\|\Delta^{-1} R_1^{n+1}\|_{-1}^{2}
	+\frac{A^2 }{\eta_0 }\|\nabla R_2^{n+1}\|^{2}\\
	&+\frac{B^2 }{\eta_0 }\| R_3^{n+1}\|_{-1}^{2}
	+\frac{\varepsilon^2 }{\eta_0 }\|\nabla R_4^{n+1}\|^{2}
	+\frac{ L^2}{ \varepsilon^2 \eta_0}\|R_5^{n+1}\|_{-1}^{2}.
	\end{split}
	\end{equation}

	For the $R_1, \ldots, R_5$ terms, we have following estimates:
	\begin{equation} \label{newcn:R1}
	\|\Delta^{-1} R_{1}^{n+1}\|_{-1}^{2}  \lesssim
	\tau^{3}\int_{t^{n}}^{t^{n+1}}\|\Delta^{-1}\phi_{tt}\|_{-1}^{2}{\rm d}t,
	\end{equation}
	\begin{equation} \label{newcn:R2}
	\|\nabla R_{2}^{n+1}\|^{2}  \lesssim
	\tau^{3}\int_{t^{n}}^{t^{n+1}}\|\nabla \phi_t\|^{2}{\rm d}t,
	\end{equation}
	\begin{equation}\label{newcn:R3}
	\|R_{3}^{n+1}\|_{-1}^{2}  \lesssim 6\tau^{3}\int_{t^{n-1}}^{t^{n+1}}\|\phi_{tt}\|_{-1}^{2}{\rm d}t,
	\end{equation}
	\begin{equation}\label{newcn:R4}
	\|\nabla R_{4}^{n+1}\|^{2}   \lesssim
	\tau^{3}\int_{t^{n}}^{t^{n+1}}\| \nabla \phi_{tt}\|^{2}{\rm d}t,
	\end{equation}
	\begin{equation}\label{newcn:R5}
	\|R_{5}^{n+1}\|_{-1}^{2}  \lesssim \tau^{3}\int_{t^{n-1}}^{t^{n+1}}\|\phi_{tt}\|_{-1}^{2}{\rm d}t.
	\end{equation}   
	
	For including $\rho$ terms, using \eqref{eq:Rh3}-\eqref{eq:Rh4}, we have the following estimates:
	\begin{equation} \label{eq:rho2}
	\begin{split}
	\frac{ L^2}{ \varepsilon^2 \eta_0}\|\frac{3}{2}\rho^{n}-\frac{1}{2}\rho^{n-1}\|_{-1}^{2}
	\leq& \frac{ L^2}{ 2\varepsilon^2 \eta_0}(9\|\rho^{n}\|^2_{-1} + \|\rho^{n-1}\|^2_{-1})\\
	\lesssim  &\frac{ L^2}{ 2\varepsilon^2 \eta_0}h^{2(q+2)}(9\| \phi^{n}\|^2_{H^{q+1}} + \| \phi^{n-1}\|^2_{H^{q+1}}),
	\end{split}
	\end{equation}
	
	
	\begin{equation} \label{eq:rho４}
	\begin{split}
	\frac{B^2 }{\eta_0}\| \delta_{tt}\rho^{n+1}\|_{-1}^{2}
	\lesssim & \frac{B^2 }{\eta_0} h^{2(q+2)}\|\delta_{tt} \phi^{n+1}\|^2_{H^{q+1}}\\
	\lesssim & \frac{6B^2 }{\eta_0} \tau^{3}　h^{2(q+2)}\int_{t^{n-1}}^{t^{n+1}}\| \phi_{tt}\|^2_{H^{q+1}}{\rm d}t,
	\end{split}
	\end{equation}
	
	\begin{equation} \label{eq:rho5}
	\begin{split}
	\frac{1 }{\eta_0 }\|\Delta^{-1} \frac{\delta_{t}\rho^{n+1}}{\tau}\|_{-1}^{2}
	\lesssim & \frac{1 }{\eta_0 \tau^2} h^{2(q+2)}\|\delta_t \phi^{n+1}\|^2_{H^{q+1}}\\
	\lesssim & \frac{1 }{\eta_0 \tau} 　h^{2(q+2)}\int_{t^{n}}^{t^{n+1}}\| \phi_t\|^2_{H^{q+1}}{\rm d}t,
	\end{split}
	\end{equation}
	
	\begin{equation} \label{eq:rho6}
	\begin{split}
	\frac{1 }{\eta_0 }\| \mu^{n+\frac{1}{2}}-R_h\mu^{n+\frac{1}{2}}\|_{-1}^{2}
	\lesssim & \frac{1 }{\eta_0 } h^{2(q+2)}\| \mu^{n+1}\|^2_{H^{q+1}}.
	\end{split}
	\end{equation}
	
	Multiplying \eqref{newcn:cet2} with $2\tau$, taking $\eta_0 = 2 \varepsilon/11$,
	and submitting \eqref{newcn:R1}-\eqref{newcn:R5}, \eqref{eq:rho2}-\eqref{eq:rho6} into \eqref{newcn:cet2}, we have
	\begin{equation}\label{newcn:cet3}
	\begin{split}
	&(\|\sigma_h^{n+1}\|_{-1}^{2}-\|\sigma_h^{n}\|_{-1}^{2})
	+\frac{1}{4}(\|\delta_t \sigma_h^{n+1}\|_{-1}^{2}-\|\delta_t \sigma_h^{n}\|_{-1}^{2})
	+\frac{1}{4 }\|\delta_{tt} \sigma_h^{n+1}\|_{-1}^{2}\\
	&+A\tau^2(\|\nabla \sigma_h^{n+1}\|^{2}-\|\nabla \sigma_h^{n}\|^{2})
	+\frac{A\tau^2}{4}(\|\nabla  \delta_t\sigma_h^{n+1}\|^{2}-\|\nabla \delta_t \sigma_h^{n}\|^{2})\\
	&+\frac{A\tau^2}{4}\|\nabla \delta_{tt}\sigma_h^{n+1}\|^{2}
	+\varepsilon \tau\|\nabla \frac{3\sigma_h^{n+1}+\sigma_h^{n-1}}{4}\|^2\\
	\leq & \frac{11 B^2 \tau}{ \varepsilon}\|\delta_{tt}\sigma_h^{n+1}\|^2_{-1}
	+\frac{99 L^2}{2\varepsilon^3}  \tau\|\sigma_h^{n}\|^2_{-1}
	+\frac{11 L^2}{2\varepsilon^3} \tau \|\sigma_h^{n-1}\|^2_{-1}\\
	&+C_1^{n+1} \tau^4 + C_2^{n+1} h^{2q+4},
	\end{split}
	\end{equation}  
	where 
	\begin{equation}\label{eq:C1}
	\begin{split}
	C_1^{n+1} =&\frac{11}{\varepsilon} \int_{t^{n}}^{t^{n+1}}  
	(\|\Delta^{-1}\phi_{tt}\|_{-1}^{2}+ A^2\|\nabla \phi_t\|^{2}
	+ \varepsilon^2 \| \nabla \phi_{tt}\|^{2}){\rm d}t\\
	&+\frac{11}{\varepsilon}\int_{t^{n-1}}^{t^{n+1}} (6B^2 \|\phi_{tt}\|_{-1}^{2}
	+\frac{L^2}{\varepsilon^2}\|\phi_{tt}\|_{-1}^{2}){\rm d}t,
	\end{split}
	\end{equation} 
	\begin{equation}\label{eq:C2}
	\begin{split}
	C_2^{n+1} =&\frac{11  L^2 }{2 \varepsilon^3}\tau(9 \| \phi^{n}\|^2_{H^{q+1}} 
	+ \| \phi^{n-1}\|^2_{H^{q+1}})
	+\frac{11  L^2 }{ \varepsilon}\tau \| \mu^{n+1}\|^2_{H^{q+1}} \\
	&+\frac{11 }{ \varepsilon}　\int_{t^{n}}^{t^{n+1}} 
	\| \phi_t\|^2_{H^{q+1}}{\rm d}t
	+\frac{66B^2 }{\varepsilon} \tau^{4}\int_{t^{n-1}}^{t^{n+1}}
	\| \phi_{tt}\|^2_{H^{q+1}}{\rm d}t.
	\end{split}
	\end{equation} 
	Suppose $\tau \lesssim \varepsilon^3$, 
	then　$\frac{11B^2}{\varepsilon}\tau \|\delta_{tt}\sigma_h^{n+1}\|_{-1}^{2} 
	\leq \frac{1}{4 } \|\delta_{tt}\sigma_h^{n+1}\|_{-1}^{2} $, we get \eqref{newcn:cet3-1}.
	Summing up \eqref{newcn:cet3} from $n=1$ to $n=N$, 
	by discrete Gronwall's inequality and assumption, we get \eqref{newcn:cet3-2}, where 
	\begin{equation}\label{eq:C1:0}
	\begin{split}
	C_1 =&\frac{11}{\varepsilon} \int_{0}^{T}  
	(\|\Delta^{-1}\phi_{tt}\|_{-1}^{2}+ A^2\|\nabla \phi_t\|^{2}
	+ \varepsilon^2 \| \nabla \phi_{tt}\|^{2}){\rm d}t\\
	&+\frac{22}{\varepsilon}\int_{0}^{T} (6B^2 \|\phi_{tt}\|_{-1}^{2}
	+\frac{L^2}{\varepsilon^2}\|\phi_{tt}\|_{-1}^{2}){\rm d}t\\
	\lesssim & \varepsilon^{-\max\{\rho_{1}+1, \rho_{2}+3, \rho_{3}-1, \rho_{5}+5\}}:=\gamma_1(\varepsilon),
	\end{split}
	\end{equation} 
	and 
	\begin{equation}\label{eq:C2:0}
	\begin{split}
	C_2 =&\frac{ 55L^2 \tau}{ \varepsilon^3} \sum_{n=1}^{N} 
	\| \phi^{n}\|^2_{H^{q+1}}
	+\frac{ 11\tau}{\varepsilon}\sum_{n=1}^{N} \| \mu^{n+1}\|^2_{H^{q+1}} \\
	&+\frac{11 }{ \varepsilon}　\int_{0}^{T}\|\phi_t\|^2_{H^{q+1}}{\rm d}t
	+\frac{132B^2 }{\varepsilon} \tau^{4}\int_{0}^{T}\|\phi_{tt}\|^2_{H^{q+1}}{\rm d}t\\
	\lesssim & \min\{\varepsilon^{-(\rho_{7}+3}), \varepsilon^{-(\rho_{8}+1)}, 
	\varepsilon^{-(\rho_6+1)}, \varepsilon^{-(\rho_4+3)}\tau^4\} :=\gamma_2(\varepsilon, \tau).
	\end{split}
	\end{equation} 	
\end{proof}

Proposition (\ref{newcn:prio}) is the usual error estimate, in
which the error growth depends on $T/\varepsilon^3$
exponentially. To obtain a finer estimate on the error, we
need to use a spectral estimate of the linearized
Cahn-Hilliard operator by Chen \cite{chen_spectrum_1994} for
the case when the interface is well developed in the
Cahn-Hilliard system.
\begin{lemma}\label{newcn:prop:01}
	Let $\phi(t)$ be the exact solution of the Cahn-Hilliard
	equation \eqref{eq:CH} with interfaces well developed
	in the initial condition (i.e. conditions (1.9)-(1.15) in
	\cite{chen_spectrum_1994} are satisfied).  Then there
	exist $0<\varepsilon_{0}\ll 1$ and positive constant
	$C_{0}$ such that the principle eigenvalue of the
	linearized Cahn-Hilliard operator
	$\mathcal{L}_{CH}:=\Delta(\varepsilon\Delta-\frac{1}{\varepsilon}
	f'(\phi)I)$ satisfies for all $t\in [0,T]$
	\begin{equation}\label{newcn:eigen}
	\lambda_{CH}=\inf_{\substack{0\neq v\in H^{1}(\Omega)\\ \Delta\omega=v}}
	\frac{\varepsilon\|\nabla v\|^{2}+\frac{1}{\varepsilon}(f'(\phi(\cdot,t))v,v)}{\|\nabla\omega\|^{2}}
	\geq-C_{0}
	\end{equation}
	for $\varepsilon\in (0,\varepsilon_{0})$. 
\end{lemma}

The following lemma which was proved by 
\cite{feng_numerical_2005} and \cite{alikakos_convergence_1994},
shows that the boundedness of the solution to the Cahn-Hilliard equation, provided that the sharp interface limit Hele-Shaw problem has a global 
(in time) classical solution. This is a condition of the finer error estimate.  
\begin{lemma}\label{newcn:lm:Linf}
	Suppose that f satisfies Assumption \ref{newcn:ap:1}, and the
	corresponding Hele-Shaw problem has a global (in time)
	classical solution. Then there exists a family of smooth
	initial datum functions
	$\{\phi_{0}^{\varepsilon}\}_{0< \varepsilon \leq 1}$ and
	constants $ \varepsilon_{0} \in (0,1]$ and $C > 0$ such
	that for all $\varepsilon \in (0,\varepsilon_{0})$ the
	solution $\phi(t)$ of the Cahn-Hilliard equation(\ref{eq:CH}) 
	with the above initial data $\phi_{0}^{\varepsilon}$ satisfies
	\begin{equation}\label{newcn:eq:phi}
	\|\phi(t)\|_{L^{\infty}(0,T; \Omega) }\leq C.
	\end{equation}
\end{lemma}

Now we present the refined error estimate.
\begin{theorem}\label{newcn:error}
	Suppose all of the Assumption \ref{newcn:ap:1}(i),(ii), Assumption \ref{newcn:ap:2} 
	and Lemma \ref{lm:1}, \ref{newcn:lm:Linf} hold. Let time
	step $\tau$ satisfy the following constraint
	\begin{equation}\label{eq:condition1}
	\begin{split}
	\tau
	\lesssim \min\{\varepsilon^6, \varepsilon^{\frac{38+d}{18-d}}\gamma_3(\varepsilon)^{-\frac{4}{18-d}}\}.
	\end{split}
	\end{equation}  
	and
	\begin{equation}\label{eq:condition}
	\begin{split}
	h
	\lesssim \min\{ 
	\varepsilon^{\frac{7d-26}{8(q+2)}}\gamma_4(\varepsilon,\tau)^{-\frac{1}{2(q+2)}},
	\varepsilon^{\frac{2d+76}{(18-d)(q+2)}}
	\gamma_3(\varepsilon)^{-\frac{d-2}{2(18-d)}}
	\gamma_4(\varepsilon,\tau)^{-\frac{1}{2(q+2)}}\},
	\end{split}		
	\end{equation} 	
	then we have the error estimate
	\begin{equation}\label{eq:conclusion1}
	\begin{split}
	&\max_{1\leq n\leq N}\|e^{n+1}\|_{-1}^{2} \\
	\lesssim &
	\mathrm{exp }(5(C_{0}+L^{2})T) \left(\gamma_3(\varepsilon) 
	\tau^4 + \gamma_4(\varepsilon, \tau) h^{2q+4}\right) 
	+ \varepsilon^{-\rho_9}h^{2q+4} ,
	\end{split}
	\end{equation}
	
	\begin{equation}\label{eq:conclusion2}
	\begin{split}
	& \tau\sum_{n=1}^{N}\|\nabla\frac{3e^{n+1}+e^{n-1}}{4}\|^{2}
	\\
	\lesssim &
	\mathrm{exp }(5(C_{0}+L^{2})T) \left(\gamma_3(\varepsilon)\varepsilon^{-4} \tau^4 + \gamma_4(\varepsilon, \tau)\varepsilon^{-4} h^{2q+4}\right)
	+\varepsilon^{-\rho_7}h^{2q},
	\end{split}
	\end{equation}	
	where $\gamma_3(\varepsilon):=\varepsilon^{-\max\{\rho_{1}+4, \rho_{2}+6, \rho_{3}+2, \rho_{5}+8\}}$,\\
	$\gamma_4(\varepsilon, \tau):=\min\{\varepsilon^{-\max\{\rho_6+4, \rho_{7}+6, \rho_{8}+4)\}}, \varepsilon^{-(\rho_4+6)}\tau^4\}$.
\end{theorem}	
\begin{proof}
	(i) To get a better convergence result, we reestimate $J'_3$ as
	\begin{equation}\label{eq:right:J3:2}
	\begin{split}
	& J'_3\\
	=& -\frac{1}{\varepsilon}\left(f(\frac{3}{2}\phi_h^{n}-\frac{1}{2}\phi_h^{n-1})-f(\frac{3\phi_h^{n+1}+ \phi^{n-1}_h}{4}), \frac{3\sigma_h^{n+1}+\sigma_h^{n-1}}{4}  \right)\\
	&-\frac{1}{\varepsilon}\left(f(\frac{3\phi_h^{n+1}+ \phi^{n-1}_h}{4})-f(\phi^{n+\frac{1}{2}}), \frac{3\sigma_h^{n+1}+\sigma_h^{n-1}}{4}  \right)\\
	=&: J_3+J_4,
	\end{split}
	\end{equation}  
	
	\begin{equation}\label{eq:right:J3}
	\begin{split}
	J_3
	=& -\frac{1}{\varepsilon}\left(f(\frac{3}{2}\phi_h^{n}-\frac{1}{2}\phi_h^{n-1})-f(\frac{3\phi_h^{n+1}+ \phi^{n-1}_h}{4}), \frac{3\sigma_h^{n+1}+\sigma_h^{n-1}}{4}  \right)\\
	\leq &\frac{3L}{4 \varepsilon}\left( |\delta_{tt}\sigma_h^{n+1}+\delta_{tt}\rho^{n+1}+ R_3^{n+1}|,| \frac{3\sigma_h^{n+1}+\sigma_h^{n-1}}{4}|  \right)\\
	\leq& \frac{9 L^2}{16 \varepsilon^2 \eta}\|\delta_{tt}\sigma_h^{n+1}\|_{-1}^{2}+
	\frac{9 L^2}{16 \varepsilon^2 \eta}\|R_3^{n+1}\|_{-1}^{2}
	+\frac{9 L^2}{16 \varepsilon^2 \eta}\|\delta_{tt}\rho^{n+1}\|_{-1}^{2}\\
	&+\frac{3\eta}{4}\|\nabla \frac{3\sigma_h^{n+1}+\sigma_h^{n-1}}{4}\|^{2},
	\end{split}
	\end{equation}   
	
	\begin{equation}\label{eq:right:J4}
	\begin{split}
	J_4
	=& -\frac{1}{\varepsilon}\left(f(\frac{3\phi_h^{n+1}+ \phi^{n-1}_h}{4})-f(\phi^{n+\frac{1}{2}}), \frac{3\sigma_h^{n+1}+\sigma_h^{n-1}}{4}  \right)\\
	=&-\frac{1}{\varepsilon}\left( f'(\phi^{n+\frac{1}{2}})(\frac{3\sigma_h^{n+1}+\sigma_h^{n-1}}{4}   +
	\frac{3\rho^{n+1}+\rho^{n-1}}{4} +
	R_4^{n+1}), \frac{3\sigma_h^{n+1}+\sigma_h^{n-1}}{4} \right)\\
	&-\frac{1}{2\varepsilon}\left( f''(\theta^{n+\frac{1}{2}})(\frac{3\sigma_h^{n+1}+\sigma_h^{n-1}}{4}   +
	\frac{3\rho^{n+1}+\rho^{n-1}}{4} +
	R_4^{n+1})^2, \frac{3\sigma_h^{n+1}+\sigma_h^{n-1}}{4} \right)\\
	\leq &-\frac{1}{\varepsilon}\left( f'(\phi^{n+\frac{1}{2}})\frac{3\sigma_h^{n+1}+\sigma_h^{n-1}}{4}, \frac{3\sigma_h^{n+1}+\sigma_h^{n-1}}{4} \right)
	+\eta\|\nabla \frac{3\sigma_h^{n+1}+\sigma_h^{n-1}}{4}\|^{2}\\
	&
	+\frac{3L_2}{2\varepsilon}\|\frac{3\sigma_h^{n+1}+\sigma_h^{n-1}}{4}\|_{L^3}^3
	+ \frac{L^2}{\varepsilon^2 \eta} \|\frac{3\rho^{n+1}+\rho^{n-1}}{4}\|_{-1}^2
	+\frac{L^2}{\varepsilon^2 \eta} \|R_4^{n+1}\|_{-1}^2
	\\
	&+\frac{9L_2^2}{4\varepsilon^2\eta}\|\frac{3\rho^{n+1}+\rho^{n-1}}{4}\|_{\infty}^2\|\frac{3\rho^{n+1}+\rho^{n-1}}{4}\|_{-1}^2
	+\frac{9L_2^2}{4\varepsilon^2\eta}
	\|R_4^{n+1}\|^2_{\infty}\|R_4^{n+1}\|^2_{-1}.
	\end{split}
	\end{equation}  
	Replacing $\eta_0$ with $\eta$ and
	submitting \eqref{eq:left:1}-\eqref{eq:right:J2}, \eqref{eq:right:J5}-\eqref{eq:right:J12}, \eqref{eq:right:J3:2}-\eqref{eq:right:J4}
	into \eqref{newcn:cet1-3}, we get
	\begin{equation}\label{eq:error:1} 
	\begin{split}
	&\frac{1}{2\tau }(\|\sigma_h^{n+1}\|_{-1}^{2}-\|\sigma_h^{n}\|_{-1}^{2})
	+\frac{1}{8\tau }(\|\delta_t \sigma_h^{n+1}\|_{-1}^{2}-\|\delta_t \sigma_h^{n}\|_{-1}^{2})
	+\frac{1}{8\tau }\|\delta_{tt} \sigma_h^{n+1}\|_{-1}^{2}\\
	&+\frac{A\tau}{2}(\|\nabla \sigma_h^{n+1}\|^{2}-\|\nabla \sigma_h^{n}\|^{2})
	+\frac{A\tau}{8}(\|\nabla \delta_t\sigma_h^{n+1}\|^{2}
	-\|\nabla \delta_t\sigma_h^{n}\|^{2})\\
	&+\frac{A\tau}{8}\|\nabla \delta_{tt}\sigma_h^{n+1}\|^{2}
	+\varepsilon \|\nabla \frac{3\sigma_h^{n+1}+\sigma_h^{n-1}}{4}\|^2\\
	\leq & 
	-\frac{1}{\varepsilon}\left( f'(\phi^{n+\frac{1}{2}})\frac{3\sigma_h^{n+1}+\sigma_h^{n-1}}{4}, \frac{3\sigma_h^{n+1}+\sigma_h^{n-1}}{4} \right)
	+\frac{3L_2}{2\varepsilon}\|\frac{3\sigma_h^{n+1}+\sigma_h^{n-1}}{4}\|_{L^3}^3\\
	&+\left(\frac{B^2}{\eta} +\frac{9 L^2}{16 \varepsilon^2 \eta} 
	\right)
	\|\delta_{tt}\sigma_h^{n+1}\|_{-1}^{2}
	+\frac{15\eta}{4}\|\nabla \frac{3\sigma_h^{n+1}+\sigma_h^{n-1}}{4}\|^{2}\\
	&+\left(\frac{L^2}{\varepsilon^2 \eta} 
	+\frac{9L_2^2}{4\varepsilon^2\eta}\|\frac{3\rho^{n+1}+\rho^{n-1}}{4}\|_{\infty}^2\right)\|\frac{3\rho^{n+1}+\rho^{n-1}}{4}\|_{-1}^2
	\\
	&+\left(\frac{B^2}{\eta} +\frac{9 L^2}{16 \varepsilon^2 \eta} 
	\right) \| \delta_{tt}\rho^{n+1}\|_{-1}^{2}
	+\frac{1 }{\eta }\|\Delta^{-1} \frac{\delta_{t}\rho^{n+1}}{\tau}\|_{-1}^{2}
	+\frac{1 }{\eta}\| \mu^{n+\frac{1}{2}}-R_h\mu^{n+\frac{1}{2}}\|_{-1}^{2}\\
	&+\frac{1 }{\eta }\|\Delta^{-1} R_1^{n+1}\|_{-1}^{2}
	+\frac{A^2 }{\eta }\|\nabla R_2^{n+1}\|^{2}
	+\left(\frac{B^2}{\eta} +\frac{9 L^2}{16 \varepsilon^2 \eta} 
	\right)\| R_3^{n+1}\|_{-1}^{2}\\
	&
	+\frac{\varepsilon^2 }{\eta}\|\nabla R_4^{n+1}\|^{2}
	+\left(\frac{L^2}{\varepsilon^2 \eta} 
	+\frac{9L_2^2}{4\varepsilon^2\eta}
	\|R_4^{n+1}\|^2_{\infty}\right)\|R_4^{n+1}\|^2_{-1}.
	\end{split}
	\end{equation} 
	We need to bound the last two terms on the right hand side 
	of the above inequality.
	
	(ii) Now, we estimate the last two terms of the right hand side of \eqref{eq:error:1}. The spectrum estimate \eqref{newcn:eigen} leads to 
	\begin{equation}\label{eq:error:2} 
	\begin{split}
	&\varepsilon \|\nabla \frac{3\sigma_h^{n+1}+\sigma_h^{n-1}}{4}\|^2
	+\frac{1}{\varepsilon}\left( f'(\phi^{n+\frac{1}{2}})\frac{3\sigma_h^{n+1}+\sigma_h^{n-1}}{4}, \frac{3\sigma_h^{n+1}+\sigma_h^{n-1}}{4} \right)\\
	\geq &-C_0 \| \frac{3\sigma_h^{n+1}+\sigma_h^{n-1}}{4}\|^2_{-1},
	\end{split}
	\end{equation} 
	applying \eqref{eq:error:2} with a scaling factor $(1-\eta_1)$ 
	close to but smaller than 1, we get 
	\begin{equation}\label{eq:error:3} 
	\begin{split}
	&-(1-\eta_1)\frac{1}{\varepsilon}\left( f'(\phi^{n+\frac{1}{2}})\frac{3\sigma_h^{n+1}+\sigma_h^{n-1}}{4}, \frac{3\sigma_h^{n+1}+\sigma_h^{n-1}}{4} \right)\\
	\leq &C_0 (1-\eta_1)\| \frac{3\sigma_h^{n+1}+\sigma_h^{n-1}}{4}\|^2_{-1}
	+(1-\eta_1)\varepsilon \|\nabla \frac{3\sigma_h^{n+1}+\sigma_h^{n-1}}{4}\|^2.
	\end{split}
	\end{equation}
	On the other hand,
	\begin{equation}\label{eq:error:4} 
	\begin{split}
	&-\frac{\eta_1}{\varepsilon}\left( f'(\phi^{n+\frac{1}{2}})\frac{3\sigma_h^{n+1}+\sigma_h^{n-1}}{4}, \frac{3\sigma_h^{n+1}+\sigma_h^{n-1}}{4} \right)\\
	\leq &\frac{L^2 \eta_1}{\varepsilon^2 \eta_2}\| \frac{3\sigma_h^{n+1}+\sigma_h^{n-1}}{4}\|^2_{-1}
	+\frac{\eta_1 \eta_2}{4} \|\nabla \frac{3\sigma_h^{n+1}+\sigma_h^{n-1}}{4}\|^2.
	\end{split}
	\end{equation}
	
	Now, we estimate the $L^{3}$ term. By interpolating $L^{3}$ between $L^{2}$ and $H^{1}$ then using Poincare inequality for the error function, we get
	\[\|\frac{3\sigma_h^{n+1}+\sigma_h^{n-1}}{4}\|_{L^{3}}^{3}\leq K \|\nabla \frac{3\sigma_h^{n+1}+\sigma_h^{n-1}}{4}\|^{\frac{d}{2}}
	\|\frac{3\sigma_h^{n+1}+\sigma_h^{n-1}}{4}\|^{\frac{6-d}{2}},\]
	where K is a constant independent of $\varepsilon$ and $\tau$.	
	We continue the estimate by using
	$\|\frac{\sigma_h^{n+1}+ \sigma_h^{n}}{2}\|^{2} \leq \|\nabla \frac{\sigma_h^{n+1}+ \sigma_h^{n}}{2}\| \|\frac{\sigma_h^{n+1}+ \sigma_h^{n}}{2}\|_{-1}$ to get
	\begin{equation}\label{newcn:t20}
	\begin{split}
	\frac{3L_{2}}{2\varepsilon}\|\frac{3\sigma_h^{n+1}+\sigma_h^{n-1}}{4}\|_{L^{3}}^{3}
	\leq &\frac{3L_{2}}{2\varepsilon}K \|\nabla \frac{3\sigma_h^{n+1}+\sigma_h^{n-1}}{4}\|^{\frac{d}{2}+\frac{6-d}{4}}
	\|\frac{3\sigma_h^{n+1}+\sigma_h^{n-1}}{4}\|_{-1}^{\frac{6-d}{4}}\\
	:=&G^{n+1}\|\nabla \frac{3\sigma_h^{n+1}+\sigma_h^{n-1}}{4}\|^{2},
	\end{split}
	\end{equation}
	where $G^{n+1}=\frac{3L_{2}}{2\varepsilon}K \|\nabla \frac{3\sigma_h^{n+1}+\sigma_h^{n-1}}{4}\|^{\frac{d-2}{4}}
	\|\frac{3\sigma_h^{n+1}+\sigma_h^{n-1}}{4}\|_{-1}^{\frac{6-d}{4}}$. 
	
	Now plugging equation \eqref{eq:error:3}, \eqref{eq:error:4} and (\ref{newcn:t20}) into \eqref{eq:error:1}, we get
	\begin{equation}\label{eq:error:5} 
	\begin{split}
	&\frac{1}{2\tau }(\|\sigma_h^{n+1}\|_{-1}^{2}-\|\sigma_h^{n}\|_{-1}^{2})
	+\frac{1}{8\tau }(\|\delta_t \sigma_h^{n+1}\|_{-1}^{2}-\|\delta_t \sigma_h^{n}\|_{-1}^{2})
	+\frac{1}{8\tau }\|\delta_{tt} \sigma_h^{n+1}\|_{-1}^{2}\\
	&+\frac{A\tau}{2}(\|\nabla \sigma_h^{n+1}\|^{2}-\|\nabla \sigma_h^{n}\|^{2})
	+\frac{A\tau}{8}(\|\nabla \delta_t\sigma_h^{n+1}\|^{2}
	-\|\nabla \delta_t\sigma_h^{n}\|^{2})\\
	&+\frac{A\tau}{8}\|\nabla \delta_{tt}\sigma_h^{n+1}\|^{2}
	+\eta_1\varepsilon \|\nabla \frac{3\sigma_h^{n+1}+\sigma_h^{n-1}}{4}\|^2\\
	\leq & 
	\left(C_0 (1-\eta_1) +\frac{L^2 \eta_1}{\varepsilon^2 \eta_2}\right)\| \frac{3\sigma_h^{n+1}+\sigma_h^{n-1}}{4}\|^2_{-1}
	+\left(\frac{15\eta}{4}+ \frac{\eta_1 \eta_2}{4}\right) \|\nabla \frac{3\sigma_h^{n+1}+\sigma_h^{n-1}}{4}\|^2\\
	& + G^{n+1} \|\nabla \frac{3\sigma_h^{n+1}+\sigma_h^{n-1}}{4}\|^2
	+\left(\frac{B^2}{\eta} +\frac{9 L^2}{16 \varepsilon^2 \eta} 
	\right)
	\|\delta_{tt}\sigma_h^{n+1}\|_{-1}^{2}\\
	&+\left(\frac{L^2}{\varepsilon^2 \eta} 
	+\frac{9L_2^2}{4\varepsilon^2\eta}\|\frac{3\rho^{n+1}
		+\rho^{n-1}}{4}\|_{\infty}^2\right)\|\frac{3\rho^{n+1}+\rho^{n-1}}{4}\|_{-1}^2 \\
	&+\left(\frac{B^2}{\eta} +\frac{9 L^2}{16 \varepsilon^2 \eta} 
	\right) \| \delta_{tt}\rho^{n+1}\|_{-1}^{2}
	+\frac{1 }{\eta }\|\Delta^{-1} \frac{\delta_{t}\rho^{n+1}}{\tau}\|_{-1}^{2}
	+\frac{1 }{\eta}\| \mu^{n+\frac{1}{2}}-R_h\mu^{n+\frac{1}{2}}\|_{-1}^{2}\\
	&+\frac{1 }{\eta }\|\Delta^{-1} R_1^{n+1}\|_{-1}^{2}
	+\frac{A^2 }{\eta }\|\nabla R_2^{n+1}\|^{2}
	+\left(\frac{B^2}{\eta} +\frac{9 L^2}{16 \varepsilon^2 \eta} 
	\right)\| R_3^{n+1}\|_{-1}^{2}\\
	&
	+\frac{\varepsilon^2 }{\eta}\|\nabla R_4^{n+1}\|^{2}
	+\left(\frac{L^2}{\varepsilon^2 \eta} 
	+\frac{9L_2^2}{4\varepsilon^2\eta}
	\|R_4^{n+1}\|^2_{\infty}\right)\|R_4^{n+1}\|^2_{-1}.
	\end{split}
	\end{equation} 
	Take $\eta_{1}=\varepsilon^{3}, \ \eta_{2}=\varepsilon,\ 
	\eta=\varepsilon^{4}/{15}$,
	such that
	\[ \frac{L^{2}\eta_1}{\varepsilon^{2}\eta_{2}}=L^{2}, \ \
	\frac{15\eta}{4} +\frac{\eta_{1}\eta_{2}}{4}
	= \frac{\varepsilon^{4}}{2},\]
	and take 
	\begin{equation}
	\tau \leq 
	\frac{1}{\frac{8B^{2}}{\eta}+\frac{9L^2}{2\varepsilon^2\eta}} 
	\lesssim \varepsilon^6,
	\end{equation}
	such that
	\begin{equation} \label{eq:error:6}
	\Big(\frac{B^{2}}{\eta}+\frac{9L^2}{16\varepsilon^2\eta}
	\Big)\|\delta_{t}\sigma_h^{n+1}\|^{2}_{-1}
	\leq
	\frac{1}{8\tau} \|\delta_{t}\sigma_h^{n+1}\|^{2}_{-1}.
	\end{equation}	
	By using \eqref{eq:error:6} and the taken values, multiplying $4\tau$ on both sides of inequality \eqref{eq:error:5}, we get
	\begin{equation}\label{eq:error:7} 
	\begin{split}
	&2(\|\sigma_h^{n+1}\|_{-1}^{2}-\|\sigma_h^{n}\|_{-1}^{2})
	+\frac{1}{2 }(\|\delta_t \sigma_h^{n+1}\|_{-1}^{2}-\|\delta_t \sigma_h^{n}\|_{-1}^{2})\\
	&+2A\tau^2(\|\nabla \sigma_h^{n+1}\|^{2}-\|\nabla \sigma_h^{n}\|^{2})
	+\frac{A\tau^2}{2}(\|\nabla \delta_t\sigma_h^{n+1}\|^{2}
	-\|\nabla \delta_t\sigma_h^{n}\|^{2})\\
	&+\frac{A\tau^2}{2}\|\nabla \delta_{tt}\sigma_h^{n+1}\|^{2}
	+2\varepsilon^4 \tau \|\nabla \frac{3\sigma_h^{n+1}+\sigma_h^{n-1}}{4}\|^2\\
	\leq & 
	\left(C_0 +L^2\right)\tau
	\left( \frac{9}{2}\|\sigma_h^{n+1}\|_{-1}^2 + \frac{1}{2}\|\sigma_h^{n-1}\|_{-1}^2 \right)
	+ 4G^{n+1}\tau \|\nabla \frac{3\sigma_h^{n+1}+\sigma_h^{n-1}}{4}\|^2\\
	&+ C_3^{n+1} \tau^4 + C_4^{n+1} h^{2q+4}.
	\end{split}
	\end{equation}   
	By using 
	\begin{equation}\label{newcn:R4-1}
	\|R_{4}^{n+1}\|^{2}_{-1}   \lesssim
	\tau^{3}\int_{t^{n}}^{t^{n+1}}\| \phi_{tt}\|^{2}_{-1}{\rm d}t,
	\end{equation}
	and $\|R_4^{n+1}\|^2_{\infty} \leq 8 C^2$,
	we have
	\begin{equation}\label{eq:C3}
	\begin{split}
	C_3^{n+1} =&\frac{60}{\varepsilon^4} 
	\int_{t^{n}}^{t^{n+1}}  \|\Delta^{-1}\phi_{tt}\|_{-1}^{2}
	+ A^2\|\nabla \phi_t\|^{2}
	+ \varepsilon^2 \| \nabla \phi_{tt}\|^{2}{\rm d}t\\
	&+\frac{60}{\varepsilon^4}\int_{t^{n}}^{t^{n+1}} \left(\frac{L^2}{\varepsilon^2 } 
	+\frac{18L_2^2 C^2}{\varepsilon^2}\right)
	\|\phi_{tt}\|_{-1}^{2}{\rm d}t\\
	&+\frac{60}{\varepsilon^4}\int_{t^{n-1}}^{t^{n+1}} 
	\left(6B^2+ \frac{27L^2}{8 \varepsilon^2}\right) \|\phi_{tt}\|_{-1}^{2}{\rm d}t.
	\end{split}
	\end{equation}
	On the other hand,
	\begin{equation}\label{eq:C4}
	\begin{split}
	C_4^{n+1} = &\frac{15}{4\varepsilon^6} \left(2L^2+9L^2_2C^2 \right)
	\tau (9\| \phi^{n+1}\|^2_{H^{q+1}} + \|\phi^{n-1}\|^2_{H^{q+1}})\\
	&+\frac{15}{\varepsilon^4}\left(24B^2 +\frac{27 L^2}{2 \varepsilon^2 } 
	\right)\tau^{4}　\int_{t^{n-1}}^{t^{n+1}}\|\phi_{tt}\|^2_{H^{q+1}}{\rm d}t\\
	&+\frac{60 }{\varepsilon^4} \int_{t^{n}}^{t^{n+1}}\|\phi_t\|^2_{H^{q+1}}{\rm d}t
	+\frac{60 }{\varepsilon^4 }\tau\| \mu^{n+1}\|^2_{H^{q+1}},
	\end{split}
	\end{equation} 	
	where
	\begin{equation}
	\begin{split}
	\|\frac{3\rho^{n+1}+\rho^{n-1}}{4}\|_{\infty}^2
	&\leq \frac{9}{8}\|\rho^{n+1}\|_{\infty}^2 + \frac{1}{8}\|\rho^{n-1}\|_{\infty}^2\\
	&\leq \frac{9}{8}\|\phi^{n+1}\|_{\infty}^2 + \frac{1}{8}\|\phi^{n-1}\|_{\infty}^2
	\leq2C^2,
	\end{split}
	\end{equation}	
	\begin{equation}
	\begin{split}
	\|\frac{3\rho^{n+1}+\rho^{n-1}}{4}\|_{-1}^2
	&\leq \frac{9}{8}\|\rho^{n+1}\|_{-1}^2 + \frac{1}{8}\|\rho^{n-1}\|_{-1}^2\\
	&\leq h^{2(q+2)} \left(\frac{9}{8}\|\phi^{n+1}\|_{H^{q+1}}^2 + \frac{1}{8}\|\phi^{n-1}\|_{H^{q+1}}^2 \right).
	\end{split}
	\end{equation}	
	
	Now, if $G^{n+1}$ is uniformly bounded by constant
	$\varepsilon^{4}/4$, we can sum up the inequality (\ref{eq:error:7}) for $n=1$ to
	$N$ to get the following estimate:
	\begin{equation}\label{eq:error:８} 
	\begin{split}
	&2(\|\sigma_h^{N+1}\|_{-1}^{2}-\|\sigma_h^{1}\|_{-1}^{2})
	+\frac{1}{2 }(\|\delta_t \sigma_h^{N+1}\|_{-1}^{2}-\|\delta_t \sigma_h^{1}\|_{-1}^{2})\\
	&+2A\tau^2(\|\nabla \sigma_h^{N+1}\|^{2}-\|\nabla \sigma_h^{1}\|^{2})
	+\frac{A\tau^2}{2}(\|\nabla\delta_t \sigma_h^{N+1}\|^{2}-\|\nabla \delta_t \sigma_h^{1}\|^{2})\\
	&+\frac{A\tau^2}{2}\sum_{n=1}^{N}\|\nabla \delta_{tt}\sigma_h^{n+1}\|^{2}
	+\varepsilon^4 \tau\sum_{n=1}^{N} \|\nabla \frac{3\sigma_h^{n+1}+\sigma_h^{n-1}}{4}\|^2\\
	\leq & 
	\frac{9}{2}\left(C_0 +L^2\right)\tau\|\sigma_h^{N+1}\|_{-1}^2 
	+5\left(C_0 +L^2\right)\tau\sum_{n=1}^{N}\|\sigma_h^{n}\|_{-1}^2\\
	&+ C_3 \tau^4 + C_4 h^{2q+4},
	\end{split}
	\end{equation} 
	where
	\begin{equation}\label{eq:C3:0}
	\begin{split}
	C_3 =&\frac{60}{\varepsilon^4} \int_{0}^{T}  \|\Delta^{-1}\phi_{tt}\|_{-1}^{2}
	+ A^2\|\nabla \phi_t\|^{2}
	+ \varepsilon^2 \|\nabla \phi_{tt}\|^{2}{\rm d}t\\
	&+\frac{15}{\varepsilon^6}\int_{0}^{T}
	\left(31L^2+72L_2^2 C^2 + 48B^2 \varepsilon^2 \right)
	\|\partial_{tt}\phi(t)\|_{-1}^{2}{\rm d}t\\
	\lesssim & \varepsilon^{-\max\{\rho_{1}+4, \rho_{2}+6, \rho_{3}+2, \rho_{5}+8\}}:=\gamma_3(\varepsilon),
	\end{split}
	\end{equation} 
	and
	\begin{equation}\label{eq:C4:0}
	\begin{split}
	C_4 = &\frac{75}{2\varepsilon^6} \left(2L^2+9L^2_2\right)
	\tau \sum_{n=1}^{N+1}\| \phi^{n}\|^2_{H^{q+1}}
	+\frac{60 }{\varepsilon^4 }\tau \sum_{n=1}^{N+1}\| \mu^{n}\|^2_{H^{q+1}}\\
	&+\frac{15}{\varepsilon^6}\left(48B^2 \varepsilon^2 +27 L^2
	\right)\tau^{4}\int_{0}^{T}\| \phi_{tt}\|^2_{H^{q+1}}{\rm d}t
	+\frac{60 }{\varepsilon^4}\int_{0}^{T}\| \phi_t\|^2_{H^{q+1}}{\rm d}t\\
	\lesssim & \min\{\varepsilon^{-(\rho_6+4)},\varepsilon^{-(\rho_{7}+6)}, \varepsilon^{-(\rho_{8}+4)}, 
	\varepsilon^{-(\rho_4+6)}\tau^4\} :=\gamma_4(\varepsilon, \tau).
	\end{split}
	\end{equation} 	
	Choose $\tau \leq 2/{9(C_{0}+L^{2})}$, then we can get a
	finer error estimate by discrete Gronwall's inequality
	and the assumption of first step error (\ref{eq:AP:phisigma}):
	\begin{equation}\label{newcn:t14}
	\begin{split}
	&\max_{1\leq n\leq N}\left(
	\|\sigma_h^{n+1}\|_{-1}^{2}
	+2A\tau^2\|\nabla \sigma_h^{n+1}\|^{2}
	+\frac{1}{2}\|\delta_{t}\sigma_h^{n+1}\|^{2}_{-1}
	+\frac{A\tau^2}{2}\|\nabla \delta_t\sigma_h^{n+1}\|^{2}\right)\\
	&+\frac{A\tau^2}{2}\sum_{n=1}^{N}\|\nabla \delta_{tt}\sigma_h^{n+1}\|^{2}
	+\varepsilon^4 \tau\sum_{n=1}^{N}\|\nabla\frac{3\sigma_h^{n+1}+\sigma_h^{n-1}}{4}\|^{2}
	\\
	\lesssim &
	\mathrm{exp }(5(C_{0}+L^{2})T) \left(\gamma_3(\varepsilon) \tau^4 + \gamma_4(\varepsilon, \tau) h^{2q+4}\right) +\varepsilon^{-\tilde{\sigma}_1}\tau^4
	+ \varepsilon^{-\tilde{\sigma}_2}h^{2q+4}.\\
	\end{split}
	\end{equation}
	
	We prove this by induction. Assuming that
	the above estimate holds for all first $N$ time steps.
	Since $\tau \lesssim \varepsilon^{6}$,
	then the coarse estimate \eqref{newcn:cet3-1} leads to
	\begin{equation}\label{eq:error:8}
	\begin{split}
	&\|\sigma_h^{N+1}\|_{-1}^{2}
	+\frac{1}{4}\|\delta_t \sigma_h^{N+1}\|_{-1}^{2}
	+A\tau^2\|\nabla \sigma_h^{N+1}\|^{2}
	+\frac{A\tau^2}{2}\|\nabla \delta_t\sigma_h^{N+1}\|^{2}\\
	&+\frac{A\tau^2}{2}\|\nabla \delta_{tt}\sigma_h^{N+1}\|^{2}
	+\varepsilon \tau \|\nabla \frac{3\sigma_h^{N+1}+\sigma_h^{N-1}}{4}\|^2\\
	\leq & \|\sigma_h^{N}\|_{-1}^{2}　
	+\frac{1}{4}\|\delta_t \sigma_h^{N}\|_{-1}^{2}
	+A\tau^2\|\nabla \sigma_h^{N}\|^{2}
	+\frac{A\tau^2}{2}\|\nabla \delta_t\sigma_h^{N}\|^{2}\\
	&+\frac{99 L^2}{2\varepsilon^3}  \tau\|\sigma_h^{N}\|^2_{-1}
	+\frac{11 L^2}{2\varepsilon^3} \tau \|\sigma_h^{N-1}\|^2_{-1}
	+\gamma_1(\varepsilon) \tau^4 + \gamma_2(\varepsilon, \tau) h^{2q+4}\\
	\lesssim&	\mathrm{exp }(5(C_{0}+L^{2})T) \left(\gamma_3(\varepsilon) \tau^4 
	+ \gamma_4(\varepsilon, \tau) h^{2q+4}\right)
	+\varepsilon^{-\tilde{\sigma}_1}\tau^4+ \varepsilon^{-\tilde{\sigma}_2}h^{2q+4}\\
	&+\gamma_1(\varepsilon) \tau^4 + \gamma_2(\varepsilon, \tau) h^{2q+4},
	\end{split}
	\end{equation} 	
	To obtain $G^{N+1}\leq \varepsilon^{4}/4$, by using \eqref{eq:error:8}, 
	$\varepsilon^{-\tilde{\sigma}_1} \leq \gamma_1(\varepsilon) \leq \gamma_3(\varepsilon)$
	and $\varepsilon^{-\tilde{\sigma}_2} \leq \gamma_2(\varepsilon,\tau) \leq \gamma_4(\varepsilon,\tau)$,
	we easily get
	\begin{equation}\label{eq:error:9}
	\begin{split}
	G^{N+1}=&\frac{3L_{2}}{2\varepsilon}K \|\nabla \frac{3\sigma_h^{N+1}+\sigma_h^{N}}{4}\|^{\frac{d-2}{4}}
	\|\frac{3\sigma_h^{N+1}+\sigma_h^{N}}{4}\|_{-1}^{\frac{6-d}{4}}\\
	\lesssim&\frac{3L_{2}}{2\varepsilon}K \left(
	\gamma_3(\varepsilon)\varepsilon^{-1}\tau^{3}\right)^{\frac{d-2}{8}} 
	\left(\gamma_3(\varepsilon)\tau^{4}\right)^{\frac{6-d}{8}}
	\leq \dfrac{\varepsilon^{4}}{4},
	\end{split}
	\end{equation}
	and
	\begin{equation}\label{eq:error:9-1}
	\begin{split}
	G^{N+1}=&\frac{3L_{2}}{2\varepsilon}K \|\nabla \frac{3\sigma_h^{N+1}+\sigma_h^{N}}{4}\|^{\frac{d-2}{4}}
	\|\frac{3\sigma_h^{N+1}+\sigma_h^{N}}{4}\|_{-1}^{\frac{6-d}{4}}\\
	\lesssim&\frac{3L_{2}}{2\varepsilon}K \left(
	\gamma_4(\varepsilon, \tau)h^{2q+4}\varepsilon^{-1}\tau^{-1}\right)^{\frac{d-2}{8}} 
	\left(\gamma_4(\varepsilon, \tau)h^{2q+4}\right)^{\frac{6-d}{8}}
	\leq \dfrac{\varepsilon^{4}}{4}.
	\end{split}	
	\end{equation}		
	Solving \eqref{eq:error:9}, we get the condition for time step:
	\begin{equation}\label{eq:error:10}
	\begin{split}
	\tau
	\lesssim \varepsilon^{\frac{38+d}{18-d}}\gamma_3(\varepsilon)^{-\frac{4}{18-d}}.\\
	\end{split}
	\end{equation}
	Solving \eqref{eq:error:9-1}, we get the condition for spatial ratio:
	\begin{equation}\label{eq:error:10-1}
	\begin{split}
	h^{q+2}
	\lesssim \varepsilon^{5+\frac{d-2}{8}}\tau^{\frac{d-2}{8}}
	\gamma_4(\varepsilon, \tau)^{-\frac{1}{2}},
	\end{split}	
	\end{equation}	
	by using the definition of $\gamma_4(\varepsilon, \tau)$ in \eqref{eq:C4:0},
	and submitting	
	$\tau\lesssim \min\{\varepsilon^6, \varepsilon^{\frac{38+d}{18-d}}\gamma_3(\varepsilon)^{-\frac{4}{18-d}}\}$ into
	\eqref{eq:error:10-1}, then we get
	\begin{equation}\label{eq:error:10-2}
	\begin{split}
	h
	\lesssim \min\{ 
	\varepsilon^{\frac{7d-26}{8(q+2)}}\gamma_4(\varepsilon,\tau)^{-\frac{1}{2(q+2)}},
	\varepsilon^{\frac{2d+76}{(18-d)(q+2)}}
	\gamma_3(\varepsilon)^{-\frac{d-2}{2(18-d)}}
	\gamma_4(\varepsilon,\tau)^{-\frac{1}{2(q+2)}}\},
	\end{split}	
	\end{equation}			
	From \eqref{eq:Rh3}-\eqref{eq:Rh4}, we easily get  
	\begin{eqnarray}
	\label{eq:error:11:1}
	\|\rho^{n+1}\|_{-1} & \leq Ch^{q+2}\|\phi^{n+1}\|_{H^{q+1}},\\
	\label{eq:error:11:2}
	\|\nabla \rho^{n+1}\| &\leq Ch^{q}\|\phi^{n+1}\|_{H^{q+1}}.
	\end{eqnarray}
	Estimate \eqref{eq:conclusion1}-\eqref{eq:conclusion2} follows from the application of the triangle inequality for \eqref{eq:error:11:1}-\eqref{eq:error:11:2} and \eqref{newcn:t14}.	
	We complete the proof.
\end{proof}  

\begin{remark}\label{newcn:rem:conv1}
	Note that the spectral estimate \eqref{newcn:eigen} is essential
	to the proof. Compared to Crank-Nicolson
	discretization, the diffusive Crank-Nicolson
	discretization has an extra numerical diffusion 
	$\varepsilon\delta_{tt}\Delta\phi^{n+1}/4$, 
	it is easier to bound the error growth. Here, we do not need
	$B>L/2\varepsilon$ to get the convergence, while
	in SL-CN scheme, there is a necessary requirement
	\cite{WangYu2018b}.
\end{remark}
\begin{remark}\label{newcn:rem:conv2}
	We present the error estimate of the fully discrete SLD-CN scheme.
	It needs stronger regularity described in Lemma \ref{lm:1}.
\end{remark}

\section{Implementation and numerical results}
\label{4}
We will give several examples to illustrate the performance of our schemes.

To test the numerical scheme, we solve \eqref{eq:CH} in
tensor product 2-dimensional domain $\Omega=[-1,1]\times[-1,1]$.
\eqref{eq:CN:full:1}-\eqref{eq:CN:full:2} is a linear system 
with constant coefficients for
$(\phi^{n+1}_h, \mu^{n+\frac12}_h)$, which can be efficiently solved.
We use a spectral transform with double quadrature points to
reduce the aliasing error and efficiently evaluate the
integration $(f(\frac32\phi^n_h-\frac12\phi^{n-1}_h),\varphi_h)$ 
in equation\eqref{eq:CN:full:2}.

Given $\phi^0_h$, to start the second order scheme, we use 
following first order stabilized scheme to generate $\phi^1_h\in V_M$
\begin{align}
\frac{1}{s}(\phi^{1}_h-\phi^0_h, \psi_h)=-(\nabla \omega^{1}_h, \nabla  \psi_h),
&\qquad \forall \psi_h \in V_M,\\
(\omega^{1}_h, \varphi_h)= \varepsilon(\nabla \phi^{1}_h, \varphi_h)+\frac{1}{\varepsilon}
(f(\phi^1_h),\varphi_h)+S\delta_t(\phi^{1}_h,\varphi_h),
&\qquad \forall \varphi_h \in V_M,
\end{align}
where $S = 1/\varepsilon$ is a stabilization constant.
Note that the BDF1 scheme generates a second-order
accurate solution at the first time step.

We take $\varepsilon = 0.05$ and $M = 127$(except Example \ref{ex:2}) and use two different
initial values to test the stability and accuracy of the SLD-CN scheme:
\begin{enumerate}
	\item[(i)] $\phi_0$: $\{\phi_0(x_i,y_j) \}\in\, {\bf{R}}^{2M\times2M}$
	with $x_i, y_j$ are tensor product Legendre-Gauss
	quadrature points and $\phi_0(x_i,y_j)$ is a uniformly
	distributed random number between $-1$ and $1$;
	\item[(ii)] $\phi_1$: the solution of the Cahn-Hilliard equation at
	$t=64\varepsilon^3$ which takes $\phi_0$ as its initial
	value.
\end{enumerate}

\subsection{Stability results}
Table \ref{tstab0:A},\ref{tstab0:B} show the required
minimum values of $A$ (resp. $B$) with different $\gamma$,
$B$ (resp. $A$) and $\tau$ values for stably solving (not
blow up in 4096 time steps) the Cahn-Hilliard equation
\eqref{eq:CH} with initial value $\phi_0$.  The results for
the initial value $\phi_1$ are similar. From the two tables,
we observe that for smaller $\tau$ values, the SLD-CN 
scheme is more stable than the SL-CN scheme proposed in 
\cite{WangYu2018b, wang_two_2018}), while both of
them are stable with $A=0$ and $B=0$ when $\gamma$ and $\tau$ 
are small enough. Due to the fact that the SLD-CN scheme 
has larger diffusion term than SL-CN scheme, SLD-CN schemes need
relatively smaller $A$ and $B$ than SL-CN scheme.
\begin{table}[htbp]
	\begin{tabular}{|c|c|c|c|c|c|c|c|c|}
		\hline
		\multirow{3}{*}{$\tau$ }
		&  \multicolumn{4}{|c|}{SL-CN}
		&  \multicolumn{4}{|c|}{SLD-CN}
		\\ \cline{2-9}&  \multicolumn{2}{|c|}{$\gamma=0.0025$ }
		&  \multicolumn{2}{|c|}{$\gamma=1$ }
		&  \multicolumn{2}{|c|}{$\gamma=0.0025$ }
		&  \multicolumn{2}{|c|}{$\gamma=1$ }
		\\ \cline{2-9} & $B=0$ & $B=10$ & $B=0$ & $B=10 $ &  $B=0$ & $B=10$ & $B=0$ & $B=10 $\\ \hline
		10  & 0.32 & 0.04  & 1   & 1  &  0.32 & 0.04 & 1  & 1\\ \hline
		1   & 0.32 & 0.08  & 8   & 4  &  0.32 & 0.08 & 8  & 2\\ \hline
		0.1 & 0.64 & 0.16  & 64  & 16 &  0.32 & 0    & 64 & 16\\ \hline
		0.01& 1.28 & 0.32  & 128 & 32 & 0     & 0    & 128& 16\\ \hline
		0.001&1.28 & 0.16  & 256 & 32 & 0     & 0    & 256& 8\\ \hline
		0.0001&0   & 0     & 256 &128 & 0     & 0    & 64 & 0\\ \hline
		1E-05 &0   & 0     & 512 &128 & 0     & 0    & 0  & 0  \\ \hline
		1E-06 &0   & 0     & 128 & 0  & 0     & 0    & 0  & 0\\ \hline
	\end{tabular} 
	\centering 
	\caption{The minimum values of $A$(only values $\{0, 2^i,i=0,\ldots,11\}\times\gamma$ are tested for $A$) 
		to make schemes SL-CN and SLD-CN stable when $\gamma$, $B$ and $\tau$ are taking different values.
	}
	\label{tstab0:A}
\end{table}
\begin{table}[htbp]
	\begin{tabular}{|c|c|c|c|c|c|c|c|c|}
		\hline
		\multirow{3}{*}{$\tau$ }
		&  \multicolumn{4}{|c|}{SL-CN}
		&  \multicolumn{4}{|c|}{SLD-CN}
		\\ \cline{2-9}&  \multicolumn{2}{|c|}{$\gamma=0.0025$ }
		&  \multicolumn{2}{|c|}{$\gamma=1$ }
		&  \multicolumn{2}{|c|}{$\gamma=0.0025$ }
		&  \multicolumn{2}{|c|}{$\gamma=1$ }
		\\ \cline{2-9} & $A=0$ & $A=1$ & $A=0$ & $A=4$ & $A=0$ & $A=1$ & $A=0$ & $A=4$\\ \hline
		10  & 32 & 0  & 64   & 0  & 32 & 0 & 32 & 0\\ \hline
		1   & 32 & 0  & 64   & 8  & 16 & 0 & 32 & 8\\ \hline
		0.1 & 32 & 0  & 64   & 16 & 8  & 0 & 32 & 16\\ \hline
		0.01& 32 & 0  & 64   & 16 & 0  & 0 & 32 & 16\\ \hline
		0.001&16 & 0  & 32   & 16 & 0  & 0 & 16 & 16\\ \hline
		0.0001&0 & 0  & 32   & 32 & 0  & 0 & 2  & 2\\ \hline
		1E-05 &0 & 0  & 32   & 32 & 0  & 0 & 0  & 0 \\ \hline
		1E-06 &0 & 0  & 8    & 8  & 0  & 0 & 0  & 0\\ \hline
	\end{tabular} 
	\centering 
	\caption{The minimum values of $B$(only values $\{0, 2^i,i=0,\ldots,9\}$ are tested for $B$) 
		to make schemes SL-CN and SLD-CN stable when $\gamma$, $A$ and $\tau$ are taking different values. 
	}
	\label{tstab0:B}
\end{table}

\subsection{Accuracy results}
\begin{ex}\label{ex:1}
We take initial value $\phi_1$ to test the temporal accuracy of the
two schemes: SLD-CN scheme and SL-BDF2 scheme.  
The Cahn-Hilliard equation with
$\gamma=0.0025$ is solved from $t=0$ to $T=12.8$.
We take stability constants  $A=0.25$ and $B=5$ in the both schemes.
 To calculate the numerical error, we use the numerical result
generated by the SL-BDF2 scheme using $\tau=10^{-3}$ as a reference of exact
solution. We see that the SLD-CN scheme is second order accurate in
$ L^2$ norm by the time step $\tau=0.01, 0.02, 0.04, 0.08, 0.16$. 
From Figure \ref{figerror}, 
we find the error of the SLD-CN scheme is obviously 
smaller than the error of the SL-BDF2 scheme.
\begin{figure}[!h]
	\centering
	\includegraphics[trim=40mm 90mm 40mm 90mm, clip, width=0.60\textwidth]{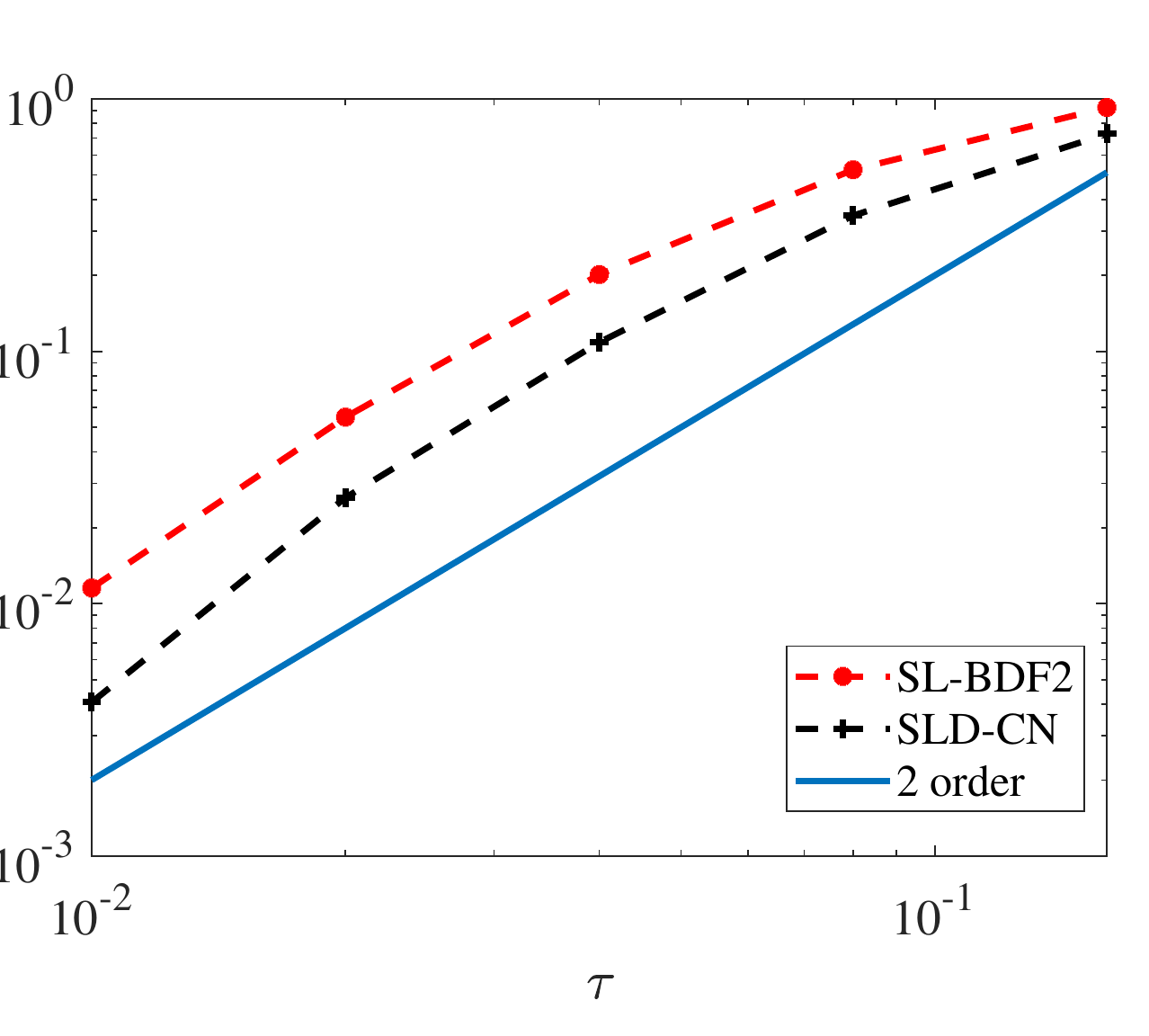}
	\caption{Temporal convergence of SLD-CN scheme and SL-BDF2 scheme.}
	\label{figerror}
\end{figure}
\end{ex}

\begin{ex}\label{ex:2}
	We take initial value $\phi_1$ to test the spatial accuracy of the
    SLD-CN scheme.  The Cahn-Hilliard equation with
	$\gamma=0.0025$ are solved from $t=0$ to $T=1$ with time step size $\tau=10^{-5}$.
	We take stability constants  $A=0.025$ and $B=0.5$.
	To calculate the numerical error, we use the numerical result
	generated using $M=255$ as a reference of exact
	solution. Figure \ref{fighorder} presents the semilogy plot of errors in $H^{-1}$ norm, $ L^2$ norm and $H^1$ norm against the polynomial degree $M= 17, 33, 49, 65, 87$ for the SLD-CN scheme.
    We observe that the SLD-CN scheme in $H^{1}$ norm, $ L^2$ norm and $H^{-1}$ norm are all spectral convergent. 
    The convergence rate in $H^{-1}$ norm is higher than it in $H^{1}$ norm, which is as expected in Theorem \ref{newcn:error}.
\begin{figure}[!h]
	\centering
	\includegraphics[trim=40mm 90mm 40mm 90mm, clip, width=0.60\textwidth]{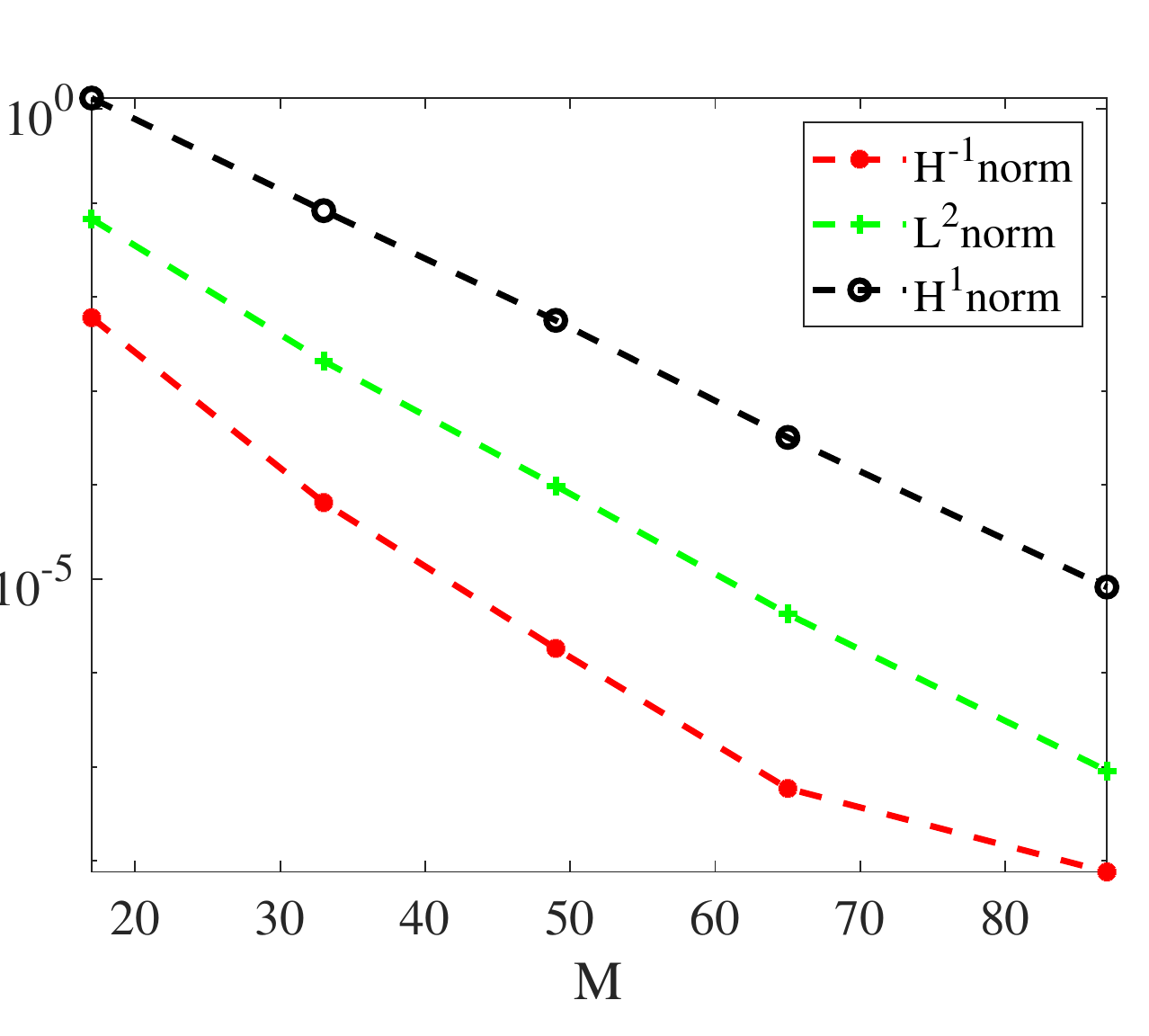}
	\caption{Spatial convergence 
		of SLD-CN scheme.}
	\label{fighorder}
\end{figure}
\end{ex}

\subsection{Adaptive time stepping}
Several adaptive time stepping strategies have been
implemented to Cahn-Hilliard equation.  
We propose an adaptive time-stepping strategy in which the time step 
is defined by the moving speed of the interface for SLD-CN scheme. 
The method is presented in Algorithm \ref{al1}.
We update the time step using the equation
$A_{dp}(e_{n+1},\tau_{n+1})$, which is proposed 
by Gomez and Hughes \cite{gomez_provably_2011}.
Our default values for the safety coefficient $\rho$ and the tolerance $tol$ are given as $\rho=0.9$, $tol=10^{-3}$. The minimum and maximum time steps are taken as $\tau_{min}=10^{-6}$ and $\tau_{max}=0.01$, respectively. $e_{n+1}$ is the approximation of the relative ratio between the interface velocity and the interface thickness at the $(n+1)$th time level. The initial time step is taken as $10^{-3}$. 
\begin{al}\label{al1}
	Time step adaptive procedure:
	\begin{itemize}
		\item Step 1: Compute $\phi^{n+1}$ by SLD-CN scheme with $\tau_{n+1}$;\\
		\item Step 2: Calculate $E_{C}^{n+1}(\phi^{n+1}, \phi^{n}, B)$.
		\\
		\item Step 3: Calculate
		$e_{n+1}= 10\left( \frac{\| \phi^{n+1} - \phi^{n}\|}{\varepsilon E_{C}^{n+1}(\phi^{n+1}, \phi^{n}, B)} \right)^2$ and\\ 
		$A_{dp}(e_{n+1},\tau_{n+1})=\rho \left(\frac{tol}{e_{n+1}}\right)^{1/2}\tau_{n+1};$
		\item Step 4: if $e_{n+1}>tol$, then\\
		recalculate time step:
		$\tau_{n+1}\leftarrow 
		\max\{\tau_{min}, \min\{A_{dp}(e_{n+1},\tau_{n+1}), \tau_{max}\}\}$;\\
		goto Step 1;\\
		else\\ 
		update time step size 
		$\tau_{n+2}\leftarrow 
		\min\{A_{dp}(e_{n+1},\tau_{n+1}), \tau_{max}\}$;\\
		continue to next time step.
	\end{itemize}
\end{al}

We solve the Cahn-Hilliard equation with initial value $\phi_0$ and $M=63$ until $T=30$. We take $\gamma=0.0025$, $A=1$, $B=0.25$.
We present numerical results of phase evolutions using large time steps, 
adaptive time steps, and small time steps for Cahn-Hilliard equation in Figure \ref{fig1}.
We take a uniform large time step $\tau=0.01$ and a uniform 
small time step $\tau=10^{-5}$ for comparison.
It is noted that the solutions by adaptive time steps in the second row are consistent with the solutions by uniform small time step in the third row. On the other hand, the uniform large time step solutions in the first row are far different from the adaptive time steps solutions. Figure \ref{fig2} presents the adaptive time steps and discrete energy accordingly with the time.
The time steps almost grow from $\tau=10^{-6}$ to $\tau=10^{-2}$.
The last time step decreases because it is only 0.0034 from the second last step to the end time.
Also, the discrete energy curve of adaptive time steps coincides with it of uniform small time steps $\tau=10^{-5}$, and does not coincide with that of uniform large time steps $\tau=0.01$. 
It indicates that the adaptive time stepping for the SLD-CN scheme is very effective.

\begin{figure}[!h]
	\centering
	\subfigure[Adaptive time step size.]{
	\includegraphics[trim=40mm 90mm 40mm 90mm, clip, width=0.45\textwidth]{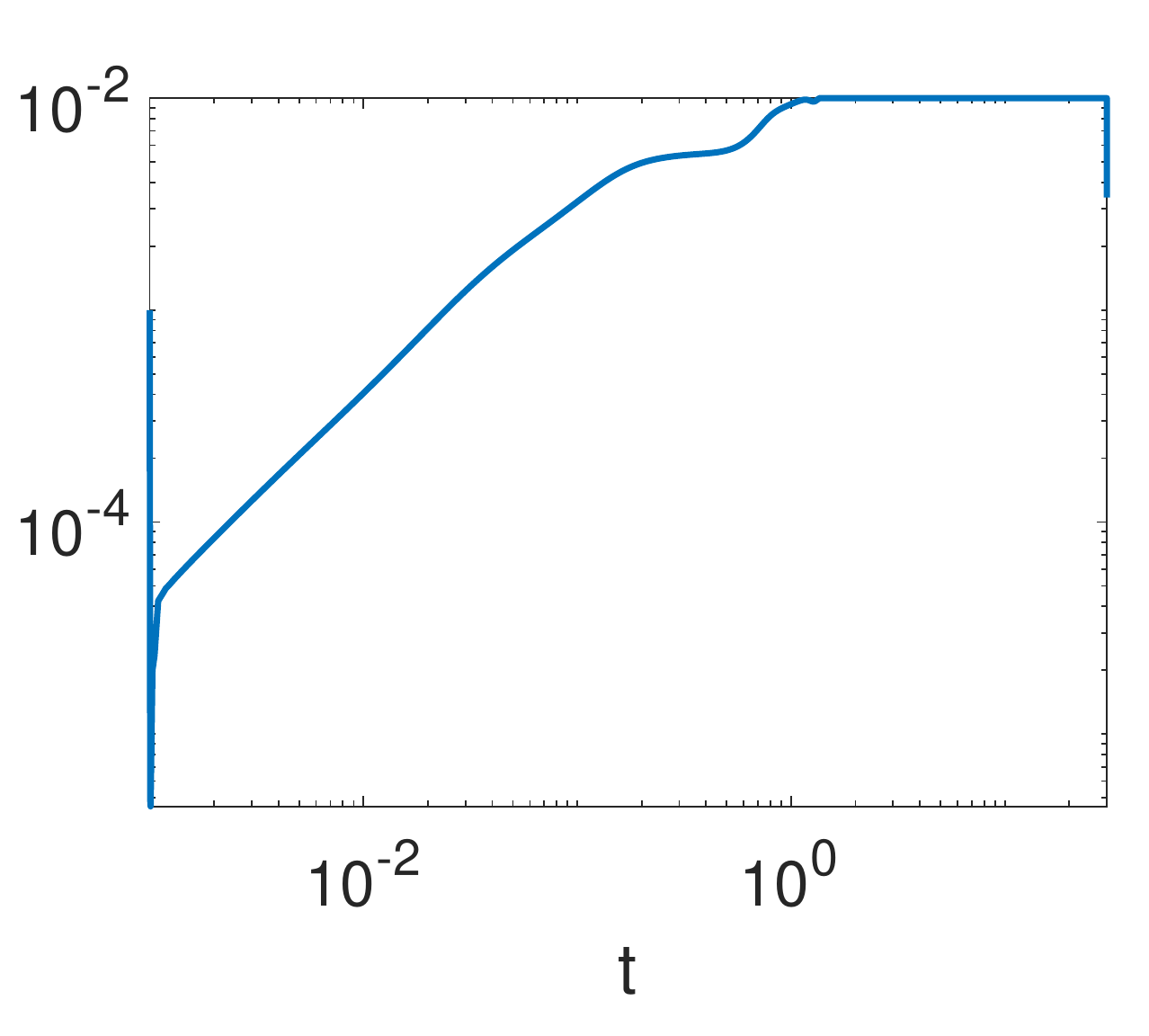}
    }
	\subfigure[Discrete energy.]{
	\includegraphics[trim=40mm 90mm 40mm 90mm, clip, width=0.45\textwidth]{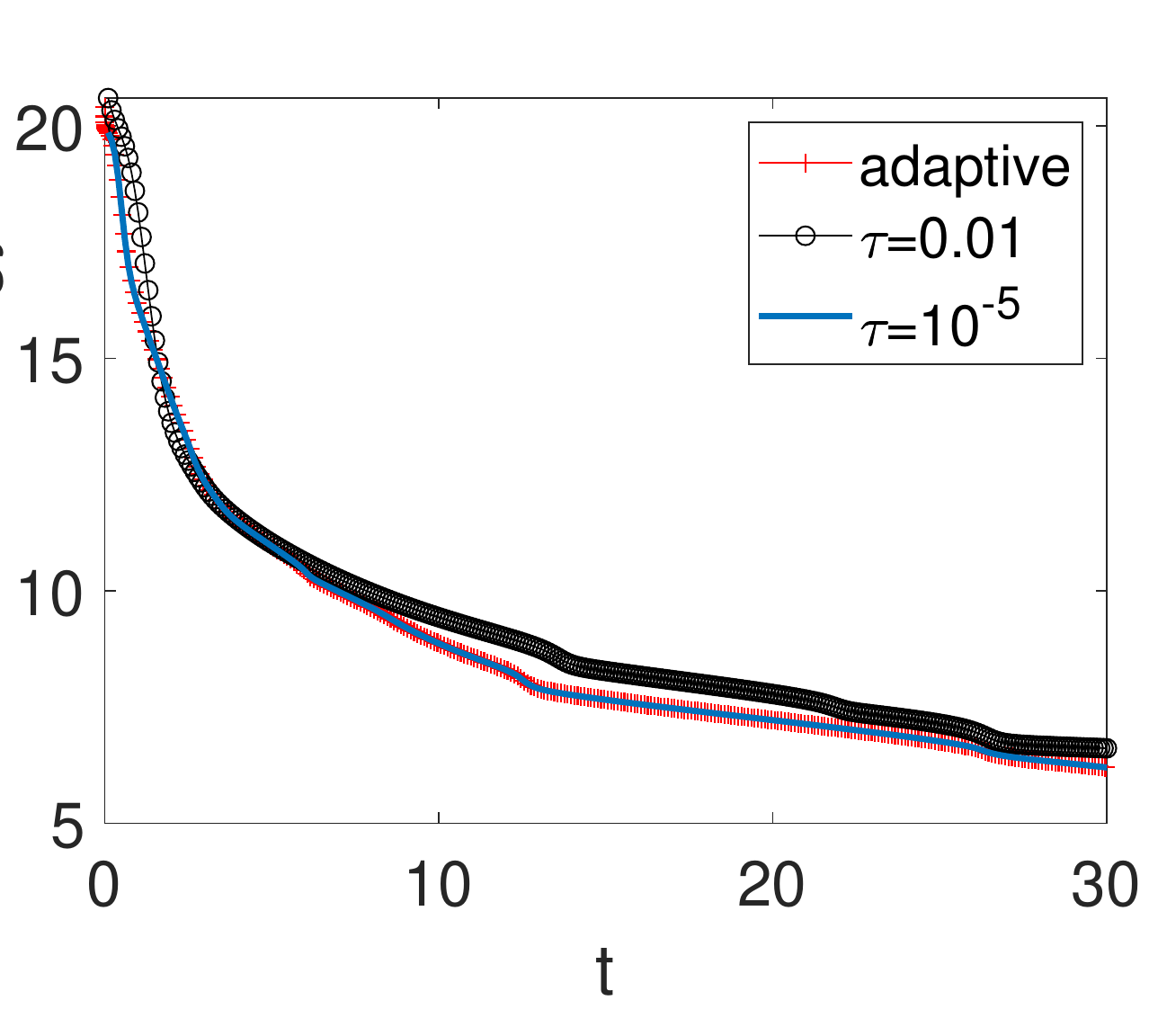}
    }
	\caption{Adaptive time steps and discrete energy against time until $T=30$. }
	\label{fig2}
\end{figure}

\begin{figure}[!h]
	\centering
	\subfigure[$\tau=0.01$]{
		\begin{minipage}{1\textwidth}
			\includegraphics[trim=12mm 3mm 12mm 2mm, clip, width=0.32\textwidth]{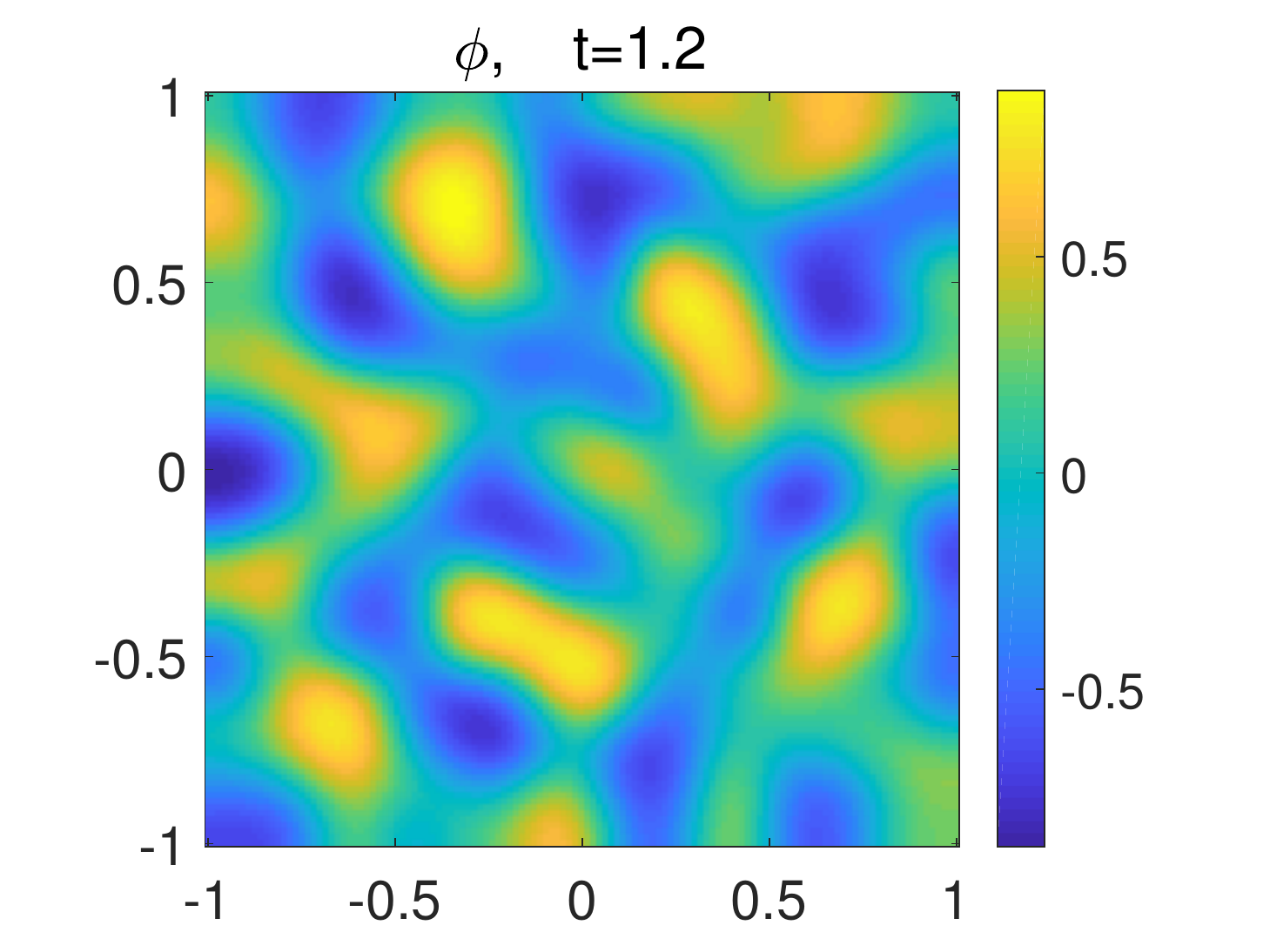} 
			\includegraphics[trim=12mm 3mm 12mm 2mm, clip, width=0.32\textwidth]{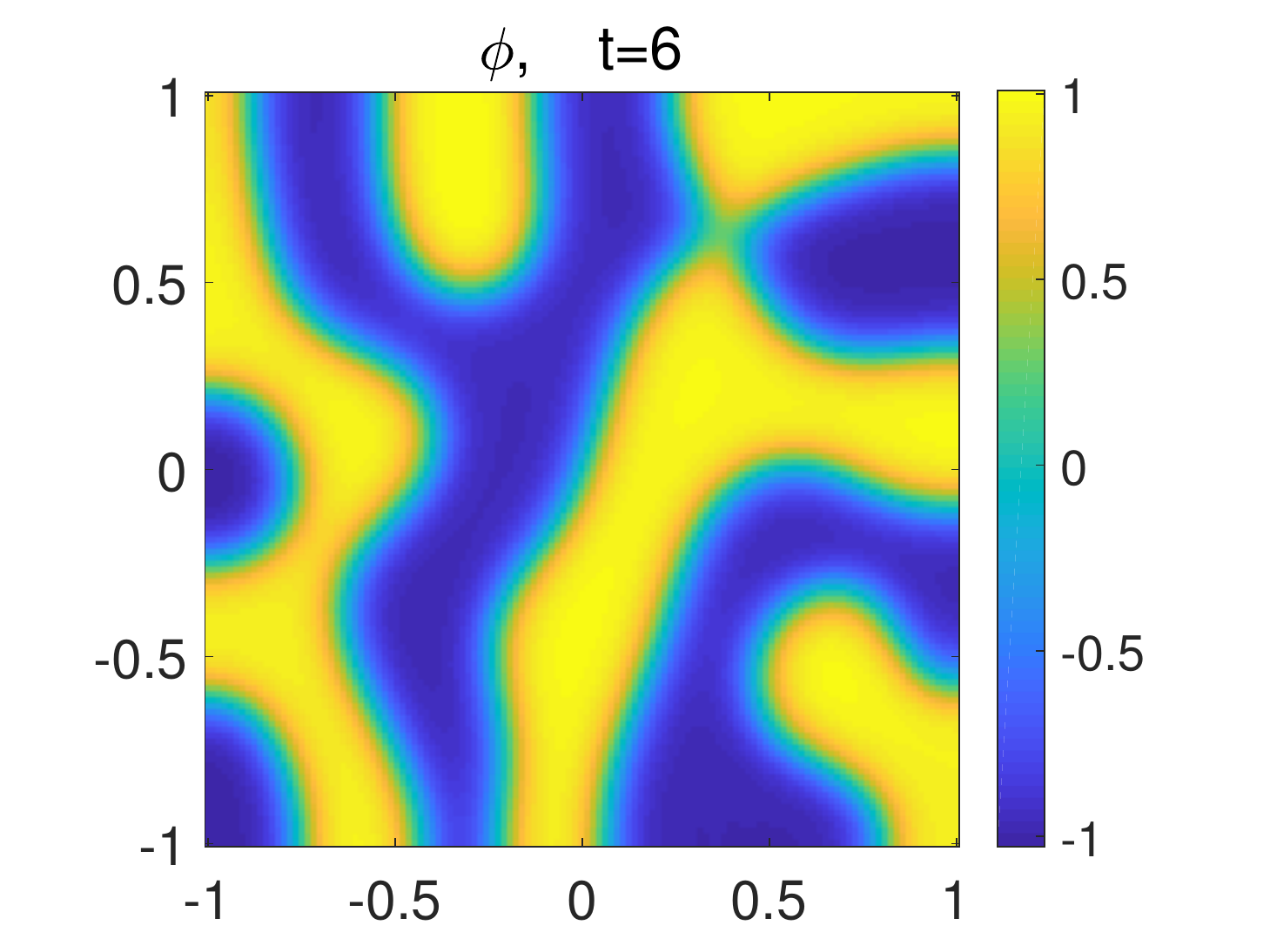}
			\includegraphics[trim=12mm 3mm 12mm 2mm, clip, width=0.32\textwidth]{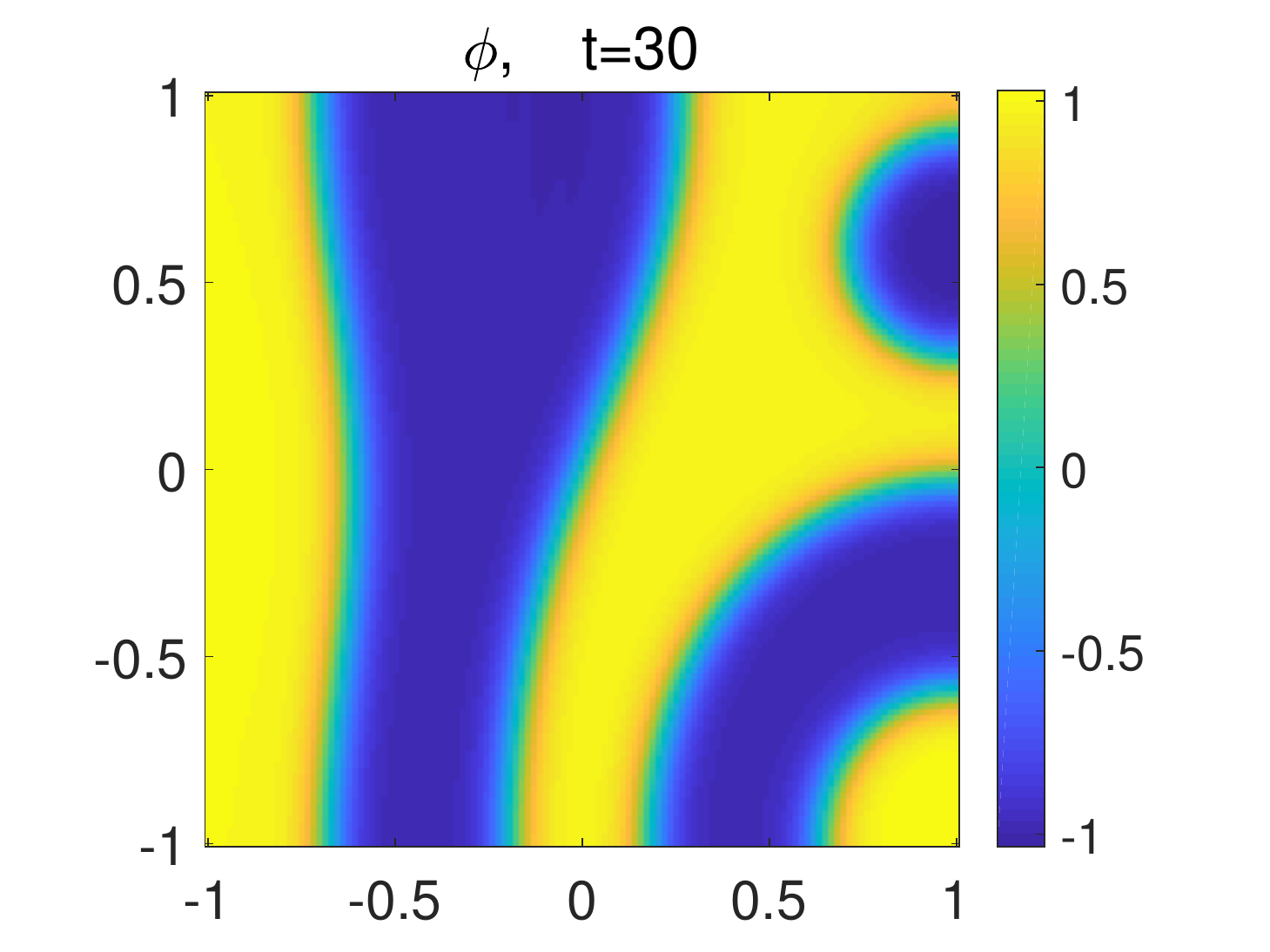}
		\end{minipage}
	}
	\subfigure[Adaptive time steps]{
		\begin{minipage}{1\textwidth}
			\includegraphics[trim=12mm 3mm 12mm 2mm, clip, width=0.32\textwidth]{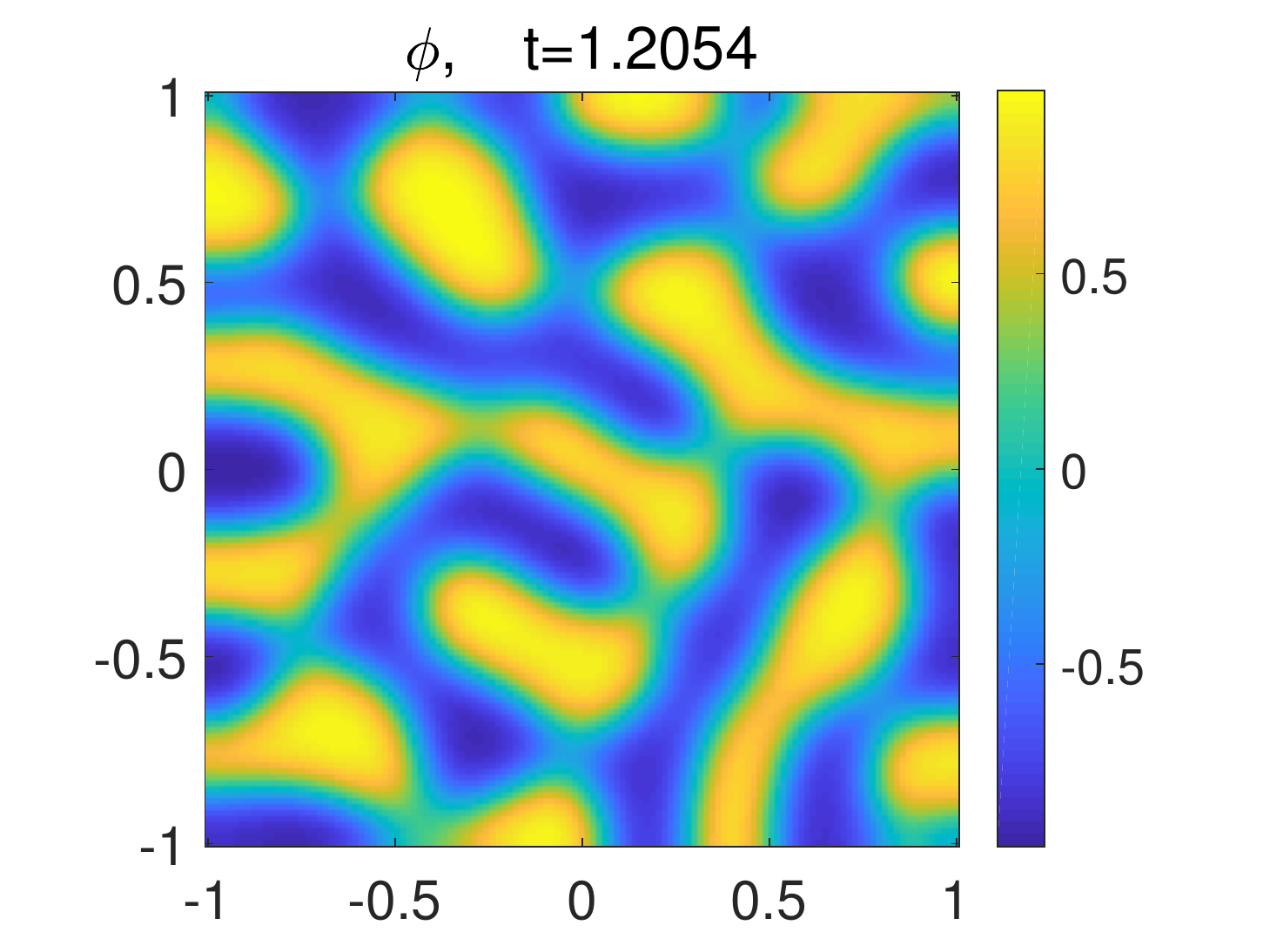}
			\includegraphics[trim=12mm 3mm 12mm 2mm, clip, width=0.32\textwidth]{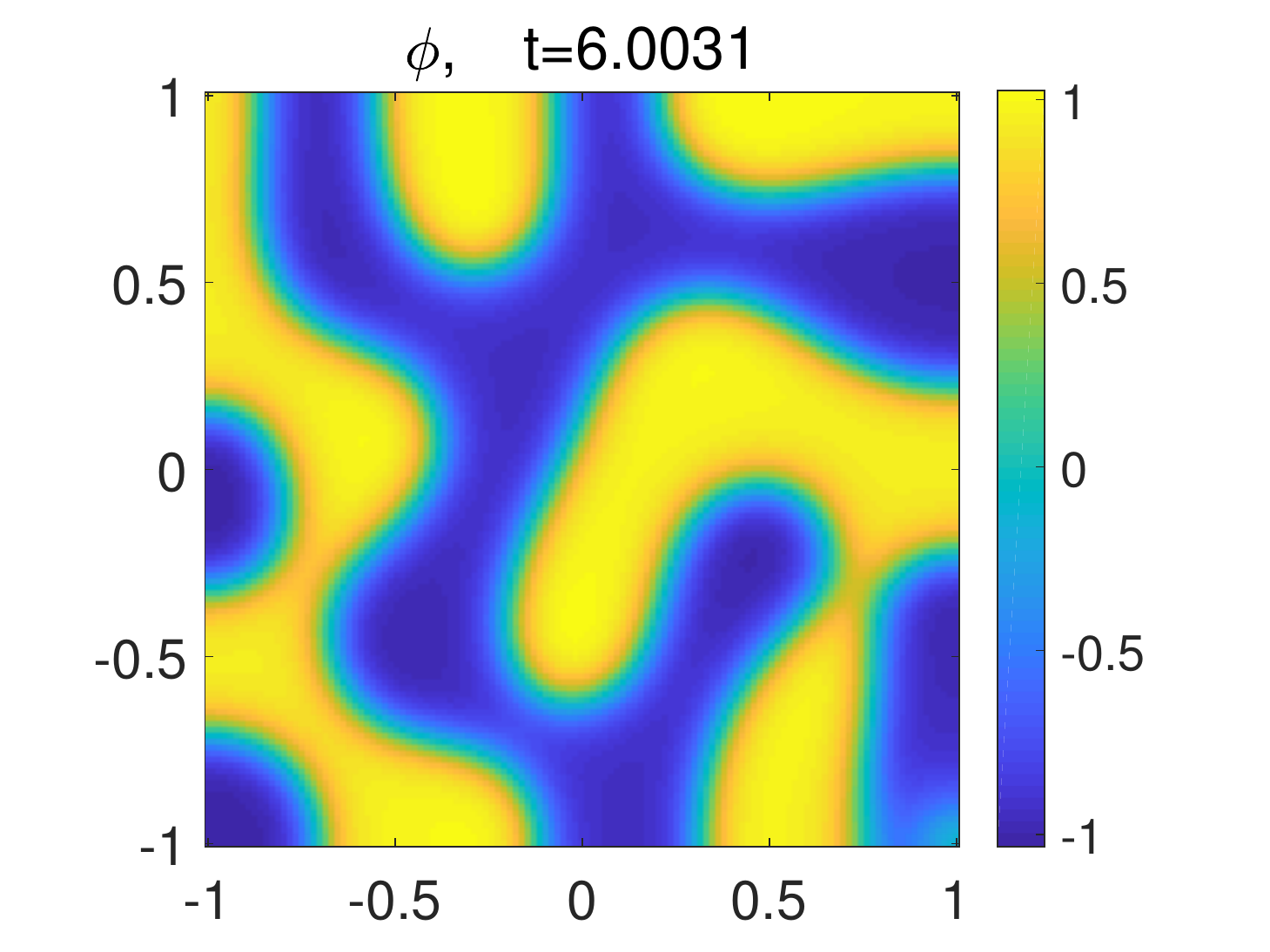}
			\includegraphics[trim=12mm 3mm 12mm 2mm, clip, width=0.32\textwidth]{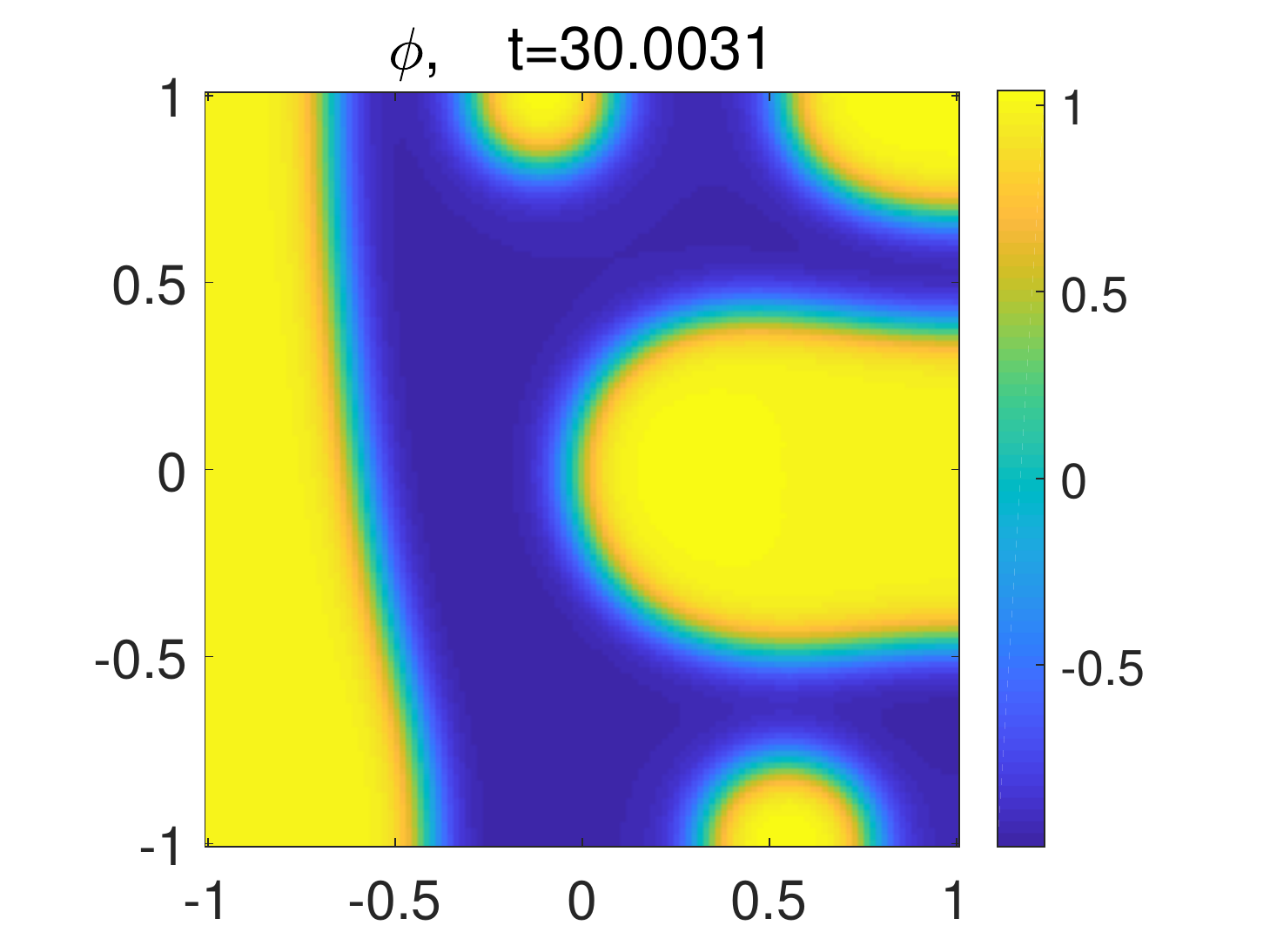}
		\end{minipage}
	}
	\subfigure[$\tau=0.00001$]{
		\begin{minipage}{1\textwidth}
			\includegraphics[trim=12mm 3mm 12mm 2mm, clip, width=0.32\textwidth]{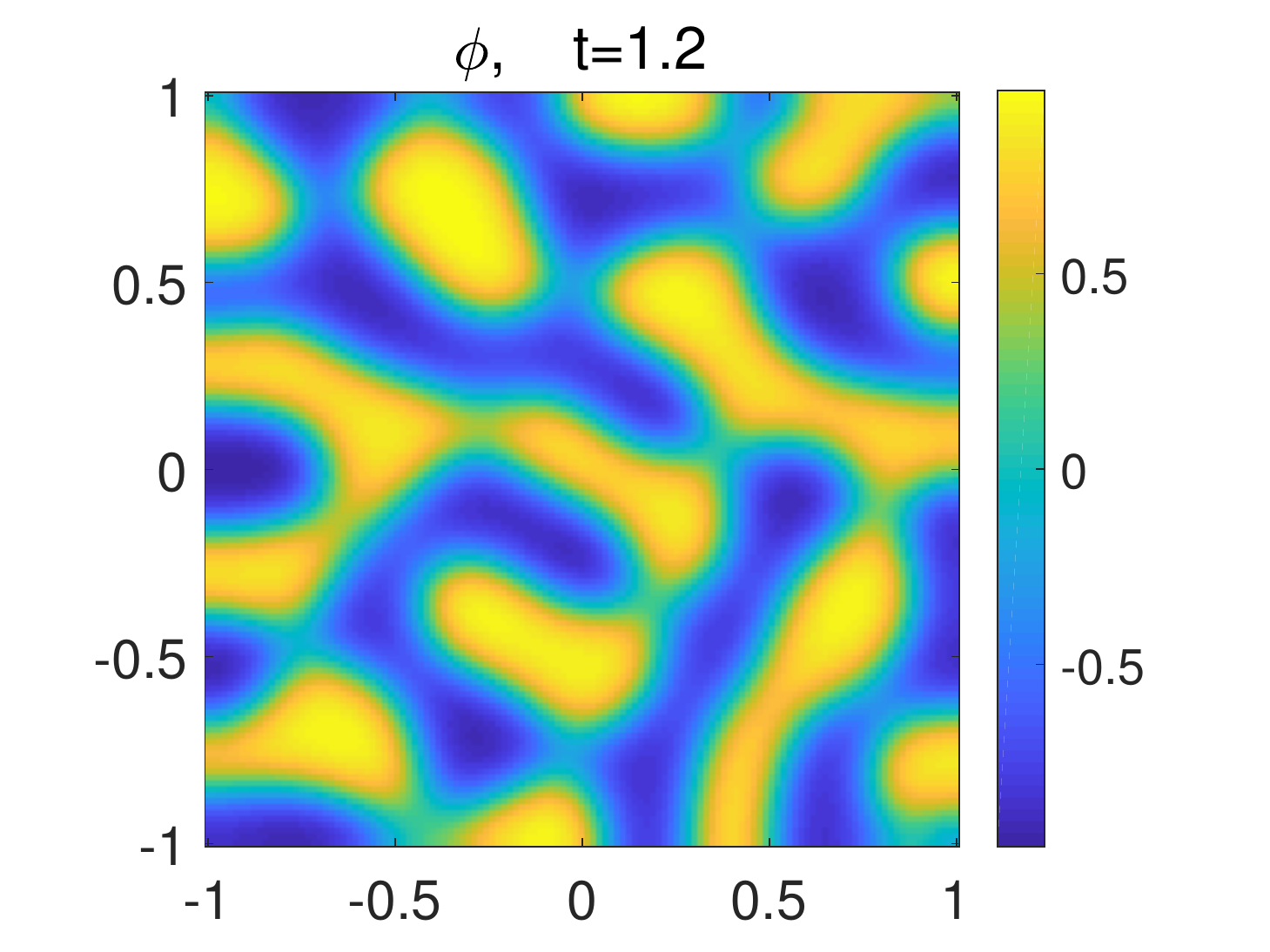}
			\includegraphics[trim=12mm 3mm 12mm 2mm, clip, width=0.32\textwidth]{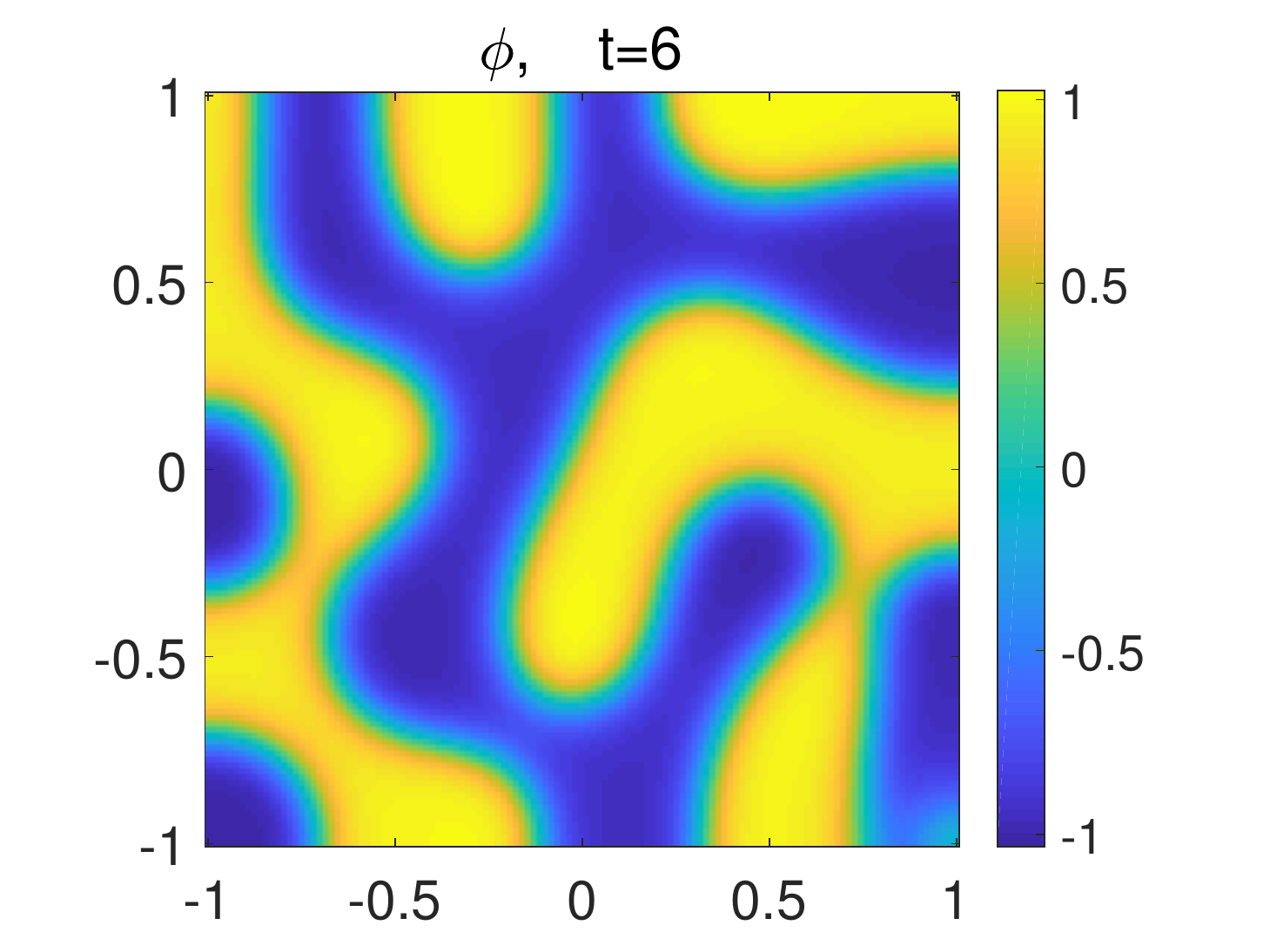}
			\includegraphics[trim=12mm 3mm 12mm 2mm, clip, width=0.32\textwidth]{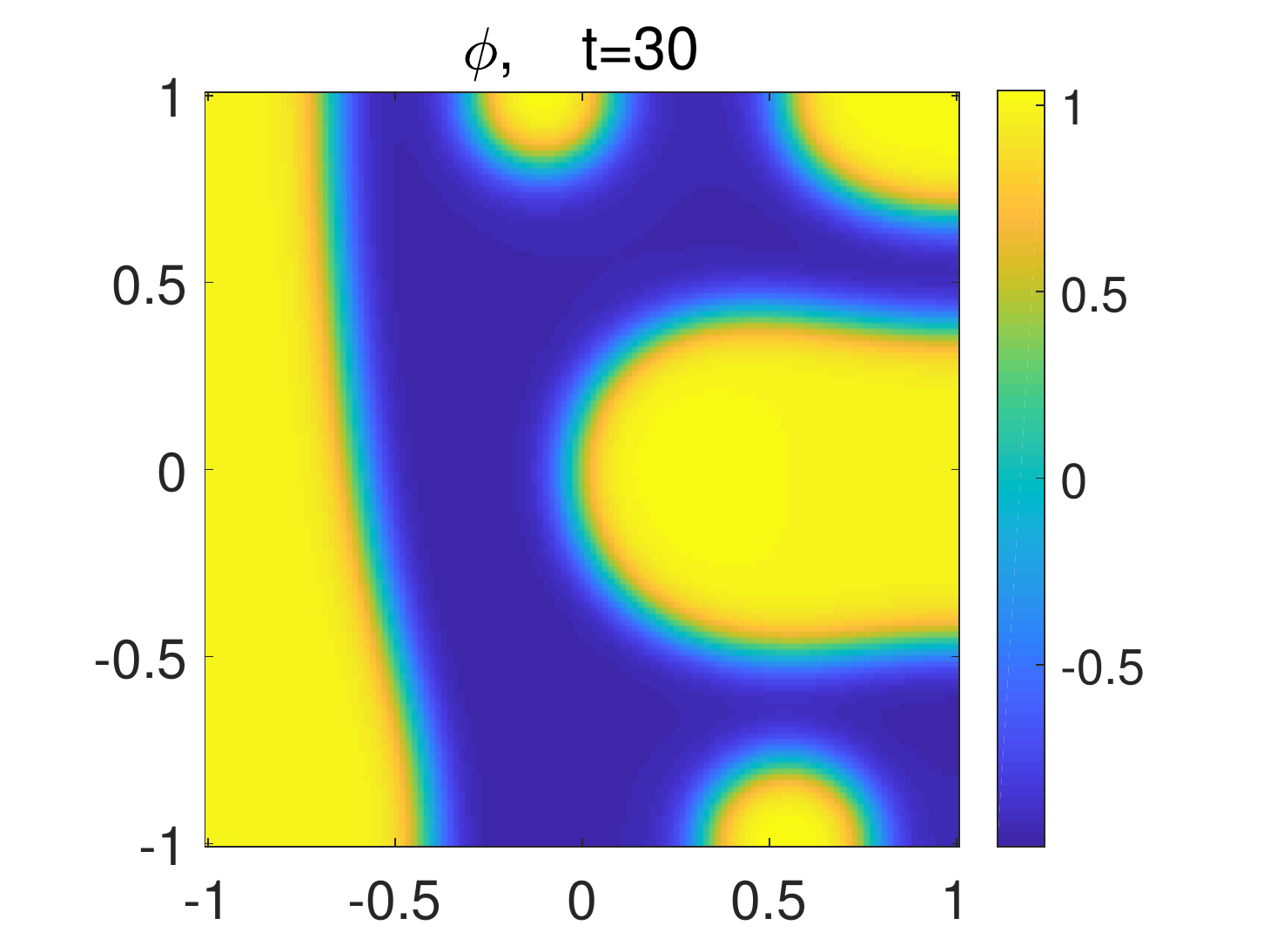}
		\end{minipage}
	}
	\caption{Numerical comparisons among large time steps, 
		adaptive time steps, and small
		time steps for Cahn-Hilliard equation. }
	\label{fig1}
\end{figure}

\section{Conclusions}
\label{5}
We propose the SLD-CN scheme by modifying the stabilized linear 
Crank-Nicolson scheme for the Cahn-Hilliard equation. In the scheme, the nonlinear bulk force 
is treated explicitly with two additional linear
stabilization terms: $-A\tau \Delta\delta_t\phi^{n+1}$ and $B\delta_{tt}\phi^{n+1}$. 
We give a rigorous optimal error analysis of the fully discrete 
SLD-CN scheme, which removes the condition $B>L/2\varepsilon$ 
for the error analysis of the SL-CN scheme.
This error analysis holds for the special case $A=0$ and/or 
$B=0$ as well. Numerical results verified the stability and 
accuracy of the proposed schemes. 

\section*{Acknowledgment}
This work was partially supported by NNSFC Grant 11771439 and 91852116 and China National Program on Key Basic Research Project 2015CB856003. 

\appendix
\section{Estimate of the constants in Assumption \ref{lm:1}}
\label{A1}
\begin{lemma}\label{lm:reg} 
	Suppose Assumption \ref{newcn:ap:1} (i)-(iii) and Assumption \ref{newcn:ap:2} are satisfied. We have following regularity results for the exact solution $\phi$ of (\ref{eq:CH}) with $\gamma = 1$.
\begin{enumerate}
	\item[(i)]
	$\int_{0}^{\infty}\|\phi_t\|_{-1}^{2}{\rm d}t + \esssup \limits_{t\in[0,\infty]}
	E_{\varepsilon}(\phi) \lesssim \varepsilon^{-\beta_1}$,
	and $\| \phi \|_{H^1}^2 \lesssim \varepsilon^{-(\sigma_1+1)}$;
		
	\item[(ii)]
	$   \esssup \limits_{t\in[0,\infty]}  \|\phi_{t}\|_{-1}^2+ \varepsilon\int_{0}^{\infty} \|\nabla \phi_{t}\|^2{\rm d} t
	\lesssim \varepsilon^{-\beta_2} $;
		
	\item[(iii)]
	$\esssup \limits_{t\in[0,\infty]}\|\phi_t\|^2 + \varepsilon \int_{0}^{\infty} \|\Delta \phi_t\|^2 {\rm d} t
	\lesssim \varepsilon^{- \beta_3}$;
	
	\item[(iv)]
	$ \int_{0}^{\infty} \|\phi_{tt}\|_{-1}^2 {\rm  d} t
	+  \esssup \limits_{t\in[0,\infty]} \varepsilon\|\nabla \phi_{t}\|^2
	\lesssim \varepsilon^{-\beta_4}$;
	
	\item[(v)]
	$\esssup \limits_{t\in[0,\infty]}  \| \Delta^{-1}\phi_{tt}\|^2 + \varepsilon\int_{0}^{\infty} \| \phi_{tt} \|^2 {\rm d} t
	\lesssim \varepsilon^{-\beta_5}$;
	
	\item[(vi)]
	$ \esssup \limits_{t\in[0,\infty]}  \| \phi_{tt}\|_{-1}^2 
	+ \varepsilon\int_{0}^{\infty} \| \nabla \phi_{tt} \|^2 {\rm d} t
	\lesssim  \varepsilon^{-\beta_6}$;
	
	\item[(vii)]
	$ \esssup \limits_{t\in[0,\infty]}  \| \phi_{tt}\|^2 
	+ \varepsilon\int_{0}^{\infty} \| \Delta \phi_{tt} \|^2 {\rm d} t
	\lesssim{} \varepsilon^{-\beta_7}$;	
	
	\item[(viii)]
	$\int_{0}^{\infty}  \|\Delta^{-1}\phi_{tt}\|_{-1}^2{\rm d} t
	+ \varepsilon \esssup \limits_{t\in[0,\infty]}   \|\phi_t\|_{-1}^2
	\lesssim{} \varepsilon^{-\beta_8}$;

	\item[(ix)]
	$\esssup \limits_{t\in[0,\infty]} \|\Delta^{-1}\phi_{t}\|^2
	+  \varepsilon \int_{0}^{\infty} \|\phi_t\|^2 {\rm d} t\lesssim{} \varepsilon^{-\beta_9}$;
	
	\item[(x)]
	$\int_{0}^{\infty}  \|\phi_{t}\|_{H^{q+1}}^2{\rm d} t
	\lesssim{} \varepsilon^{-\beta_{10}}$, when $q=1$;
	
	\item[(xi)]
	$\int_{0}^{\infty}  \|\phi_{tt}\|_{H^{q+1}}^2{\rm d} t
	\lesssim{} \varepsilon^{-\beta_{11}}$, when $q=1$;
	
	\item[(xii)]
	$\int_{0}^{T} \|\phi\|^2_{H^2}{\rm d} t
	\lesssim \varepsilon^{-(\sigma_1+3)}$;
	
	\item[(xiii)]
	$\int_{0}^{T}\|\mu\|^2_{H^2}{\rm d} t
	\lesssim  \varepsilon^{-\beta_{12}}$.
\end{enumerate}
where $\beta_1=\sigma_1$ and
\begin{equation*}
\begin{split}
\beta_2&=\max\{\sigma_1+3, \sigma_2\}, \\
\beta_3&=\max\{ (\sigma_1+1)(p-2)+\beta_2+4, \sigma_3\}, \\
\beta_4&=\max\{ \beta_2+2+\tfrac12\beta_3+\tfrac12(\sigma_1+1)(p-3)^+, \sigma_4\}, \\
\beta_5&=\max\{ \beta_2+\beta_4+1+(\sigma_1+1)(p-3)^+, (\sigma_1+1)(p-2)+\beta_4+3, \sigma_5\}, \\
\beta_6&=\max\{ (\sigma_1+1)(p-3)^++\beta_2+\beta_4+5, \beta_5+2, \sigma_6\},\\
\beta_7&=\max\{ (\sigma_1+1)(p-3)^++\beta_2+\beta_4+5, (\sigma_1+1)(p-2)+\beta_6+4, \sigma_7\},\\
\beta_8&=\max\{ (\sigma_1+1)(p-2)+\beta_1+2, (\sigma_1+1)(p-2)+\beta_2+3, \sigma_2+1\},\\
\beta_9&=(\sigma_1+1)(p-2)+\beta_1 +3,\\
\beta_{10}&=\max\{ \beta_2+1, \beta_3+1, \beta_9+1\},\\
\beta_{11}&=\max\{\beta_5+1, \beta_6+1, \beta_7+1\},\\
\beta_{12}&=\max\{\sigma_{1}, (\sigma_{1}+1)(p-1)+3\}.
\end{split}
\end{equation*}	
\end{lemma}

\begin{proof}
	We first write down some inequalities that will be frequently used. 
	The first one is the Holder's inequality
	\begin{equation}\label{eq:Holder3}
	\| u v w\|_{L^s} \le \|u\|_{L^p} \|v\|_{L^q} \| w\|_{L^r}, \quad\forall\ p,q,r\in(0,\infty],\ \frac{1}{s} = \frac{1}{p} + \frac{1}{q} + \frac{1}{r}.
	\end{equation}The second one is the Sobolev inequality
	\begin{equation}\label{eq:Sobolev}
	\| u \|_{L^q} \le C_s \| u \|_1, 
	\end{equation}
	where $q \in [2, \infty)$ for $d=2$;  $q\in [2, \frac{2d}{d-2}]$ 
	for $d>2$; $C_s$ is a general constant independent of $\phi$. 
	We can further use Poincare's inequality to get
	\begin{equation} \label{eq:Sobolev2}
	\| v \|_{L^q} \le C_s \| \nabla v \|,\quad \forall v\in L^2_0(\Omega).
	\end{equation}
	For $v \in L^2_0(\Omega)$,  we also have following inequality
	\begin{equation}\label{eq:L02}
	\| v\|^2 = (\nabla v,\nabla (-\Delta)^{-1} v) \le \frac{1}{2\delta} \|\nabla v\|^2 + \frac{\delta}{2} \| v \|_{-1}^2,
	\end{equation} where $\delta>0$ is an arbitrary constant.
	
	Now, we begin the proof.
	\begin{enumerate}
		\item[(i)]
		When $\gamma=1$, we have Cahn-Hilliard equation
		\begin{equation}\label{eq:CH0}
		\phi_{t}+ \varepsilon \Delta^2 \phi =\dfrac{1}{\varepsilon} 
		\Delta f(\phi).
		\end{equation}
		Multiplying \eqref{eq:CH0} by $-\Delta^{-1} \phi_t$ and using integration by parts, we get
		\begin{equation}\label{eq:CH1}
		\|\phi_t\|_{-1}^2 + \frac{\varepsilon}{2}\frac{d}{ d t} 
		\|\nabla \phi\|^2
		=-\frac{1}{\varepsilon}(f(\phi),\phi_t)
		=-\frac{1}{\varepsilon}\frac{d}{d t}\int_{\Omega} F(\phi) dx.
		\end{equation}
		After integrating over $[0, T]$,  we obtain
		\begin{equation}
		\int_{0}^{T}\|\phi_t\|_{-1}^{2}{\rm d}t + 
		E_{\varepsilon}(\phi(T)) = E_{\varepsilon}(\phi^0)
		\end{equation} 
		Taking maximum values of terms on the left hand side for $T \in [0, \infty]$, we get the first part of (i) from \eqref{eq:AP:E0}.
		From the definition of $E_\varepsilon(\phi)$, and assumption \eqref{eq:AP:Fcoercive} we know
		\begin{equation}
		\|\phi\|^2 \le B_0 |\Omega| + B_1 \varepsilon^{-\sigma_1+1} \lesssim \varepsilon^{-(\sigma_1-1)^+}.
		\end{equation}
		Combining above estimate with the fact $\frac\varepsilon2\| 
		\nabla \phi \| ^2 \lesssim \varepsilon^{-\sigma_1}$, we get 
		\begin{equation} \label{eq:CHregH1}
		\| \phi \|_{H^1}^2 \lesssim \varepsilon^{-(\sigma_1+1)}.
		\end{equation}
		\item[(ii)]
		We formally differentiate \eqref{eq:CH0} in time to obtain
		\begin{equation}\label{eq:CH00}
		\phi_{tt}+ \varepsilon \Delta^2 \phi_{t} = \dfrac{1}{\varepsilon}
		\Delta \left( f'(\phi)\phi_t \right). \\
		\end{equation}
		
		Pairing \eqref{eq:CH00} with $ -\Delta^{-1}\phi_t$ and using \eqref{eq:L02}, yields
		\begin{equation}\label{eq:CH13}
		\begin{split}
		\frac{1}{2}\frac{d}{d t}\|\phi_{t}\|_{-1}^2
		+ \varepsilon \| \nabla \phi_{t}\|^2
		= &-\dfrac{1}{\varepsilon}\left( f'(\phi)\phi_t, \phi_t \right)
		\leq \frac{\tilde{c}_0}{\varepsilon}\|\phi_t\|^2\\
		\leq & \frac{\varepsilon}{2} \|\nabla \phi_{t}\|^2 + \frac{\tilde{c}_{0}^2}{2 \varepsilon^3}\|\phi_{t}\|_{-1}^2.\\
		\end{split}
		\end{equation}
		Integrating \eqref{eq:CH13} over $[0,T]$ and taking maximum 
		values for terms depending on $T$, we get
		\begin{equation}\label{eq:CH14}
		\esssup \limits_{t\in[0,\infty]}  \|\phi_{t}\|_{-1}^2+ \varepsilon\int_{0}^{\infty} \|\nabla \phi_{t}\|^2{\rm d} t
		\lesssim \frac{\tilde{c}_{0}^2}{ \varepsilon^3} \int_{0}^{\infty } \|\phi_t\|_{-1}^2 {\rm d} t
		+\|\phi_{t}^0\|_{-1}^2. \\
		\end{equation}
		The assertion then follows from (i) and the inequality 
		\eqref{eq:AP:2} of Assumption \ref{newcn:ap:2}.
		
		\item[(iii)]
		Testing \eqref{eq:CH00} with $\phi_t$, using \eqref{eq:Holder3} 
		and \eqref{eq:Sobolev} with Poincare's inequality, we get
		\begin{equation}\label{eq:CH21}
		\begin{split}
		\frac{1}{2} \frac{d}{d t}\| \phi_t\|^2 + \varepsilon \|\Delta \phi_t \|^2
		= & \frac{1}{\varepsilon}(f'(\phi) \phi_t, \Delta \phi_t)
		\leq \frac{1}{\varepsilon} \|f'(\phi)\|_{L^3} \| \phi_t\|_{L^6} \| \Delta \phi_t\|\\
		\leq &
		\frac{\varepsilon}{2} \|\Delta \phi_t\|^2
		+\frac{1}{2 \varepsilon^3} \|f'(\phi)\|_{L^3}^2 \| \phi_t \|_{L^6}^2\\
		\leq &
		\frac{\varepsilon}{2} \|\Delta \phi_t\|^2
		+\frac{C_s^2}{2 \varepsilon^3} \|f'(\phi)\|_{L^3}^2 \| \nabla \phi_t \|^2,
		\end{split}
		\end{equation}
		which leads to 
		\begin{equation}\label{eq:CH22}
		\begin{split}
		&\esssup \limits_{t\in[0,\infty]}\|\phi_t\|^2 + \varepsilon \int_{0}^{\infty} \|\Delta \phi_t\|^2 {\rm d} t\\
		\lesssim&
		\frac{C_s}{\varepsilon^3}  \esssup_{t\in[0,\infty]} \|f'(\phi)\|_{L^3}^2  
		\int_{0}^{\infty} \|\nabla \phi_t\|^2 {\rm d} t
		+\|\phi_t^0\|^2.
		 \end{split}
		\end{equation}
		On the other hand side, by assumption \eqref{eq:AP:fp}, the Sobolev inequality \eqref{eq:Sobolev} and estimate \eqref{eq:CHregH1}, we have
		\begin{equation}\label{eq:fpL3}
		\|f'(\phi)\|_{L^3}^2 \lesssim \tilde{c}_2 \|\phi\|_{L^{3(p-2)}}^{2(p-2)} +\tilde{c}_3  
		\lesssim \tilde{c}_2 \| \phi \|_1^{2(p-2)} + \tilde{c}_3
		\lesssim \varepsilon^{-(\sigma_1+1)(p-2)}
		\end{equation}
		The assertion then follows from \eqref{eq:CH22}, \eqref{eq:fpL3},  (ii) and assumption \eqref{eq:AP:3}.
		
		\item[(iv)]
		Testing \eqref{eq:CH00} with $-\Delta^{-1}\phi_{tt}$, we get
		\begin{equation}\label{eq:CH23}
		\begin{split}
		&\|\phi_{tt}\|_{-1}^2+ \frac{\varepsilon}{2}\frac{d}{dt} \|\nabla \phi_{t}\|^2
		=- \dfrac{1}{\varepsilon}(f'(\phi)\phi_t, \phi_{tt})\\
		= & -\frac{1}{2\varepsilon}\frac{d}{dt}(f'(\phi)\phi_t,\phi_t) 
		+ \frac{1}{2\varepsilon} (f''(\phi)\phi_t^2,\phi_t) \\
		\leq & -\frac{1}{2\varepsilon}\frac{d}{dt}(f'(\phi)\phi_t,\phi_t) 
		+ \frac{1}{2\varepsilon} \|f''\|_{L^{6}} \| \phi_t^2\|_{L^3} \|\phi_t\| \\
		\le & -\frac{1}{2\varepsilon}\frac{d}{dt}(f'(\phi)\phi_t,\phi_t) 
		+ \frac{C_s^2}{2\varepsilon} \|f''\|_{L^{6}} \| \nabla \phi_t\|^2 \| \phi_t\| \\
		\end{split}
		\end{equation}
		Integrate  \eqref{eq:CH23} over $[0,T]$, we continue the estimate as
		\begin{equation}\label{eq:CH24}
		\begin{split}
		&2\int_{0}^{T} \|\phi_{tt}\|_{-1}^2 {\rm  d} t
		+ \varepsilon\|\nabla \phi_{t}(T)\|^2
		- \varepsilon\|\nabla \phi_{t}^0\|^2\\
		\leq & -\frac{1}{\varepsilon}(f'(\phi)\phi_t,\phi_t)|_{t=T} + \frac{1}{\varepsilon} (f'(\phi^0)\phi_t^0,\phi_t^0)\\
		&+ \frac{C_s^2}{\varepsilon} \esssup_{t\in[0,T]}\{\|f''\|_{L^{6}} \|\phi_t\|\}\int_0^T \|\nabla\phi_t\|^2 {\rm d}t\\
		\leq &  \frac{\varepsilon}{2} \|\nabla \phi_{t}(T)\|^2 + \frac{\tilde{c}_{0}^2}{2 \varepsilon^3}\|\phi_{t}(T)\|_{-1}^2
		+ \frac{1}{\varepsilon} (f'(\phi^0)\phi_t^0,\phi_t^0)\\
		&+ \frac{C_s^2}{\varepsilon} \esssup_{t\in[0,T]}\{\|f''\|_{L^{6}} \|\phi_t\|\}\int_0^T \|\nabla\phi_t\|^2{\rm d}t, 
		\end{split}
		\end{equation}
		i.e. 
		\begin{equation}\label{eq:CH24-c}
		\begin{split}
		2\int_{0}^{T} \|\phi_{tt}\|_{-1}^2 {\rm  d} t
		&+ \frac{\varepsilon}{2}\|\nabla \phi_{t}(T)\|^2
		\leq 
		\varepsilon\|\nabla \phi_{t}^0\|^2 
		+ \frac{1}{\varepsilon} (f'(\phi^0)\phi_t^0,\phi_t^0)\\
		& + \frac{\tilde{c}_{0}^2}{2 \varepsilon^3}\|\phi_{t}(T)\|_{-1}^2
		+ \frac{C_s^2}{\varepsilon} \esssup_{t\in[0,T]} \{\|f''\|_{L^{6}} \|\phi_t\| \}\int_0^T \|\nabla\phi_t\|^2{\rm d}t.
		\end{split}
		\end{equation}
		On the other hand, by \eqref{eq:AP:fpp}, the Sobolev inequality \eqref{eq:Sobolev} and estimate \eqref{eq:CHregH1}, we have
		\begin{equation}\label{eq:fppL6}
		\|f''(\phi)\|_{L^{6}}^2 \lesssim \tilde{c}_4\| \phi \|_{L^{6(p-3)^+}}^{2(p-3)^+} + \tilde{c}_5 
		\lesssim \|\phi\|_1^{2(p-3)^+} 
		\lesssim \varepsilon^{-(\sigma_1+1)(p-3)^+}
		\end{equation}
		By taking maximum for terms depending on $T$ in \eqref{eq:CH24-c} and using \eqref{eq:fppL6}, (ii), (iii) and the inequality \eqref{eq:AP:4} of Assumption \ref{newcn:ap:2}. we obtain the assertion (iv).

		\item[(v)]
		We formally differentiate \eqref{eq:CH00} in time to derive
		\begin{equation}\label{eq:CH000}
		\phi_{ttt}+ \varepsilon \Delta^2 \phi_{tt}
		=\dfrac{1}{\varepsilon}\Delta \left( f''(\phi)(\phi_t)^2+f'(\phi)\phi_{tt} \right). \\
		\end{equation}
		Testing \eqref{eq:CH000} with $\Delta^{-2} \phi_{tt}$, we obtain
		\begin{equation}\label{eq:CH31}
		\begin{split}
		& \frac{1}{2} \frac{d}{ dt } \| \Delta^{-1}\phi_{tt}\|^2 + \varepsilon \| \phi_{tt} \|^2
		=\dfrac{1}{\varepsilon} \left( f''(\phi)(\phi_t)^2+f'(\phi)\phi_{tt},  \Delta^{-1}\phi_{tt}\right)\\
		\leq{}&\dfrac{\varepsilon}{2} \|f''(\phi)\|^2_{L^2} \|\phi_t\|_{L^6}^{4} +
		\dfrac{1}{2\varepsilon^3} \|\Delta^{-1} \phi_{tt}\|^2_{L^6}
		+\frac{1}{2\varepsilon^3}\|f'(\phi)\|_{L^{3}}^2\|\Delta^{-1} \phi_{tt}\|_{L^6}^2
		+\frac{\varepsilon}{2} \|\phi_{tt}\|^2 \\
		\le{}& \dfrac{\varepsilon}{2} C_s^4\|f''(\phi)\|^2_{L^2} \|\nabla \phi_t\|^{4}
		+ \dfrac{C_s^2}{2\varepsilon^3}\| \phi_{tt}\|^2_{-1}
		+\frac{C_s^2}{2\varepsilon^3} \|f'(\phi)\|_{L^{3}}^2
		\| \phi_{tt}\|_{-1}^2
		+\frac{\varepsilon}{2} \|\phi_{tt}\|^2. \\
		\end{split}
		\end{equation}
		After taking integration from $[0,T]$ and taking maximum for terms
		depending on $T$, we have
		\begin{equation}\label{eq:CH32}
		\begin{split}
		& \esssup \limits_{t\in[0,\infty]}  \| \Delta^{-1}\phi_{tt}\|^2 + \varepsilon\int_{0}^{\infty} \| \phi_{tt} \|^2 {\rm d} t\\
		\lesssim{} & {\varepsilon}  \esssup \limits_{t\in[0,\infty]} \left(\|f''(\phi)\|^2_{L^2} \|\nabla \phi_t\|^{2}\right)
		\int_{0}^{\infty} \|\nabla \phi_t\|^{2}{\rm d} t \\
		&+ \frac{1}{\varepsilon^3}
		\left(\esssup_{t\in[0,\infty]} \|f'(\phi)\|_{L^{3}}^2+1\right)
		\int_{0}^{\infty} \| \phi_{tt}\|^2_{-1}{\rm d} t
		+ \| \Delta^{-1}\phi_{tt}^0\|^2. \\
		\end{split}
		\end{equation}
		The assertion then follows from \eqref{eq:fpL3}, the following estimate
		\begin{equation}\label{eq:fppL2}
		\|f''(\phi)\|_{L^{2}}^2 \lesssim \tilde{c}_4\| \phi \|_{L^{2(p-3)^+}}^{2(p-3)^+} + \tilde{c}_5 
		\lesssim \|\phi\|_1^{2(p-3)^+} 
		\lesssim \varepsilon^{-(\sigma_1+1)(p-3)^+},
		\end{equation}
		(ii), (iv) and the inequality \eqref{eq:AP:5} of Assumption \ref{newcn:ap:2}.
		
		\item[(vi)]
		Pairing \eqref{eq:CH000} with $-\Delta^{-1} \phi_{tt}$, we obtain
		\begin{equation}\label{eq:CH3:1}
		\begin{split}
		& \frac{\varepsilon}{2}\frac{d}{d t} \|\phi_{tt}\|_{-1}^2
		+\varepsilon \|\nabla \phi_{tt}\|^2\\
		= &-\dfrac{1}{\varepsilon}\left( f''(\phi)(\phi_t)^2+f'(\phi)\phi_{tt}, \phi_{tt} \right)\\
		\leq&  \dfrac{C_s^2}{2\varepsilon^3} \|f''(\phi)\|_{L^2}^2 \|\phi_{t}\|_{L^6}^4
		+\frac{\varepsilon}{2C_s^2}\|\phi_{tt}\|_{L^6}^2
		+\frac{\tilde{c}_0}{\varepsilon}\|\phi_{tt}\|^2\\
		\leq&  \dfrac{C_s^6}{2\varepsilon^3} \|f''(\phi)\|_{L^2}^2 \|\nabla \phi_{t}\|^4
		+\frac{\varepsilon}{2}  \|\nabla \phi_{tt}\|^2 
		+\frac{\tilde{c}_0}{\varepsilon}\|\phi_{tt}\|^2.\\
		\end{split}
		\end{equation}
		Integrating \eqref{eq:CH3:1} from  $[0,\infty)$, we have
		\begin{equation}\label{eq:CH3:2}
		\begin{split}
		& \esssup \limits_{t\in[0,\infty]}  \| \phi_{tt}\|_{-1}^2 
		+ \varepsilon\int_{0}^{\infty} \| \nabla \phi_{tt} \|^2 {\rm d} t\\
		\lesssim{} & \frac{C_s^6}{\varepsilon^3}  \esssup \limits_{t\in[0,\infty]} \left(\|f''(\phi)\|^2_{L^2} \|\nabla \phi_t\|^{2}\right)
		\int_{0}^{\infty} \|\nabla \phi_t\|^{2}{\rm d} t \\
		&+ \frac{2\tilde{c}_0}{\varepsilon}
		\int_{0}^{\infty} \| \phi_{tt}\|^2{\rm d} t
		+ \|\phi_{tt}^0\|_{-1}^2. \\
		\end{split}
		\end{equation}
		The assertion then follows from \eqref{eq:fppL2}, (ii), (iv), (vi) and the inequality \eqref{eq:AP:6} of Assumption \ref{newcn:ap:2}.	
		
		\item[(vii)]
		Pairing \eqref{eq:CH000} with $ \phi_{tt}$, we obtain
		\begin{equation}\label{eq:CH3:4}
		\begin{split}
		&\frac{1}{2}\frac{d}{d t}\|\phi_{tt}\|^2
		+ \varepsilon \|\Delta \phi_{tt}\|^2\\
		= &\dfrac{1}{\varepsilon}\left( f''(\phi)(\phi_t)^2+f'(\phi)\phi_{tt}, \Delta \phi_{tt} \right)\\
		\leq& \frac{1}{\varepsilon}\|f''(\phi)\|_{L^6}\|\phi_t^2\|_{L^3}\|\Delta \phi_{tt}\|
		+\frac{1}{\varepsilon}\|f'(\phi)\|_{L^3}\|\phi_{tt}\|_{L^6}\|\Delta \phi_{tt}\| \\
		\leq&  \dfrac{1}{\varepsilon^3} \left( \|f''(\phi)\|_{L^6}^2 \|\phi_{t}\|_{L^6}^4
		+ \|f'(\phi)\|_{L^3}^2 \|\phi_{tt}\|_{L^6}^2 \right)
		+ \frac{\varepsilon}{2}\|\Delta \phi_{tt}\|^2\\
		\leq&  \dfrac{1}{\varepsilon^3} \left(C_s^4 \|f''(\phi)\|_{L^6}^2 \|\nabla \phi_{t}\|^4
		+ C_s^2\|f'(\phi)\|_{L^3}^2 \|\nabla \phi_{tt}\|^2 \right)
		+ \frac{\varepsilon}{2}\|\Delta \phi_{tt}\|^2.\\
		\end{split}
		\end{equation}
		After taking integration from $[0,T]$ and taking maximum for terms
		depending on $T$, we have
		\begin{equation}\label{eq:CH32:0}
		\begin{split}
		& \esssup \limits_{t\in[0,\infty]}  \| \phi_{tt}\|^2 
		+ \varepsilon\int_{0}^{\infty} \| \Delta \phi_{tt} \|^2 {\rm d} t\\
		\lesssim{} & \frac{2C_s^4}{\varepsilon^3}  \esssup \limits_{t\in[0,\infty]} \left(\|f''(\phi)\|^2_{L^6} \|\nabla \phi_t\|^{2}\right)
		\int_{0}^{\infty} \|\nabla \phi_t\|^{2}{\rm d} t \\
		&+ \frac{2C_s^2}{\varepsilon^3}
		\esssup_{t\in[0,\infty]} \|f'(\phi)\|_{L^{3}}^2
		\int_{0}^{\infty} \| \nabla \phi_{tt}\|^2{\rm d} t
		+ \| \phi_{tt}^0\|^2. \\
		\end{split}
		\end{equation}
		The assertion then follows from \eqref{eq:fppL6}, \eqref{eq:fpL3}, (ii), (iv), (v) and the inequality \eqref{eq:AP:7} of Assumption \ref{newcn:ap:2}.

		\item[(viii)]
		Pairing \eqref{eq:CH00} with $ -\Delta^{-3}\phi_{tt}$, we obtain
		\begin{equation}\label{eq:CH2:1}
		\begin{split}
		&\|\Delta^{-1}\phi_{tt}\|_{-1}^2
		+ \frac{\varepsilon}{2}\frac{d}{d t}  \|\phi_t\|_{-1}^2\\
		=&-\dfrac{1}{\varepsilon} ( f'(\phi) \phi_t, \Delta^{-2}\phi_{tt})
		\leq  \frac{1}{\varepsilon}\|f'(\phi)\|_{L^3}\|\phi_t\|\|\Delta^{-2}\phi_{tt}\|_{L^6}\\
		\leq & \frac{1}{2C_s^2}\|\Delta^{-2}\phi_{tt}\|_{L^6}^2
		+\frac{C_s^2}{2 \varepsilon^2}\|f'(\phi)\|_{L^3}^2\|\phi_t\|^2\\
		\leq & \frac{1}{2}\|\Delta^{-1}\phi_{tt}\|_{-1}^2 + 
		\frac{C_s^2}{2 \varepsilon^2}\|f'(\phi)\|_{L^3}^2(\|\phi_t\|_{-1}^2  +\|\nabla\phi_t\|^2).
		\end{split}
		\end{equation}
		After taking integration from $[0,T]$ and taking maximum for terms
		depending on $T$, we have
		\begin{equation}\label{eq:CH2:2}
		\begin{split}
		&\int_{0}^{\infty}  \|\Delta^{-1}\phi_{tt}\|_{-1}^2{\rm d} t
		+ \varepsilon \esssup \limits_{t\in[0,\infty]}   \|\phi_t\|_{-1}^2\\
		\lesssim{} & 
		\frac{C_s^2}{ \varepsilon^2}  \esssup \limits_{t\in[0,\infty]}\|f'(\phi)\|_{L^3}^2 \int_{0}^{\infty} (\|\phi_t\|_{-1}^2  +\|\nabla\phi_t\|^2){\rm d} t
		+ \varepsilon \|\phi_t^0\|_{-1}^2.
		\end{split}
		\end{equation}
		The assertion then follows from \eqref{eq:fpL3}, (i), (ii) and the inequality \eqref{eq:AP:2} of Assumption \ref{newcn:ap:2}.

		\item[(ix)]
		Pairing \eqref{eq:CH00} with $ \Delta^{-2}\phi_{t}$, we obtain
		\begin{equation}\label{eq:CH2:3}
		\begin{split}
		&\frac{1}{2}\frac{d}{d t} \|\Delta^{-1}\phi_{t}\|^2
		+  \varepsilon \|\phi_t\|^2\\
		=&\dfrac{1}{\varepsilon} ( f'(\phi) \phi_t, \Delta^{-1}\phi_{t})
		\leq  \frac{1}{\varepsilon}\|f'(\phi)\|_{L^3}\|\phi_t\|\|\Delta^{-1}\phi_{t}\|_{L^6}\\
		\leq & \frac{\varepsilon}{2}\|\phi_{t}\|^2
		+\frac{C_s^2}{2 \varepsilon^3}\|f'(\phi)\|_{L^3}^2\|\phi_t\|_{-1}^2.
		\end{split}
		\end{equation}
		After taking integration from $[0,T]$ and taking maximum for terms
		depending on $T$, we have
		\begin{equation}\label{eq:CH2:4}
		\begin{split}
		&\esssup \limits_{t\in[0,\infty]} \|\Delta^{-1}\phi_{t}\|^2
		+  \varepsilon \int_{0}^{\infty} \|\phi_t\|^2 {\rm d} t\\
		\leq &
		\frac{C_s^2}{ \varepsilon^3} \esssup \limits_{t\in[0,\infty]}\|f'(\phi)\|_{L^3}^2
		\int_{0}^{\infty} \|\phi_t\|_{-1}^2{\rm d} t.
		\end{split}
		\end{equation}
		
		\item[(x)]
		We can easily get the proof from (ii) (iii) (ix).
		
		\item[(xi)]
		We can easily get the proof from (v) (vi) (vii).
		
		\item[(xii)]
		Multiplying \eqref{eq:CH0} by $ \phi$ and using integration by parts and $\frac\varepsilon2\| \nabla \phi \| ^2 \lesssim \varepsilon^{-\sigma_1}$, we get
		\begin{equation}\label{eq:CH01}
		\begin{split}
		&\frac{1}{2}\frac{d}{ d t}\|\phi\|^2 + \varepsilon\|\Delta \phi\|^2
		=\frac{1}{\varepsilon}(\Delta f(\phi),\phi)\\
		=&-\frac{1}{\varepsilon}(f'(\phi) \nabla\phi,\nabla \phi)
		\leq \frac{\tilde{c}_0}{\varepsilon}\|\nabla \phi\|^2
		\lesssim \varepsilon^{-(\sigma_1+2)}.
		\end{split}
		\end{equation}
		Then we easily get
		\begin{equation}\label{eq:CH02}
		\begin{split}
		\int_{0}^{T} \|\phi\|^2_{H^2}{\rm d} t
		\lesssim \int_{0}^{T}\|\phi\|^2 + \|\nabla \phi\|^2 +\|\Delta \phi\|^2{\rm d} t
		\lesssim \varepsilon^{-(\sigma_1+3)}.
		\end{split}
		\end{equation}			
		
		\item[(xiii)]
		Multiplying \eqref{eq:CH0} by $ \Delta^{-1} \phi_t$ and using integration by parts, we get
		\begin{equation}\label{eq:CH05}
		\begin{split}
		&\|\Delta^{-1}\phi_t\|^2 + \frac{\varepsilon}{2}\frac{d}{ d t} \|\phi\|^2
		=\frac{1}{\varepsilon}(f(\phi),\Delta^{-1}\phi_t)\\
		=&-\frac{1}{\varepsilon}(f'(\phi) \nabla\phi,\Delta^{-\frac{3}{2}} \phi_t)
		\leq  \frac{1}{\varepsilon}\|f'(\phi)\|_{L^3}\|\nabla \phi\|\|\Delta^{-\frac{3}{2}}\phi_t\|_{L^6}\\
		\leq & \frac{1}{2C_s^2}\|\Delta^{-\frac{3}{2}}\phi_t\|_{L^6}^2
		+\frac{C_s^2}{2 \varepsilon^2}\|f'(\phi)\|_{L^3}^2\|\nabla\phi\|^2\\
		\leq & \frac{1}{2}\|\Delta^{-1}\phi_t\|^2
		+\frac{C_s^2}{2 \varepsilon^2}\|f'(\phi)\|_{L^3}^2\|\nabla\phi\|^2.
		\end{split}
		\end{equation}
		After taking integration from $[0,T]$, we have
		\begin{equation}\label{eq:CH06}
		\begin{split}
		&\int_{0}^{T} \|\Delta^{-1}\phi_t\|^2 {\rm d} t
		+ \varepsilon \esssup\limits_{t\in[0,T]}\|\phi\|^2\\
		\leq &
		\frac{C_s^2}{ \varepsilon^2} \esssup \limits_{t\in[0,T]}\|f'(\phi)\|_{L^3}^2
		\int_{0}^{T} \|\nabla \phi\|^2{\rm d} t
		\lesssim  \varepsilon^{-(\sigma_{1}+1)(p-1)-2}.
		\end{split}
		\end{equation}
		On the other hand
		\begin{equation}\label{eq:CH07}
		\phi_{t}=\Delta \mu, 
		\end{equation}
		combining above estimate with (i) (ix), then we have
		\begin{equation}\label{eq:CH08}
		\int_{0}^{T}\|\mu\|^2_{H^2}{\rm d} t \lesssim
		\int_{0}^{T}\|\Delta^{-1} \phi_t\|^2+ \|\phi_t\|_{-1}^2 + \|\phi_t\|^2{\rm d} t
		\lesssim  \varepsilon^{-\beta_{12}}.
		\end{equation}
	\end{enumerate} 
	
\end{proof}





\section*{References}
\bibliographystyle{plainnat}  

\begin{thebibliography}{10}
	
	\bibitem{alikakos_convergence_1994}
	Nicholas~D. Alikakos, Peter~W. Bates, and Xinfu Chen.
	\newblock Convergence of the {Cahn}-{Hilliard} equation to the {Hele}-{Shaw}
	model.
	\newblock {\em Arch. Ration. Mech. Anal.}, 128(2):165--205, 1994.
	
	\bibitem{anderson_diffuse-interface_2003}
	D.~M Anderson, G.~B Mcfadden, and A.~A Wheeler.
	\newblock Diffuse-{Interface} {Methods} in {Fluid} {Mechanics}.
	\newblock {\em Annu. rev. fluid Mech}, 30(1):139--165, 2003.
	
	\bibitem{baskaran_energy_2013}
	A.~Baskaran, P.~Zhou, Z.~Hu, C.~Wang, S.~Wise, and J.~Lowengrub.
	\newblock Energy stable and efficient finite-difference nonlinear multigrid
	schemes for the modified phase field crystal equation.
	\newblock {\em J. Comput. Phys.}, 250:270--292, 2013.
	
	\bibitem{Brenner_Mathematical_2010}
	S.~Brenner and L.~Scott.
	\newblock {\em The {Mathematical} {Theory} of {Finite} {Element} {Methods}.}
	\newblock Springer-Verlag, 2010.
	
	\bibitem{caffarelli_l_1995}
	Luis~A. Caffarelli and Nora~E. Muler.
	\newblock An ${L^\infty}$ bound for solutions of the {Cahn}-{Hilliard}
	equation.
	\newblock {\em Arch. Ration. Mech. Anal.}, 133(2):129--144, 1995.
	
	\bibitem{cahn_free_1958}
	John~W. Cahn and John~E. Hilliard.
	\newblock Free energy of a nonuniform system. {I}. interfacial free energy.
	\newblock {\em J. Chem. Phys.}, 28(2):258--267, 1958.
	
	\bibitem{chen_spectrum_1994}
	Xinfu Chen.
	\newblock Spectrum for the {Allen-Cahn}, {Cahn-Hillard}, and phase-field
	equations for generic interfaces.
	\newblock {\em Commun. Part. Diff. Eq.}, 19(7):1371--1395, 1994.
	
	\bibitem{cheng_second-order_2016}
	Kelong Cheng, Cheng Wang, Steven~M. Wise, and Xingye Yue.
	\newblock A {Second}-{Order}, {Weakly} {Energy}-{Stable} {Pseudo}-spectral
	{Scheme} for the {Cahn}-{Hilliard} {Equation} and {Its} {Solution} by the
	{Homogeneous} {Linear} {Iteration} {Method}.
	\newblock {\em J. Sci. Comput.}, 69(3):1083--1114, 2016.
	
	\bibitem{condette_spectral_2011}
	Nicolas Condette, Christof Melcher, and Endre {S\"uli}.
	\newblock Spectral approximation of pattern-forming nonlinear evolution
	equations with double-well potentials of quadratic growth.
	\newblock {\em Math. Comp.}, 80(273):205--223, 2011.
	
	\bibitem{diegel_stability_2016}
	Amanda~E. Diegel, Cheng Wang, and Steven~M. Wise.
	\newblock Stability and convergence of a second order mixed finite element
	method for the {Cahn}-{Hilliard} equation.
	\newblock {\em IMA J Numer. Anal.}, 36(4):1867--1897, 2016.
	
	\bibitem{du_numerical_1991}
	Qiang Du and Roy~A. Nicolaides.
	\newblock Numerical analysis of a continuum model of phase transition.
	\newblock {\em SIAM J. Numer. Anal.}, 28(5):1310--1322, 1991.
	
	\bibitem{elliott_global_1993}
	C.~M. Elliott and A.~M. Stuart.
	\newblock The global dynamics of discrete semilinear parabolic equations.
	\newblock {\em SIAM J. Numer. Anal.}, 30:1622--1663, 1993.
	
	\bibitem{elliott_cahnhilliard_1996}
	Charles~M. Elliott and Harald Garcke.
	\newblock On the {Cahn}-{Hilliard} {Equation} with {Degenerate} {Mobility}.
	\newblock {\em SIAM J. Math. Anal.}, 27(2):404--423, 1996.
	
	\bibitem{elliott_error_1992}
	Charles~M. Elliott and Stig Larsson.
	\newblock Error estimates with smooth and nonsmooth data for a finite element
	method for the {Cahn}-{Hilliard} equation.
	\newblock {\em Math. Comp.}, 58(198):603--630, S33--S36, 1992.
	
	\bibitem{eyre_unconditionally_1998}
	D.~J. Eyre.
	\newblock Unconditionally gradient stable time marching the {Cahn}-{Hilliard}
	equation.
	\newblock In {\em Computational and Mathematical Models of Microstructural
		Evolution ({San} {Francisco}, {CA}, 1998)}, volume 529 of {\em Mater. {Res}.
		{Soc}. {Sympos}. {Proc}.}, pages 39--46. MRS, 1998.
	
	\bibitem{feng_analysis_2015}
	Xiaobing Feng and Yukun Li.
	\newblock Analysis of symmetric interior penalty discontinuous {G}alerkin
	methods for the {Allen-Cahn} equation and the mean curvature flow.
	\newblock {\em {IMA} J. Numer. Anal.}, 35(4):1622--1651, 2015.
	
	\bibitem{feng_analysis_2016}
	Xiaobing Feng, Yukun Li, and Yulong Xing.
	\newblock Analysis of mixed interior penalty discontinuous {G}alerkin methods
	for the {Cahn-Hilliard} equation and the {Hele-Shaw} flow.
	\newblock {\em {SIAM} J. Numer. Anal.}, 54(2):825--847, 2016.
	
	\bibitem{feng_numerical_2003}
	Xiaobing Feng and Andreas Prohl.
	\newblock Numerical analysis of the {Allen}-{Cahn} equation and approximation
	for mean curvature flows.
	\newblock {\em Numer. Math.}, 94(1):33--65, 2003.
	
	\bibitem{feng_error_2004}
	Xiaobing Feng and Andreas Prohl.
	\newblock Error analysis of a mixed finite element method for the
	{Cahn}-{Hilliard} equation.
	\newblock {\em Numer. Math.}, 99(1):47--84, 2004.
	
	\bibitem{feng_numerical_2005}
	Xiaobing Feng and Andreas Prohl.
	\newblock Numerical analysis of the {Cahn}-{Hilliard} equation and
	approximation for the {Hele}-{Shaw} problem.
	\newblock {\em Interfaces Free Bound.}, 7(1):1--28, 2005.
	
	\bibitem{feng_stabilized_2013}
	Xinlong Feng, Tao Tang, and Jiang Yang.
	\newblock Stabilized {Crank}-{Nicolson}/{Adams}-{Bashforth} schemes for phase
	field models.
	\newblock {\em E. Asian J. Appl. Math.}, 3(1):59--80, 2013.
	
	\bibitem{furihata_stable_2001}
	Daisuke Furihata.
	\newblock A stable and conservative finite difference scheme for the
	{C}ahn-{H}lliard equation.
	\newblock {\em Numer. Math.}, 87(4):675--699, 2001.
	
	\bibitem{gomez_provably_2011}
	Hector Gomez and Thomas J.~R. Hughes.
	\newblock Provably unconditionally stable, second-order time-accurate, mixed
	variational methods for phase-field models.
	\newblock {\em J. Comput. Phys.}, 230(13):5310--5327, 2011.
	
	\bibitem{guillen-gonzalez_linear_2013}
	F.~Guillén-González and G.~Tierra.
	\newblock On linear schemes for a {Cahn}-{Hilliard} diffuse interface model.
	\newblock {\em J. Comput. Phys.}, 234:140--171, 2013.
	
	\bibitem{guillen-gonzalez_second_2014}
	Francisco Guillén-González and Giordano Tierra.
	\newblock Second order schemes and time-step adaptivity for {Allen}-{Cahn} and
	{Cahn}-{Hilliard} models.
	\newblock {\em Comput. Math. Appl.}, 68(8):821--846, 2014.
	
	\bibitem{guo_h2_2016}
	Jing Guo, Cheng Wang, Steven~M. Wise, and Xingye Yue.
	\newblock An {$H^2$} convergence of a second-order convex-splitting, finite
	difference scheme for the three-dimensional {Cahn}-{Hilliard} equation.
	\newblock {\em Commun. Math. Sci}, 14(2):489--515, 2016.
	
	\bibitem{han_numerical_2017}
	D.~Han, A.~Brylev, X.~Yang, and Z.~Tan.
	\newblock Numerical analysis of second order, fully discrete energy stable
	schemes for phase field models of two phase incompressible flows.
	\newblock {\em J. Sci. Comput.}, 70:965--989, 2017.
	
	\bibitem{he_large_2007}
	Yinnian He, Yunxian Liu, and Tao Tang.
	\newblock On large time-stepping methods for the {Cahn}-{Hilliard} equation.
	\newblock {\em Appl. Numer. Math.}, 57(5-7):616--628, 2007.
	
	\bibitem{kessler_posteriori_2004}
	Daniel Kessler, Ricardo~H. Nochetto, and Alfred Schmidt.
	\newblock A posteriori error control for the {Allen}-{Cahn} problem:
	circumventing {Gronwall}'s inequality.
	\newblock {\em ESAIM: Math. Model. Numer. Anal.}, 38(01):129--142, 2004.
	
	\bibitem{li_second_2017}
	Dong Li and Zhonghua Qiao.
	\newblock On second order semi-implicit {Fourier} spectral methods for 2d
	{Cahn}-{Hilliard} equations.
	\newblock {\em J. Sci. Comput.}, 70(1):301--341, 2017.
	
	\bibitem{li_characterizing_2016}
	Dong Li, Zhonghua Qiao, and Tao Tang.
	\newblock Characterizing the stabilization size for semi-implicit
	{Fourier}-spectral method to phase field equations.
	\newblock {\em SIAM J Numer. Anal.}, 54(3):1653--1681, 2016.
	
	\bibitem{li_second_2018}
	Weijia Li, Wenbin Chen, Cheng Wang, Yue Yan, and Ruijian He.
	\newblock A second order energy stable linear scheme for a thin film model
	without slope selection.
	\newblock {\em J. Sci. Comput.}, 76(3):1905--1937, 2018.
	
	\bibitem{li_second-order_2017}
	Xiao Li, Zhonghua Qiao, and Hui Zhang.
	\newblock A second-order convex splitting scheme for a {C}ahn-{H}illiard
	equation with variable interfacial parameters.
	\newblock {\em J. Comput. Math.}, 35(6):693--710, 2017.
	
	\bibitem{liu_phase_2003}
	Chun Liu and Jie Shen.
	\newblock A phase field model for the mixture of two incompressible fluids and
	its approximation by a {Fourier}-spectral method.
	\newblock {\em Physica D}, 179(3-4):211--228, 2003.
	
	\bibitem{shen_efficient_1995}
	Jie Shen.
	\newblock Efficient spectral-galerkin method ii. direct solvers of second- and
	fourth-order equations using chebyshev polynomials.
	\newblock {\em SIAM J. Sci. Comput.}, 16:74--87, 1995.
	
	\bibitem{shen_scalar_2017}
	Jie Shen, Jie Xu, and Jiang Yang.
	\newblock The scalar auxiliary variable ({SAV}) approach for gradient flows.
	\newblock {\em J. Comput. Phys.}, 353:407--416, 2018.
	
	\bibitem{shen_new_2017}
	Jie Shen, Jie Xu, and Jiang Yang.
	\newblock A new class of efficient and robust energy stable schemes for
	gradient flows.
	\newblock {\em SIAM Rev.}, 61(3):474--506, 2019.
	
	\bibitem{shen_numerical_2010}
	Jie Shen and Xiaofeng Yang.
	\newblock Numerical approximations of {Allen-Cahn} and {Cahn-Hilliard}
	equations.
	\newblock {\em Discrete Cont. Dyn. A}, 28:1669--1691, 2010.
	
	\bibitem{shen_efficient_2015}
	Jie Shen, Xiaofeng Yang, and Haijun Yu.
	\newblock Efficient energy stable numerical schemes for a phase field moving
	contact line model.
	\newblock {\em J. Comput. Phys.}, 284:617--630, 2015.
	
	\bibitem{WangYu2018b}
	Lin Wang and Haijun Yu.
	\newblock Convergence analysis of an unconditionally energy stable linear
	{Crank-Nicolson} scheme for the {Cahn-Hilliard} equation.
	\newblock {\em J. Math. Study}, 51(1):89--114, 2018.
	
	\bibitem{wang_two_2018}
	Lin Wang and Haijun Yu.
	\newblock On efficient second order stabilized semi-implicit schemes for the
	{Cahn}-{Hilliard} phase-field equation.
	\newblock {\em J. Sci. Comput.}, 77(2):1185--1209, 2018.
	
	\bibitem{wang_energy_2018}
	Lin Wang and Haijun Yu.
	\newblock Energy stable second order linear schemes for the {Allen}-{Cahn}
	phase-field equation.
	\newblock {\em Commun. Math. Sci.}, 17(3):609--635, 2019.
	
	\bibitem{wu_stabilized_2014}
	X.~Wu, G.~J. van Zwieten, and K.~G. van~der Zee.
	\newblock Stabilized second-order convex splitting schemes for
	{Cahn}-{Hilliard} models with application to diffuse-interface tumor-growth
	models.
	\newblock {\em Int. J. Numer. Meth. Biomed. Engng.}, 30(2):180--203, 2014.
	
	\bibitem{yan_second-order_2018}
	Yue Yan, Wenbin Chen, Cheng Wang, and Steven Wise.
	\newblock A second-order energy stable {BDF} numerical scheme for the
	{Cahn-Hilliard} equation.
	\newblock {\em Commun. Comput. Phys.}, 23(2):572--602, 2018.
	
	\bibitem{yang_error_2009}
	Xiaofeng Yang.
	\newblock Error analysis of stabilized semi-implicit method of {Allen}-{Cahn}
	equation.
	\newblock {\em Discrete. Cont. Dyn. B.}, 11(4):1057--1070, 2009.
	
	\bibitem{yang_linear_2016}
	Xiaofeng Yang.
	\newblock Linear, first and second-order, unconditionally energy stable
	numerical schemes for the phase field model of homopolymer blends.
	\newblock {\em J. Comput. Phys.}, 327:294--316, 2016.
	
	\bibitem{yang_efficient_2017}
	Xiaofeng Yang and Lili Ju.
	\newblock Efficient linear schemes with unconditional energy stability for the
	phase field elastic bending energy model.
	\newblock {\em Comput. Method. Appl. Mech. Eng.}, 315:691--712, 2017.
	
	\bibitem{yang_yu_efficient_2017}
	Xiaofeng Yang and Haijun Yu.
	\newblock Efficient second order unconditionally stable schemes for a phase
	field moving contact line model using an invariant energy quadratization
	approach.
	\newblock {\em {SIAM} J. Sci. Comput.}, 40(3):B889--B914, 2018.
	
	\bibitem{yue_diffuse-interface_2004}
	Pengtao Yue, James~J. Feng, Chun Liu, and Jie Shen.
	\newblock A diffuse-interface method for simulating two-phase flows of complex
	fluids.
	\newblock {\em J. Fluid. Mech.}, 515:293--317, 2004.
	
\end{thebibliography}

\end{document}